\documentclass[10pt,reqno,final]{amsart}
\usepackage{epsfig,amssymb,amsmath,version}
\usepackage{amssymb,version,graphicx,fancybox,mathrsfs,multirow}
\usepackage[justification=centering]{caption}
\usepackage{epstopdf}
\usepackage{url,hyperref}
\usepackage[notcite,notref]{showkeys}
\usepackage{bm}
\usepackage{graphicx}
\usepackage{subfigure,wrapfig}
\usepackage{color,xcolor}
\usepackage{cases}
\usepackage{mathtools}
\usepackage{extarrows}
\usepackage{bbm}
\usepackage{empheq}
\usepackage{booktabs} % \cmidrule

\catcode`\@=11 \theoremstyle{plain}
\@addtoreset{equation}{section}   % Makes \section reset 'equation' counter.
\renewcommand{\theequation}{\arabic{section}.\arabic{equation}}
\@addtoreset{figure}{section}
\renewcommand\thefigure{\thesection.\@arabic\c@figure}
\renewcommand{\thefigure}{\arabic{section}.\arabic{figure}}
\newtheorem{thm}{\bf Theorem}

\newtheorem{cor}{\bf Corollary}

\newenvironment{corollary}{\begin{cor}} {\end{cor}}
\newtheorem{lmm}{\bf Lemma}

\theoremstyle{example}

\theoremstyle{remark}
\newtheorem{rem}{\bf Remark}[section]
\theoremstyle{definition}
\newtheorem{defn}{\bf Definition}[section]

\numberwithin{table}{section}

\newcommand{\pt}[1][t]{\partial_{#1}}

\newcommand{\px}[1][x]{\partial_{#1}}
\newcommand{\py}[1][y]{\partial_{#1}}

\def \af {\alpha}
\def \bt {\beta}

\DeclareMathOperator{\esssup}{ess\;sup}

\newcommand{\bs}[1]{\boldsymbol{#1}}

\renewcommand \wedge \times

\begin{document}
	\bibliographystyle{plain}
	\graphicspath{{./figures/}}

\title[Maxwell's equations in Havriliak-Negami media] {Analysis of A Backward Euler-type Scheme  for Maxwell's Equations in A Havriliak-Negami Dispersive Medium}
\author[Y.B. Yang,\;\; L.L. Wang   \;$\&$\; F.H. Zeng]{Yubo Yang$^\dag$,\;\; Li-Lian Wang$^\ddag$ \;and\; Fanhai Zeng$^\P$ }
	
\thanks{${}^\dag$Nanhu College, Jiaxing University, Jiaxing, Zhejiang,  314001, China. Email: boydman\_xm@zjxu.edu.cn. \\
\indent ${}^\ddag${\em Corresponding author}. Division of Mathematical Sciences, School of Physical and Mathematical Sciences, Nanyang Technological University, 637371, Singapore.  The research of this author is partially supported by Singapore MOE AcRF Tier 2 Grants: MOE2017-T2-144 \& MOE2018-T2-1-059. Email: lilian@ntu.edu.sg.\\
\indent ${}^\P$School of Mathematics, Shandong University, Jinan, Shandong, 250100, China. Email: fanhai\_zeng@sdu.edu.cn.\\
\indent The first  author would like to thank NTU  for hosting his visit devoted to this collaborative work.}
	
\begin{abstract}
For the Maxwell's equations in a Havriliak-Negami (H-N) dispersive medium, the associated energy dissipation law has not been settled at both continuous level and discrete level. In this paper, we rigorously show that the energy of the H-N model
can be bounded by the initial energy and the model is well-posed. We analyse a backward Euler-type semi-discrete scheme, and prove that the modified discrete energy decays monotonically in time.  Such a strong stability ensures that the scheme is unconditionally stable. We  also introduce a fast temporal convolution algorithm to alleviate the burden of the history dependence
in the polarisation relation involving the singular kernel with the  Mittag-Leffler function with three parameters.
We provide ample numerical results to demonstrate the efficiency and accuracy of a full-discrete scheme via a
spectra-Galerkin method in two  dimensions.  Finally, we consider an interesting application in the recovery of complex relative
permittivity and some related physical quantities.
%can inherit the continuous energy stability is developed. Moreover, a kind of energy decay which is monotonic is also derived. Consequently, the corresponding stability and convergence results are obtained. Compared to the classical Maxwell’s equations, the most interesting but challenging issue in a H-N medium is the treatment of the singular integral in time. Without loss of generality, we provide a fully discrete scheme for the two-dimensional case, where Legendre-Galerkin method is used for space discretisation. Stability and convergence analysis of the fully discrete scheme are also presented.
%Furthermore, a fast convolution algorithm of the time-discretisation is designed, which is very efficient and suitable for long time simulation. Finally, numerical examples are given to confirm the theoretical results.
\end{abstract}
	

	
\maketitle

\vspace*{-6pt}

\section{Introduction} \label{sect1}

\setcounter{equation}{0}
\setcounter{lmm}{0}
\setcounter{thm}{0}

In electromagnetism, the most general model for a dispersive dielectric material, i.e. a material with frequency-dependent permittivity, is the Havriliak-Negami (H-N) dielectric model (see, e.g., \cite{havriliak1966complex,havriliak1967complex,kelley1999piecewise,raju2016dielectrics}).  In this model, the complex relative  permittivity is expressed as
\begin{equation}
\label{hnperm}
\epsilon_r(\omega)=\epsilon_{\infty}+\frac{\epsilon_s-\epsilon_{\infty}}{\left(1+(i\omega\tau_0)^{\alpha}\right)^{\beta}},
\end{equation}
where $0<\alpha, \beta \leq 1$, $\epsilon_{\infty}, \epsilon_s$ and $\tau_0$ are the infinite-frequency permittivity, the static permittivity and the relaxation time respectively, and $\epsilon_s, \epsilon_{\infty}$ satisfy $\epsilon_s > \epsilon_{\infty} \geq 1$. Furthermore, $i=\sqrt{-1}$ denotes the imaginary unit, and $\omega$ is the angular frequency.
All the anomalously dispersive dielectric models are its subclasses.
 When $\alpha=\beta=1$, the H-N model reduces to the Debye model \cite{debye1929polar}, while the H-N model reduces to Davidson-Cole (D-C) model \cite{davidson1951dielectric} when $\alpha=1,$ and to Cole-Cole (C-C) model when $\beta=1$ \cite{cole1941dispersion}. Such models arise from diverse fields, which typically
 %materials, whose permittivities can be represented by \eqref{hnperm},
 include biological tissues \cite{polk1995handbook, biswas2017fractional,lopes2018fractional}, soils \cite{repo1996application}, amorphous polymers near the glass-liquid transition \cite{garcia2004amorphous}, glassy materials \cite{alegria1997alpha} among others.

In general, there are two main  strategies to simulate the electromagnetic wave propagations in dispersive media  based on
different treatments of the relation between the electric flux and  electric field intensity, governed by the polarisation equation.
The first is to introduce certain auxiliary function and related auxiliary differential equation (ADE) to deal with the  polarisation.
The second  is to formulate the polarisation as a time convolution integral equation of the electric field.   For the Debye or Debye-type model, such as Drude or Lorentz model, its time-domain expression of its relative complex permittivity can be easily formulated, because its relative complex permittivity is a function of integer powers of $i\omega$. Therefore, both approaches  can be applied.
In particular,  the ADE involves derivatives of integer order that can  be  discretised by the finite difference methods \cite{taflove2005computational,Elsherbeni2015} as usual. However, the relative complex permittivity of the H-N, D-C or C-C model involves a function of non-integer powers of $i\omega$, so its  representation in  the time domain is much more complicated. In fact, the polarisation relations  are oftentimes integro-differential equations with global fractional operators  \cite{Podlubny1999,li2011developing,Rekanos2012b,huang2018accurate,causley2011incorporating}, which pose  significant difficulties and are much more expensive to solve.

In regards to the C-C model, fractional ADE-based time-domain methods were proposed  in e.g. \cite{li2011developing,Rekanos2010Auxiliary,schuster1998fdtd,torres1996application},  where the polarisation equation involves the fractional-in-time Riemann-Liouville derivative (cf. \cite{Podlubny1999}). As such, much recent development in numerical fractional differential equations casts light on time discretisation of this model.
%A detailed review can be referred to the recent works \cite{Rekanos2010Auxiliary,li2011developing,huang2018accurate,Chakarothai2019}.
However, there has been very limited works on numerical solutions of the  D-C and H-N models, %there are only a few works dealing with these two models,
largely due to that the polarisation relation cannot be expressed in terms of ADE with usual fractional differential operators.
% because the fractional operators corresponding to their permittivities can not be expressed by the classical fractional operators, such as fractional Riemann-Liouville operator and fractional Caputo operator (cf. \cite{Podlubny1999}).
% In order to realize in the time-domain, some preprocesses of the permittivity are used in the frequency domain by these works.
%  General speaking, there are two ways to do the preprocesses:
%
Nevertheless, some interesting attempts include the approximation of  the D-C or H-N model  by the  Debye model in the frequency domain \cite{Rekanos2012b,rekanos2012fdtd,kelley2007debye,Chakarothai2019}; or  by the C-C model in the frequency domain \cite{Schonhals1991Fast,Mescia2014Fractional,bia2015novel,antonopoulos2017fdtd}.
%On the other hand, except for \cite{causley2011incorporating,li2011developing,huang2018accurate},
We remark that most works related to H-N, C-C or D-C model above are implemented by the finite difference time-domain (FDTD) method (cf. \cite{taflove2005computational}), and the  stability and convergence analysis is yet unavailable.
On the other hand, Li et al. \cite{li2011developing},  and Huang and Wang \cite{huang2018accurate} developed a finite element time-domain (FETD) method (based on a fractional differential form
 of the polarisation equation) and a spectral time-domain method (based on an integro-differential formulation) for the C-C model, respectively. Stability and convergence analysis were also provided in these two works.

In this paper, we propose and analyse  a time-domain numerical method for solving the H-N model with the polarisation relation formulated by
an integral equation involving  a singular kernel function in terms of  the Mittag-Leffler (ML) function with three parameters.
We highlight our main contributions as follows.

\begin{itemize}
\item With the aid of some useful properties of the ML function, we prove that the  energy of the H-N model can be controlled by the initial energy, which
ensures the well-posedness of the model and plays an important role in developing stable numerical methods.
%for dissipation systems due to its importance to the long time numerical simulation, see e.g., \cite{tang2019energy} and references therein. Consequently, the stability of the continuous problem is validated.
\smallskip
\item We conduct a delicate and rigorous analyse of a semi-discrete scheme which can  incorporate various spatial discretisation.
More precisely, we propose a first-order backward-Euler-type scheme, and show  for the first time that  the discrete energy (with a modification of the continuous energy by adding a history part)  decays monotonically.  This strong stability guarantees that the scheme is  unconditionally stable and is essential for the convergence analysis.  However,  it appears nontrivial to show this  if one works with the fractional differential form
of the polarisation relation in the context of the C-C model \cite{li2011developing}.   %To the best of our knowledge,

 \smallskip
%By incorporating the piecewise constant approximation of the time convolution of electric field (see \eqref{frarect}) into H-N model, we develop a backward Euler type time discrete scheme for the H-N model. It is worth noting that our method is implemented in the time-domain all the time, without any approximation in the frequency domain like in all aforementioned works. By a delicate study of the discrete coefficients and a technique derivation, a discrete energy dissipation law of the electric field and the magnetic field in the H-N model which has the exact form of the continuous energy stability is proved under the PEC boundary condition. Furthermore, a monotonic energy decay is also derived. Subsequently, stability and convergence of the time discrete scheme are given.

\item A fast temporal convolution algorithm for the H-N model is realised by following some basic ideas in \cite{Lubich2002fast,zeng2018stable}, which requires $O(\log N_t)$ storage and $O(N_t \log N_t)$ operations over $N_t$ time steps, when only cost in time direction is considered. Here, $N_t$ represents the total number of time steps. Note that the direct implementation of the scheme \eqref{frarect} would require $O(N_t)$ storage and $O(N_t^2)$ operations, which is computational expensive and forms a bottleneck for long time simulation. %Note that the kernel function in H-N model (see \eqref{gml}) is much more complex than the kernel functions in \cite{Lubich2002fast,zeng2018stable} and references therein, so how to develop a fast convolution algorithm for the H-N model is not an easy task. Until now, there are only a few works on developing fast algorithms for the H-N model. The fast algorithm in \cite{causley2011incorporating} get the approximation of local part, which is only efficient for $0.5 < \alpha, \beta \leq 1$, by the Laurent expansion in the frequency domain, while get the approximation of the history part by following the idea in \cite{Alpert2000rapid,jiang2004fast}. A bootstrap method was presented by Xu et al. \cite{xu2013a} for finding efficient sum-of-poles approximations of causal functions in the complex domain which contain H-N, C-C and D-C kernel functions.
\end{itemize}

%To our best knowledge, we would like to emphasize that the first two items are new and important results that have never been found in the exist works.

%Compared to the classical Maxwell’s equations (cf. \cite{taflove2005computational,monk2003finite,Chen2008energy}), the most interesting but challenging issue in H-N model is the treatment of the singular integral in time. Without loss of generality, we only provide a fully discrete scheme for two-dimensional H-N model, where Legendre-Galerkin method is used for space discretisation. Stability and convergence analysis of the fully discrete scheme are also presented.

The rest of this paper is organised as follows. In the next section, we introduce  the Havriliak-Negami dispersive dielectric model, and conduct the stability analysis. In section \ref{Sect3}, we propose a time discrete scheme for the H-N model,
and provide its stability and error analysis. In section \ref{Sect4}, we implement  a fast temporal convolution algorithm, and illustrate spatial discretisation through a two-dimensional H-N model. Then  we supply with ample numerical results to demonstrate the efficiency and accuracy of the proposed scheme. Furthermore, we apply the solver to an interesting application
in the recovery of the complex relative permittivity and some related physical quantities. Finally, concluding remarks are made in Section \ref{Sect5}, and some important properties of the Mittag-Leffler function are collected in Appendix \ref{AppendixA}.
 %experiments are conducted to validate our numerical methods and illustrate their application.

At the end of this section, we introduce some notations to be used throughout the paper.  Let $C$ (sometimes with a subindex) denote a generic constant independent of the time step size $\Delta t$ and the space parameter $N$. For $r\ge 0,$ let $H^r(\Omega)$ (resp. ${\bm H}^r(\Omega)$) be the usual  Sobolev space with $H^0(\Omega)=L^2(\Omega)$ (resp. ${\bm H}^0(\Omega)={\bm L}^2(\Omega)$) for the scalar (resp. vector-valued) functions on a bounded domain $\Omega$ with Lipschitz boundary.  As usual, we denote the inner product and norm of both ${L}^2(\Omega)$ and ${\bm L}^2(\Omega)$ by $(\cdot, \cdot)$ and $\|\cdot\|,$ respectively. With a little abuse of notation, we  denote the norms of both ${L}^{\infty}(\Omega)$ and ${\bm L}^{\infty}(\Omega)$ by $\|\cdot\|_{\infty}$.
% {\color{blue}
%As usual,   we use  ${\bm H}^r(\Omega)\triangleq \left(H^r(\Omega)\right)^d (d=1,2,3)$  be the Hilbert space, and the space ${\bm H}^0(\Omega)={\bm L}^2(\Omega)$ for vector functions.} The ${\bm L}^2(\Omega)$ inner product is defined as (cf. \cite{monk2003finite})
%$$({\bm u},{\bm v})=\int_{\Omega} \bm u\cdot \bar {\bm v} \;{\rm d}\bs x, \qquad \forall ~~{\bm u}=(u_1,\cdots,u_d)^T,{\bm v}=(v_1,\cdots,v_d)^T \in {\bm L}^2(\Omega).$$
%{\color{blue}Hereafter, $\|\cdot\|$ will always mean the ${\bm L}^2(\Omega)$-norm (or $L^2(\Omega)$-norm if only scalar functions are involved).}
Given  a Hilbert space $\bm X$ with the norm $\|\cdot\|_{\bm X}$, we define the spaces $L^{\infty}\left(0,T;{\bm X}\right)$ and $L^{2}\left(0,T;{\bm X}\right)$ with the norms
$$\|\bm U\|_{L^{\infty}\left(0,T;{\bm X}\right)}=\mathop{\esssup}_{0\le t \le T}\|{\bm U}(\cdot, t)\|_{\bm X}, \quad \|\bm U\|_{L^2\left(0,T;{\bm X}\right)}=\Big(\int_0^T\|{\bm U}(\cdot, t)\|_{\bm X}^2{\rm d}t\Big)^2.$$
We further introduce
$$H^{k}\left(0,T;{\bm X}\right)\triangleq \big\{{\bm v} \in L^{2}\left(0,T;{\bm X}\right)\,:\, \pt^{\ell}{\bm v} \in L^{2}\left(0,T;{\bm X}\right),~1\le \ell \le k\big\}
$$
with the norm $\|\cdot\|_{H^{k}(0,T;{\bm X})}$ (cf. \cite{Quarteroni1994}).
We also use some common notation (cf. \cite{monk2003finite})
$$H({\rm curl};\Omega)=\big\{{\bm v} \in {\bm L}^2(\Omega);\; \nabla\times{\bm v} \in {\bm L}^2(\Omega)\big\}, \quad H_0({\rm curl};\Omega)=\big\{{\bm v} \in H({\rm curl};\Omega);\; {\bm n}\times{\bm v}=0\; {\rm on}\; \partial \Omega\big\}.$$

%For simplicity, we will omit the mention of the space variable in the notation, whenever no confusion is possible.

\section{The Havriliak-Negami dispersive dielectric model}\label{Sect2}
\setcounter{equation}{0}
\setcounter{lmm}{0}
\setcounter{thm}{0}

In an  H-N medium, the time-domain Maxwell's equations take the form (cf.  \cite{havriliak1966complex,havriliak1967complex}):
%{rekanos2012fdtd,Chakarothai2019}):\footnote{(i) Is this change OK ?}
\begin{equation}\label{hnte}
\epsilon_{0}\epsilon_{\infty} \pt{\bm E}
 =  \nabla\times{\bm H} -  \pt{\bm P}, \quad \mu_{0}\pt {\bm H}
 =  -\nabla\times{\bm E}\quad\;\;      {\rm in} \;\; \Omega \times (0,T],
\end{equation}
where ${\bm P}({\bm x},t)$ is the induced electric polarisation given by   % \cite{causley2011incorporating}
\begin{equation}
\label{polartimeOri}
{\bm P}({\bm x},t)=\int_{0}^{t} \xi_{\alpha,\beta}(t-s) {\bm E}({\bm x},s)\, {\rm d}s,  \quad  \forall\, (\bm x, t)\in \Omega \times (0,T].
\end{equation}
Here $\xi_{\alpha,\beta}$ is the time-domain susceptibility kernel which involves the inverse Laplace transform as follows
\begin{equation}\label{addA1}
\xi_{\alpha,\beta}(t):=\mathscr{L}^{-1}\bigg[\frac{\epsilon_0(\epsilon_s-\epsilon_{\infty})}
{\left(1+(s\tau_0)^{\alpha}\right)^{\beta}}\bigg](t),  %=\int_0^\infty ???,
\end{equation}
and $\tau_0, \epsilon_s, \epsilon_{\infty}, \alpha, \beta$ are given in \eqref{hnperm}. As usual,  ${\bm E}$ is the electric field, ${\bm H}$ is the magnetic field, and $\epsilon_0, \mu_0$ are the permittivity and permeability of the free space, respectively.
The system  \eqref{hnte}-\eqref{polartimeOri}  is  supplemented with the perfect electrical conductor (PEC) condition
\begin{eqnarray}
\label{hntebOri}
{\bm n}\times {\bm E}=\bm 0 \quad {\rm at} \;\; \partial \Omega \times (0,T],
\end{eqnarray}
and the initial conditions
\begin{eqnarray}
\label{hnteiOri}
{\bm E}({\bm x},0) = {\bm E}_0({\bm x}),\quad {\bm H}({\bm x},0) = {\bm H}_0({\bm x}),\quad {\bm P}({\bm x},0)=\bm 0\quad {\rm in} \;\; \Omega,
\end{eqnarray}
where the last condition is a direct consequence of the representation  \eqref{polartimeOri}. Here the constitutive relations in an H-N medium are
\begin{eqnarray*}
{\bm D} = \epsilon_{0}\epsilon_{\infty}{\bm E} +{\bm P}, \quad {\bm B} = \mu_{0}{\bm H},
\end{eqnarray*}
where ${\bm D}$ is the electric flux density, and ${\bm B}$ is the magnetic flux density.

As the values of the parameters $\epsilon_0$, $\mu_0$ and $\tau_0$ are excessively small (of order $10^{-12}$,  $10^{-7}$ and  $10^{-12},$ respectively), we find it  is  more desirable to rescale the model for both computational and analysis purposes. Indeed,  the  introduction of  non-dimensional quantities can  avoid dealing with excessively small or large numbers in finite-precision arithmetic (cf. \cite[P. 294]{demkowicz2006computing}).

%%%%%%%%%%%%%%%%%%%%%%%%%%%%%%%%%%%%%%%%%%%%%%%%%%%%%%%%%%%%%%%%%%%%%%%%%%%%%%%%%%%%%
%%%%%%%%%%%%%%%%%%%%%%%%%%%%%%%%% Lemma Rescaled H-N model  %%%%%%%%%%%%%%%%%%%%%%%%%
%%%%%%%%%%%%%%%%%%%%%%%%%%%%%%%%%%%%%%%%%%%%%%%%%%%%%%%%%%%%%%%%%%%%%%%%%%%%%%%%%%%%%
\begin{lmm}\label{RescaledHN} Using the substitutions and change of variables
\begin{eqnarray}
\label{scaling}
{\bm E}\to \sqrt{\epsilon_0}{\bm E}, \quad {\bm P}\to \frac{1}{\sqrt{\epsilon_0}}{\bm P}, \quad
{\bm H}\to \sqrt{\mu_0}\,{\bm H}, \quad t \to \frac{t}{\tau_0}, \quad {\bm x} \to%\frac{\sqrt{\epsilon_0\mu_0}}{\tau_0}{\bm x}=
 \frac{{\bm x}}{c_0\tau_0},\quad c_0:=
 \frac{1}{\sqrt{\epsilon_0\mu_0}},
\end{eqnarray}
%where $c_0=\frac{1}{\sqrt{\epsilon_0\mu_0}}$ is the speed of the light in free space.
we can convert the system  \eqref{hnte}-\eqref{hnteiOri} into
\begin{subequations}\label{scaledsystem}
\begin{numcases}{}
\epsilon_{\infty} \pt{\bm E} +\pt{\bm P} = \nabla\times{\bm H}, \quad
\pt{\bm H} =-\nabla\times{\bm E} \quad {\rm in} \;\; \Omega \times (0,T], \label{rehnmodele}\\[0.2cm]
%\quad {\rm in} \;\; \Omega \times (0,T],\label{rehnmodelm}\\[0.2cm]
{\bm P}({\bm x},t)={\Delta \epsilon}\int_{0}^{t}  e_{\alpha,\alpha\beta}^{\beta}(t-s;-1){\bm E}({\bm x},s)\,{\rm d}s\quad {\rm in} \;\;  \Omega \times (0,T],\label{rehnmodelp}\\[0.2cm]
{\bm E}({\bm x},0) = {\bm E}_0({\bm x}),~~{\bm H}({\bm x},0) = {\bm H}_0({\bm x}),~~{\bm P}({\bm x},0)=\bm 0 \quad {\rm in} \;\; \Omega,\label{rehnmodeli}\\[0.2cm]
{\bm n}\times {\bm E}=\bm 0 \quad {\rm at} \;\;  \partial \Omega \times (0,T], \label{rehnmodelb}
\end{numcases}
\end{subequations}
where $\Delta \epsilon:=\epsilon_s-\epsilon_{\infty}$, and
\begin{equation}\label{gml}
e_{\rho,\mu}^{\gamma}(t;\sigma)=t^{\mu-1}E_{\rho,\mu}^{\gamma}(\sigma t^{\rho}), \quad
E_{\rho,\mu}^{\gamma}(z)=\sum_{k=0}^{\infty} \frac{\Gamma(k+\gamma) } {\Gamma(k) \Gamma(\rho k+\mu)}\frac{z^k}{k!},
\end{equation}
%with $E_{\rho,\mu}^{\gamma}(z)$ being
i.e.,  the Mittag-Leffler {\rm(}ML{\rm)} function with three parameters {\rm(}also known as the Prabhakar function, see   \cite{gorenflo2014mittag,prabhakar1971singular} and Appendix \ref{AppendixA}{\rm)}.
%\begin{eqnarray*}
%%\label{ML}
%E_{\rho,\mu}^{\gamma}(z)=\sum_{k=0}^{\infty} \frac{(\gamma)_k}{\Gamma(\rho k+\mu)}\frac{z^k}{k!},
%\end{eqnarray*}
%and $(\gamma)_k$ is the Pochhammer symbol \cite[(A.1.17)]{gorenflo2014mittag}.
\end{lmm}
\begin{proof} One verifies readily that with \eqref{scaling},  the rescaled system \eqref{scaledsystem} (except for \eqref{rehnmodelp}) can be reduced from  \eqref{hnte} and \eqref{hntebOri}-\eqref{hnteiOri}  directly.

Now,  we consider the derivation of \eqref{rehnmodelp}. According to \eqref{LTeML} and \eqref{addA1}, we have
\begin{eqnarray}
\label{LTHNker}
\xi_{\alpha,\beta}(t)=\frac{\epsilon_0{\Delta \epsilon}}{\tau_0^{\alpha\beta}}\mathscr{L}^{-1}\big[\left(s^{\alpha}+1/{\tau_0^\alpha}\right)^{-\beta}\big](t)
=\frac{\epsilon_0{\Delta \epsilon}}{\tau_0^{\alpha\beta}} e_{\alpha,\alpha\beta}^{\beta}\Big(t;-\frac{1}{{\tau_0^\alpha}}\Big).
\end{eqnarray}
Then \eqref{polartimeOri} can be written as
\begin{eqnarray*}
%\label{polartime_eml}
{\bm P}({\bm x},t)=\frac{\epsilon_0{\Delta \epsilon}}{\tau_0^{\alpha\beta}}\int_{0}^{t} e_{\alpha,\alpha\beta}^{\beta}(t-s;-1/{\tau_0^\alpha}) {\bm E}({\bm x},s)\,{\rm d}s.
\end{eqnarray*}
With the substitution $s \to \frac{s}{\tau_0}$ and $t \to \frac{t}{\tau_0}$, we can obtain \eqref{rehnmodelp} from the above. %\eqref{scaling}. This ends the derivation.
\end{proof}

Formally,  the rescaled polarisation relation \eqref{rehnmodelp} can be reformulated as a fractional ``differential" form using the Prabhakar integrals/derivatives (cf. \cite{garrappa2016models,giusti2020}), which turns out to be important for the
stability analysis of the re-scaled model  \eqref{scaledsystem}.

%%%%%%%%%%%%%%%%%%%%%%%%%%%%%%%%%%%%%%%%%%%%%%%%%%%%%%%%%%%%%%%%%%%%%%%%%%%%%%%%%%%%%
%%%%%%%%%%%%%%%%%%%%%%%%%%%% Definition Prabhakar Operators %%%%%%%%%%%%%%%%%%%%%%%%%
%%%%%%%%%%%%%%%%%%%%%%%%%%%%%%%%%%%%%%%%%%%%%%%%%%%%%%%%%%%%%%%%%%%%%%%%%%%%%%%%%%%%%
\begin{defn}[{see \cite[(B.19)-(B.23)]{garrappa2016models} or \cite[(5.3)-(5.10)]{giusti2020}}]\label{PrabhakarOper} {\em For a function $f(t) \in L^1(0,T)$, the Prabhakar integral of order $\alpha, \beta > 0$ and with the parameter $\varrho > 0$ can be defined by
%\footnote{Please change the notation of the kernel! This definition is to reformulate $\bm P$, so we need to add some and also the C-C model reduces to ??. \textcolor[rgb]{1.00,0.00,0.00}{I added a remark!}}
\begin{eqnarray}
\left(_0{\mathcal {J}}^{\alpha}_t+\varrho\right)^{\beta}f(t)=\int_{0}^{t}e_{\alpha,\alpha\beta}^{\beta}(t-s;-\varrho)f(s)\,{\rm d}s,\quad t\in (0,T).
\label{PrabhakarOper1}
\end{eqnarray}
If, in addition,  $0 < \alpha\beta <1$, the left-inverse of the above integral operator is the special derivative
\begin{eqnarray}
\left(_0{\mathcal {D}}^{\alpha}_t+\varrho\right)^{\beta}f(t)=\frac{\rm d}{{\rm d} t}\int_{0}^{t}e_{\alpha,1-\alpha\beta}^{-\beta}(t-s;-\varrho)f(s) \,{\rm d}s.
\label{PrabhakarOper2}
\end{eqnarray}
For an absolutely continuous function $f(t)$, the Caputo-type derivative as the counterpart of the above derivative operator can be defined as
\begin{equation}\label{PrabhakarOper3}
{^{C}\!\!\left(_0{\mathcal {D}}^{\alpha}_t+\varrho\right)^{\beta}}f(t)
=\left(_0{\mathcal {D}}^{\alpha}_t+\varrho\right)^{\beta}\left(f(t)-f(0^+)\right)
=\int_{0}^{t}e_{\alpha,1-\alpha\beta}^{-\beta}(t-s;-\varrho)f'(s)\,{\rm d}s.
\end{equation}}
\end{defn}

In view of  \eqref{PrabhakarOper1} with $\varrho=1$,  we can write  \eqref{rehnmodelp} as
\begin{eqnarray}
\label{caputodef01}
{\bm P}({\bm x},t)={\Delta \epsilon} \left(_0{\mathcal {J}}^{\alpha}_t+1\right)^{\beta} \bm E(\bm x,t),
\end{eqnarray}
Taking the left-inverse operation \eqref{PrabhakarOper2} on both sides of \eqref{caputodef01},  we obtain immediately from  \eqref{PrabhakarOper3} that
\begin{eqnarray}
\label{caputodef0}
{\left(_0{\mathcal {D}}^{\alpha}_t+1\right)}^{\beta} {\bm P}({\bm x},t)={^{C}\!\!\left(_0{\mathcal {D}}^{\alpha}_t+1\right)}^{\beta} {\bm P}({\bm x},t)={\Delta \epsilon}\, {\bm E}({\bm x},t).
\end{eqnarray}
It is noteworthy that when $\beta=1$ (i.e., the C-C model), the involved fractional derivatives simply become the usual fractional Riemann-Liouville derivative and Caputo derivative operators as in \cite{Podlubny1999}.
In fact, fractional ADE-based approaches for the C-C model are based upon such a formulation. However,  for the general H-N model,  we find the integral formulation  \eqref{caputodef01} is more suitable for the implementation, but the formulation \eqref{caputodef0} is useful in the   analysis.

%%%%%%%%%%%%%%%%%%%%%%%%%%%%%%%%%%%%%%%%%%%%%%%%%%%%%%%%%%%%%%%%%%%%%%%%%%%%%%%%%%%%%
%%%%%%%%%%%%%%%%%%%%%%%%%%%%%%%%% Lemma positiveker %%%%%%%%%%%%%%%%%%%%%%%%%%%%%%%%%
%%%%%%%%%%%%%%%%%%%%%%%%%%%%%%%%%%%%%%%%%%%%%%%%%%%%%%%%%%%%%%%%%%%%%%%%%%%%%%%%%%%%%
\begin{lmm}\label{lempositiveker} If $0<\alpha,\beta \leq 1$ and $\varrho>0,$ then
the kernel $e_{\alpha,1-\alpha\beta}^{-\beta}(t;-\varrho)$ in \eqref{PrabhakarOper2} and \eqref{caputodef0} is positive-definite in the sense that
\begin{eqnarray*}%\label{kernelnewA}
\int_{0}^{T}\phi(t)\int_{0}^{t}e_{\alpha,1-\alpha\beta}^{-\beta}(t-s;-\varrho)\phi(s)\, {\rm d}s\, {\rm d}t \geq 0,\quad \forall \phi \in C[0,T].
\end{eqnarray*}
\end{lmm}
\begin{proof}  According to  \cite[(1.2)]{mclean1996discretization},  it suffices to show that  the kernel function
$\mathcal{K}(t):=e_{\alpha,1-\alpha\beta}^{\;-\beta}(t;-\varrho)$ satisfies
\begin{eqnarray*}%\label{realCondition}
{\rm Re}\big\{\mathscr{L}\left[\mathcal{K}(t)\right]( i \omega)\big\} \geq 0, \quad \forall\, \omega>0,
\end{eqnarray*}
where ${\rm Re}\{u\}$ stands for the real part of $u$ and $i$ is the imaginary unit.
Using  \eqref{LTeML} with $\gamma=-\beta, \rho=\alpha, \mu=1-\alpha\beta,\sigma=\varrho$ and
$s=i\omega,$ we find from direct calculation that
\begin{equation*}
\begin{split}
\mathscr{L}\left[\mathcal{K}(t)\right]( i \omega)& =\frac{(i\omega)^{\alpha(-\beta)-(1-\alpha\beta)}}{((i\omega)^{\alpha}+\varrho)^{-\beta}} = \frac {(\varrho+(i\omega)^{\alpha})^{\beta}}{i\omega}
=-i \omega^{-1}\left(\varrho+\omega^\alpha \cos\frac{\pi \alpha}{2}+i \omega^\alpha \cos\frac{\pi \alpha}{2} \right)^\beta\\
& =-i \omega^{-1}r^{\beta}\left(\cos \beta\theta+ i \sin \beta\theta\right) =\omega^{-1}r^{\beta}\left(\sin \beta\theta-i\cos \beta\theta\right),
\end{split}
\end{equation*}
where
\begin{eqnarray*}
r=\sqrt{\varrho^2+2\omega^{\alpha}\cos\frac{\pi \alpha}{2}+\omega^{2\alpha}},\quad \tan\theta=\frac{\omega^{\alpha}\sin\frac{\pi \alpha}{2}}{\varrho+\omega^{\alpha}\cos\frac{\pi \alpha}{2}}.
\end{eqnarray*}
As the parameters $0<\alpha,\beta \leq 1$ and $\varrho>0,$ it is evident that  $\theta \in (0,\pi/2).$ Therefore,
\begin{eqnarray*}
{\rm Re}\big\{\mathscr{L}\left[\mathcal{K}(t)\right]( i \omega)\big\}= \omega^{-1}r^{\beta}\sin \beta\theta \geq 0, \quad \forall\,\omega>0,
\end{eqnarray*}
which completes the proof.
\end{proof}

\begin{rem}\label{HNdivergence} {\em With the aid of Lemma \ref{lempositiveker},
%Like  the Cole-Cole model in  \cite{li2011developing},
we can show that in the H-N model \eqref{scaledsystem},  if the initial
electric and magnetic fields are divergence free (i.e., $\nabla\cdot{\bm E}_0=\nabla\cdot{\bm H}_0=0$), then we have
$\nabla\cdot{\bm E}=\nabla\cdot{\bm H}=\nabla\cdot{\bm P}=0$ in $\Omega \times (0,T].$ Indeed,
taking the divergence of two equations in \eqref{rehnmodele}, we find readily that  $\nabla\cdot{\bm H}=0$  and
$\nabla\cdot(\epsilon_{\infty} {\bm E} +{\bm P})=0.$ Thus we derive from \eqref{PrabhakarOper3} and  \eqref{caputodef0} that
%Similarly, we derive from \eqref{caputodef01} and  \eqref{DivDFree} that
\begin{equation}\label{DivP}
\epsilon_{\infty}{^{C}\!\!\left(_0{\mathcal {D}}^{\alpha}_t+1\right)}^{\beta}\, \nabla\cdot{\bm P}=
\epsilon_{\infty} \int_{0}^{t}e_{\alpha,1-\alpha\beta}^{-\beta}(t-s;-1)\, \partial_s \nabla\cdot{\bm P}\, {\rm d}s =-{\Delta \epsilon}\, \nabla\cdot{\bm P}.
\end{equation}
Testing \eqref{DivP}  with $\pt\,\nabla\cdot{\bm P}$ leads to
\begin{equation}\label{DivP2}
\epsilon_{\infty}\big({^{C}\!\!\left(_0{\mathcal {D}}^{\alpha}_t+1\right)}^{\beta} \nabla\cdot{\bm P}(\cdot,t),\pt\nabla\cdot{\bm P}(\cdot,t)\big) +{\Delta \epsilon} \big(\nabla\cdot{\bm P}(\cdot,t),\pt\nabla\cdot{\bm P}(\cdot,t)\big)=0.
\end{equation}
Integrating it with respect to $t$ from $t=0$ to any $T>0,$  we deduce from  \eqref{DivP2} and   Lemma \ref{lempositiveker} that  the  first term of the resulting equation is nonnegative, and $\|\nabla\cdot{\bm P}(\cdot,T)\|\le 0.$ Therefore, we can claim  $\nabla\cdot{\bm P}=0,$ so we can further derive  $\epsilon_{\infty}\nabla\cdot {\bm E}=-\nabla\cdot {\bm P}=0.$  }
\end{rem}

With Lemma \ref{lempositiveker} at our disposal, we  can prove  the following stability result for  the H-N model \eqref{scaledsystem}.
%%%%%%%%%%%%%%%%%%%%%%%%%%%%%%%%%%%%%%%%%%%%%%%%%%%%%%%%%%%%%%%%%%%%%%%%%%%%%%%%%%%%%
%%%%%%%%%%%%%%%%%%%%%%%%%%%%%%%%% Theorem ConStab %%%%%%%%%%%%%%%%%%%%%%%%%%%%%%%%%%%%%%%
%%%%%%%%%%%%%%%%%%%%%%%%%%%%%%%%%%%%%%%%%%%%%%%%%%%%%%%%%%%%%%%%%%%%%%%%%%%%%%%%%%%%%
\begin{thm}\label{thmConStab} If $\bm E_0, \bm H_0\in {\bm L}^2(\Omega)$ in \eqref{scaledsystem},   then its solution $\bm E, \bm H, \bm P\in L^\infty(0,T; {\bm L}^2(\Omega))$ satisfying
\begin{equation}\label{constabthm2}
{\mathscr E}(t):=\epsilon_{\infty}\|{\bm E}(\cdot, t) \|^2+\|{\bm H}(\cdot, t)\|^2
\leq \epsilon_{\infty}\|{\bm E}_0\|^2+\|{\bm H_0}\|^2:= {\mathscr E}_0,\quad \forall t\in (0,T),
\end{equation}
and
\begin{eqnarray}
& &\|{\bm P}\|_{L^{\infty}(0,T;{\bm L}^2(\Omega))}\leq {\Delta \epsilon}\, B~\|{\bm E}\|_{L^{\infty}(0,T;{\bm L}^2(\Omega))}, \label{constabthm3}
\end{eqnarray}
where the constant $B$ is given by
\begin{eqnarray*}
%\label{MLBound}
B=T^{\alpha\beta}\sum_{k=0}^{\infty}\frac{|(\beta)_k|}{\alpha( k+\beta)|\Gamma(\alpha k+\alpha\beta)|}\frac{T^{\alpha k}}{k!}.
\end{eqnarray*}
%which is proved to be a constant in \cite[(4.4)]{kilbas2004generalized}.
\end{thm}
\begin{proof} Multiplying the first equation in \eqref{rehnmodele} by ${\bm E}$, and integrating the resulted equation by  the Green's formula over $\Omega$, we  obtain that
\begin{eqnarray}
\label{constab1}
\epsilon_{\infty} \left(\partial_t{\bm E},{\bm E}\right)+\left(\partial_t {\bm P},{\bm E}\right)-\left({\bm H}, \nabla\times{\bm E}\right)=0,
\end{eqnarray}
where we  used  the boundary condition \eqref{rehnmodelb}. Similarly, we derive from the second equation  in \eqref{rehnmodele}  that
\begin{eqnarray}
\label{constab2}
\left(\partial_t {\bm H},{\bm H}\right)+\left(\nabla\times{\bm E},{\bm H}\right)=0.
\end{eqnarray}
As a direct consequence of \eqref{constab1}-\eqref{constab2}, we have
\begin{eqnarray*}
%\label{constab3}
\epsilon_{\infty} \left(\partial_t{\bm E},{\bm E}\right)+\left(\partial_t {\bm H},{\bm H}\right)+\left(\partial_t {\bm P},{\bm E}\right)=0,\;\; {\rm i.e.,}\;\;\;  \frac {1} 2  {\mathscr E}'(t)  %\Big\{ \epsilon_{\infty} \|\bm E\|^2 +\|\bm H\|^2 \Big\}
=-\left(\partial_t {\bm P},{\bm E}\right).
\end{eqnarray*}
In view of \eqref{caputodef0}, we  eliminate $\bm E$ and integrate the resulted equation with respect to $t$
over $(0,T),$ which, together with \eqref{PrabhakarOper3} and Lemma \ref{lempositiveker},  leads to
\begin{equation*}
%\label{caputodefpra}
\begin{split}
  \frac {1} 2  {\mathscr E}(T)-  \frac {1} 2  {\mathscr E}_0 & =-\frac 1 {\Delta\epsilon}\int_0^T \big(
  \partial_t {\bm P},{^{C}\!\!\left(_0{\mathcal {D}}^{\alpha}_t+1\right)}^{\beta} {\bm P}\big)\, {\rm d} t\\
  &=
  -\frac 1 {\Delta\epsilon}\int_0^T\int_{0}^{t} e_{\alpha,1-\alpha\beta}^{-\beta}(t-s;-1)
   \big(\partial_t {\bm P},\partial_s {\bm P} \big) \,{\rm d}s\, {\rm d} t \le 0,
  \end{split}
\end{equation*}
where we recall that $\Delta\epsilon>0.$ This yields \eqref{constabthm2}.

We now turn to \eqref{constabthm3}. It is clear that  by \eqref{rehnmodelp},
\begin{equation*}
\begin{split}
|{\bm P}(\bm x, t)|
& \leq {\Delta \epsilon}\int_{0}^{t} \big|e_{\alpha,\alpha\beta}^{\beta}(t-s;-1){\bm E}(\bm x, s)\big|\,{\rm d}s
%\\&
\leq {\Delta \epsilon}\sup_{0\leq s \leq t}|{\bm E}(\bm x, s)|\int_{0}^{t} \big|e_{\alpha,\alpha\beta}^{\beta}(t-s;-1)\big|\,{\rm d}s\\
&\leq {\Delta \epsilon}\sup_{0\leq s \leq t}|{\bm E}(\bm x, s)| \int_{0}^{T} \big|e_{\alpha,\alpha\beta}^{\beta}(u;-1)\big|\,{\rm d}u.
\end{split}
\end{equation*}
We derive from  \eqref{gml} and direct calculation that
\begin{equation*}
\begin{split}
& \int_{0}^{T} \big|e_{\alpha,\alpha\beta}^{\beta}(u;-1)\big|\,{\rm d}u
=\int_{0}^{T} u^{\alpha\beta-1}\big|E_{\alpha,\alpha\beta}^{\beta}(u;-1)\big|\,{\rm d}u
\leq \sum_{k=0}^{\infty}\frac{|(\beta)_k|}{\Gamma(\alpha k+\alpha\beta)}\frac{1}{k!}\int_{0}^{T}u^{\alpha k+\alpha\beta-1}\,{\rm d}u:=B,
\end{split}
\end{equation*}
where the quantity $B$ is finite  (cf.  \cite[Theorem 5]{kilbas2004generalized} for the estimates of generalised ML functions).
Thus,  we obtain  the first inequality but wish to show the second inequality below
\begin{eqnarray}\label{constab9}
|{\bm P}(\bm x, t)|^2 \leq ({\Delta \epsilon})^2 B^2\, \Big(\sup_{0\leq s \leq t}|{\bm E}(\bm x, s)|\Big)^2\le ({\Delta \epsilon})^2 B^2 \, \sup_{0\leq s \leq t}|{\bm E}(\bm x, s)|^2.
\end{eqnarray}
Let  $s_0 \in [0,t]$ satisfy
$$|{\bm E}(\bm x, s_0)|^2=\sup_{0\leq s \leq t} |{\bm E}(\bm x, s)|^2,\;\;  {\rm so}\;\; |{\bm E}(\bm x, s_0)| \geq |{\bm E}(\bm x, s)|,\;\;  \forall s \in [0,t],$$
which implies
$$|{\bm E}(\bm x, s_0)|^2 \geq \big(\sup_{0\leq s \leq t} |{\bm E}(\bm x, s)|\big)^2,$$
leading to the second inequality in \eqref{constab9}.
Therefore, we have
\begin{eqnarray*}
%\label{constab9}
\sup_{0\leq s \leq t}|{\bm P}(\bm x,s)|^2 \leq ({\Delta \epsilon})^2B^2 \sup_{0\leq s \leq t}|{\bm E}(\bm x,s)|^2,\quad \forall t\in (0,T].
\end{eqnarray*}
Integrating this inequality over $\Omega$, leads to \eqref{constabthm3}.
\end{proof}

\begin{rem}\label{HNregularity} {\em It is known from  standard analysis that the $L^2$-stability in Theorem  \ref{thmConStab} can ensure the uniqueness of the solution of the H-N model \eqref{scaledsystem}.  In fact, we can follow the argument in  \cite[Theorem 3.8]{li2013time-domain} for the Drude model to show the existence of the solution.  Here, we sketch the idea for the readers' reference.  Let $\hat f(s)$ be the Laplace transform of  $f(t), t\ge 0$. %by ${\hat f}(s) = \int_{0}^{\infty}e^{-st}f(t){\rm d}t$. T
 Then we can transform  \eqref{scaledsystem} into
\begin{equation}\label{Eequation}
 \epsilon_{\infty} (s{\hat {\bm E}}-{\bm E}_0)+s{\hat {\bm P}}= \nabla\times{\hat {\bm H}},\quad
 s{\hat {\bm H}}-{\bm H}_0= -\nabla\times{\hat {\bm E}},\quad
{\hat {\bm P}} = \frac{\Delta \epsilon}{(1+s^{\af})^{\bt}}{\hat {\bm E}}.
\end{equation}
%which leads to
%\begin{eqnarray*}
%& & \left[\epsilon_{\infty}+\frac{\Delta \epsilon}{(1+s^{\af})^{\bt}}\right]s{\hat {\bm E}}= \nabla\times{\hat {\bm H}}+\epsilon_{\infty}{\bm E}_0+{\bm P}_0,\\
%& & s{\hat {\bm H}}-{\bm H}_0= -\nabla\times{\hat {\bm E}}.
%\end{eqnarray*}
Eliminating ${\hat {\bm H}}$ and ${\hat {\bm P}}$ from the first equation by other two equations, yields
\begin{eqnarray*}
\Big(\epsilon_{\infty}+\frac{\Delta \epsilon}{(1+s^{\af})^{\bt}}\Big)s^2\,{\hat {\bm E}} + \nabla\times\nabla\times{\hat {\bm E}}=\epsilon_{\infty}s\,{\bm E}_0+\nabla\times{\bm H}_0.
\end{eqnarray*}
A weak form is to find ${\hat {\bm E}} \in H_0({\rm curl},\Omega)$ such that
\begin{eqnarray*}
\Big(\epsilon_{\infty}+\frac{\Delta \epsilon}{(1+s^{\af})^{\bt}}\Big)s^2({\hat {\bm E}},{\bm \phi}) + (\nabla\times{\hat {\bm E}},\nabla\times {\bm \phi})=(\epsilon_{\infty}s{\bm E}_0+\nabla\times{\bm H}_0,{\bm \phi}), \quad \forall {\bm \phi} \in H_0({\rm curl},\Omega).
\end{eqnarray*}
 We infer from  the Lax-Milgram lemma that for any $s>0,$ it admits  a unique solution ${\hat {\bm E}} \in H_0({\rm curl},\Omega)$, provided that ${\bm E}_0, \nabla\times{\bm H}_0\in \bs L^2(\Omega)$ {\rm(}cf. \cite{cohen2017finite}{\rm).}  The inverse Laplace transform of  ${\hat {\bm E}}$ is  ${\bm E}$, and the uniqueness of ${\bm E}\in H_0({\rm curl},\Omega)$ follows from the uniqueness of the Laplace transform.  Then we have the regularity of $\bs P \in H_0({\rm curl},\Omega)$ from the last equation of \eqref{Eequation}.    The existence and uniqueness of ${\bm H} \in H({\rm curl},\Omega)$ can be assured by the same argument.}
\end{rem}

\section{Analysis of a semi-discrete time-discretisation scheme}\label{Sect3}
\setcounter{equation}{0}
\setcounter{lmm}{0}
\setcounter{thm}{0}

In this section, we propose a  time-discretisation scheme  for the H-N model \eqref{scaledsystem},
%\eqref{rehnmodele}-\eqref{rehnmodelb},
and conduct the stability and convergence analysis.

\subsection{Time discretisation} \label{subsect31}
We start with  a weak form of \eqref{scaledsystem}. Multiplying  three equations in \eqref{scaledsystem} by the respective test functions, integrating over $\Omega$ and using the boundary conditions, we  follow the framework in \cite[P. 18-19]{cohen2017finite}  and arrive at  the weak form, that is, to find ${\bm E} \in H_0({\rm curl},\Omega)$ and ${\bm H}, {\bm P}  \in {\bm L}^2(\Omega)$
%and ${\bm P} \in {\bm L}^2(\Omega$
such that
\begin{subequations}\label{31eqn}
%\begin{numcases}{}
\begin{align}
& \epsilon_{\infty} \left(\partial_t {\bm E}, {\bm \phi}\right)+\left(\partial_t {\bm P}, {\bm \phi}\right)-\left({\bm H}, \nabla\times{\bm \phi}\right)=0, \quad \quad \;\;\; \quad \forall {\bm \phi} \in H_0({\rm curl},\Omega),  \label{hnteweak1}
\\[4pt] &\left(\partial_t{\bm H}, {\bm \psi}\right)+\left(\nabla\times{\bm E}, {\bm \psi}\right)=0, \;\; \quad \qquad \qquad \qquad \qquad\quad ~~ \forall {\bm \psi} \in {\bm L}^2(\Omega),  \label{hnteweak2}
%\end{equation}
\\& \left({\bm P}, {\bm \varphi}\right)={\Delta \epsilon}\int_{0}^{t} e_{\alpha,\alpha\beta}^{\beta}(t-s;-1)\left({\bm E}(\cdot, s), {\bm \varphi}\right){\rm d}s,\quad\;\;\;  \forall {\bm \varphi} \in {\bm L}^2(\Omega). \label{hnteweak3}
%%\end{numcases}
\end{align}
\end{subequations}
Note that the result in Remark \ref{HNregularity} (based on the argument in \cite[Theorem 3.8]{li2013time-domain}) carries over to this problem with a suitable  regularity assumption on the initial fields $\bs E_0, \bs H_0$.

We partition the time interval $[0,T],$ and denote
$$ t_k=k{{\Delta t}},\quad k=0,1,\cdots,N_t,\quad \Delta t=T/{N_t};  \quad \delta_t u^k=\frac{u^k-u^{k-1}}{{{\Delta t}}},$$
where $u^k$ stands for the approximation of $u$ at time $t_k.$

We first consider the time discretisation  of  \eqref{rehnmodelp}, and employ  the piecewise constant approximation $I_{\Delta t} {\bm E}$ of $\bm E:$
\begin{equation}
\label{frarect}
\begin{split}
{\bm P}(\bm x, t_k)
& =  {\Delta \epsilon} \int_{0}^{t_k} e_{\alpha,\alpha\beta}^{\beta}(t_k-s;-1)\,I_{\Delta t} {\bm E}(\bm x, s)\,{\rm d}s+{\Delta \epsilon}\,{\bm R}^k_0(\bm x)\\
& =  {\Delta \epsilon} \sum_{j=1}^{k}\Big(\int_{t_{j-1}}^{t_j} e_{\alpha,\alpha\beta}^{\beta}(t_k-s;-1)\,{\rm d}s\Big){\bm E}(\bm x, t_j)+{\Delta \epsilon}\,{\bm R}^k_0(\bm x)\\
& = {\Delta \epsilon}\sum_{j=1}^{k}\varpi^{(\alpha,\beta)}_{k-j}{\bm E}(\bm x, t_j)+{\Delta \epsilon}\,{\bm R}^k_0(\bm x),  \quad k \geq 1,
\end{split}
\end{equation}
where the residual and the weights are given by
\begin{equation}\label{TruncErr0}
\begin{split}
 {\bm R}^k_0 (\bm x) & :=\int_{0}^{t_k} e_{\alpha,\alpha\beta}^{\beta}(t_k-s;-1)\,\left({\bm E}(\bm x, s)-I_{\Delta t} {\bm E}(\bm x, s)\right){\rm d}s\\
& =\sum_{j=1}^{k}\int_{t_{j-1}}^{t_j}\! e_{\alpha,\alpha\beta}^{\beta}(t_k-s;-1)\,\left({\bm E}(\bm x, s)-{\bm E}(\bm x, t_j)\right){\rm d}s,
\end{split}
\end{equation}
and
\begin{equation}\label{HNweights}
\varpi^{(\alpha,\beta)}_{k-j}:=\int_{t_{j-1}}^{t_j} e_{\alpha,\alpha\beta}^{\beta}\left(t_k-s;-1\right)\,{\rm d}s
= \int_{(k-j)\Delta t}^{(k-j+1)\Delta t} e_{\alpha,\alpha\beta}^{\beta}\left(s;-1\right)\,{\rm d}s,
\end{equation}
respectively. By \eqref{IntgML}, we can rewrite the weights as
\begin{equation}\label{fracrectw1}
\begin{split}
\varpi^{(\alpha,\beta)}_{k-j}
%& = &\int_{t_{j-1}}^{t_j} e_{\alpha,\alpha\beta}^{\beta}\left(t_k-s;-1\right){\rm d}s \nonumber \\
&=e_{\alpha,\alpha\beta+1}^{\beta}\big((k-j+1){\Delta t};-1\big)-e_{\alpha,\alpha\beta+1}^{\beta}\big((k-j){\Delta t};-1\big), \quad 1\leq j \leq k.
%\\&= ({\Delta t})^{\alpha\beta}\left\{e_{\alpha,\alpha\beta+1}^{\beta}\left(k-j+1;-({\Delta t})^\alpha\right)-
%e_{\alpha,\alpha\beta+1}^{\beta}\left(k-j;-({\Delta t})^\alpha\right)\right\}, \quad 1\leq j \leq k,
\end{split}
\end{equation}
Note  that  we can compute them accurately by using the codes in \cite{garrappa2015numerical} for the ML functions.
%%%%%%%%%%%%%%%%%%%%%%%%%%%%%%%%%%%%%%%%%%%%%%%%%%%%%%%%%%%%%%%%%%%%%%%%%%%%%%%%%%%%%%
%%%%%%%%%%%%%%%%%%%%%%%%%%%%%%%%%% Lemma BoundKernel %%%%%%%%%%%%%%%%%%%%%%%%%%%%%%%%%%%
%%%%%%%%%%%%%%%%%%%%%%%%%%%%%%%%%%%%%%%%%%%%%%%%%%%%%%%%%%%%%%%%%%%%%%%%%%%%%%%%%%%%%%
%\begin{lmm}\label{lemBdker} {\rm (see \cite[p.\!\! 9]{prabhakar1971singular}).} For all $\rho >0$ and real $\gamma,\mu$, the Mittag-Leffler function with three parameters $E_{\rho,\mu}^{\gamma}(-z)$ is bounded in a finite interval.
%\end{lmm}

We have the following important property of the weights  in \eqref{fracrectw1}.
%%%%%%%%%%%%%%%%%%%%%%%%%%%%%%%%%%%%%%%%%%%%%%%%%%%%%%%%%%%%%%%%%%%%%%%%%%%%%%%%%%%%%
%%%%%%%%%%%%%%%%%%%%%%%%%%%%%%%%% Lemma DisKernel %%%%%%%%%%%%%%%%%%%%%%%%%%%%%%%%%%%
%%%%%%%%%%%%%%%%%%%%%%%%%%%%%%%%%%%%%%%%%%%%%%%%%%%%%%%%%%%%%%%%%%%%%%%%%%%%%%%%%%%%%
\begin{lmm}\label{lemdisker} For $0<\alpha<1, 0<\beta \leq 1,$ and $1\le k\le N_t,$  the weights
$\big\{\varpi^{(\alpha,\beta)}_{k-j}\big\}_{j=1}^k$ satisfy
%defined in \eqref{fracrectw1} be called as the discrete kernels corresponding to the continuous kernel $e_{\alpha,\alpha\beta}^{\beta}\left(t_k-s;-1\right), s \in (t_{k-1},t_k)$, then the discrete kernels satisfy the following monotonicity {\rm :}
\begin{equation*}
%\label{monweights}
0\le \varpi^{(\alpha,\beta)}_{k-1}\le  \varpi^{(\alpha,\beta)}_{k-2} \le \cdots\le  \varpi^{(\alpha,\beta)}_{1}
\le \varpi^{(\alpha,\beta)}_{0}=({{\Delta t}})^{\alpha\beta}E_{\alpha,\alpha\beta+1}^{\beta}\big(\!-({\Delta t})^{\alpha}\big),
\end{equation*}
and  $\varpi^{(\alpha,\beta)}_{0}$ is finite.
\end{lmm}
\begin{proof} Using  the integral mean-value theorem, we find from \eqref{HNweights} that
\begin{equation*}
%\label{fracrectw2}
\varpi^{(\alpha,\beta)}_{k-j}= \int_{t_{j-1}}^{t_j} e_{\alpha,\alpha\beta}^{\beta}\left(t_k-s;-1\right){\rm d}s={\Delta t}\, e_{\alpha,\alpha\beta}^{\beta}\left(t_k-\theta;-1\right), \quad \exists\,\theta \in (t_{j-1},t_j),\quad 1\le j\le k.
\end{equation*}
It is nonnegative and decreasing, since $e_{\alpha,\alpha\beta}^{\beta}\left(z;-1\right)$ is completely monotonic for $0<\alpha<1, 0<\beta \leq 1,$ and $z>0$ (see \eqref{CMML}). Therefore, we have the monotonicity of the discrete kernels and $\varpi^{(\alpha,\beta)}_{k-1} \geq 0~(1\leq k \leq N_t)$.
By virtue of \eqref{fracrectw1}, we have  $\varpi^{(\alpha,\beta)}_{0}=({\Delta t})^{\alpha\beta}E_{\alpha,\alpha\beta+1}^{\beta}(-(\Delta t)^{\alpha}),$ which is finite due to \eqref{BoundML}.
\end{proof}

\begin{rem} {\em In what follows, we shall not consider  the D-C model {\rm(}i.e., $\alpha=1${\rm)}.  In fact,  the computational codes for the ML functions in \cite{garrappa2015numerical}  excludes the case with $\alpha=1$.
In fact, the D-C model can be solved  by a very different method which we plan to report  in a separate future work.}
%The second reason is that the D-C model is a special case of H-N model, which can be solved effectively due to its own properties. We shall  report this in a separated future work.}
\end{rem}

Now we present the semi-discrete time-discretisation  scheme for the H-N model \eqref{scaledsystem}:
find ${\bm E}^k \in H_0({\rm curl},\Omega)$ and ${\bm H}^k, {\bm P}^k \in {\bm L}^2(\Omega)$  such that
\begin{subequations}\label{39eqn}
\begin{align}
& \epsilon_{\infty} \left(\delta_t {\bm E}^k, {\bm \phi}\right)+\left(\delta_t {\bm P}^k, {\bm \phi}\right)-\left({\bm H}^k, \nabla\times{\bm \phi}\right)=0, \quad  \forall {\bm \phi} \in H_0({\rm curl},\Omega),  \label{hntefd1} \\[4pt]
& \left(\delta_t {\bm H}^k, {\bm \psi}\right)+\left(\nabla\times{\bm E}^k, {\bm \psi}\right)=0, ~~~\quad \qquad \qquad \qquad \qquad~~ \forall {\bm \psi} \in {\bm L}^2(\Omega),  \label{hntefd2}\\
& \left({\bm P}^k, {\bm \varphi}\right)={\Delta \epsilon}\sum_{j=1}^{k}\varpi^{(\alpha,\beta)}_{k-j}\left({\bm E}^j,{\bm \varphi}\right), \qquad \qquad \qquad\;\; ~~ \forall {\bm \varphi} \in {\bm L}^2(\Omega), \label{hntefd3}
\end{align}
\end{subequations}
for $k=1,2,\cdots, N_t$, where ${\bm E}^0={\bm E}_0({\bm x})$, ${\bm H}^0={\bm H}_0({\bm x})$ and ${\bm P}^0=\bm 0$.
%\textcolor[rgb]{1.00,0.00,0.00}{
%\begin{eqnarray*}
%\nabla\cdot(\delta_t{\bm H}^k) =-\nabla\cdot(\nabla\times{\bm E}^k)=0
%\Rightarrow \nabla\cdot{\bm H}^k=\nabla\cdot{\bm H}^{k-1}
%\Rightarrow \nabla\cdot{\bm H}^k=0 ({\rm if}\;\nabla\cdot{\bm H}^0=0);
%\end{eqnarray*}
%\begin{eqnarray*}
%& &\nabla\cdot(\epsilon_{\infty} \delta_t{\bm E}^k +\delta_t{\bm P}^k) =\nabla\cdot(\nabla\times{\bm H}^k)=0
%\Rightarrow \nabla\cdot{\bm D}^k=\nabla\cdot{\bm D}^{k-1}
%\Rightarrow \nabla\cdot{\bm D}^k=0 ({\rm if}\;\nabla\cdot{\bm E}^0=0), \\
%& &\Rightarrow \epsilon_{\infty}\nabla\cdot{\bm E}^k+\nabla\cdot{\bm P}^k=0;
%\end{eqnarray*}
%\begin{eqnarray*}
%& &{\bm P}^1={\Delta \epsilon}\varpi^{(\alpha,\beta)}_{0}{\bm E}^1, \epsilon_{\infty}\nabla\cdot{\bm E}^1+\nabla\cdot{\bm P}^1=0 \Rightarrow \nabla\cdot{\bm E}^1=0\Rightarrow \nabla\cdot{\bm P}^1=0\\
%& &{\bm P}^2={\Delta \epsilon}(\varpi^{(\alpha,\beta)}_{0}{\bm E}^2+\varpi^{(\alpha,\beta)}_{0}{\bm E}^1), \epsilon_{\infty}\nabla\cdot{\bm E}^2+\nabla\cdot{\bm P}^2=0, \nabla\cdot{\bm E}^1=0 \Rightarrow \nabla\cdot{\bm E}^2=0\Rightarrow \nabla\cdot{\bm P}^2=0 \\
%& &\cdots
%\end{eqnarray*}
%}
\begin{rem}\label{TDisDivFree} {\em Following the derivation in Remark \ref{HNdivergence}, we can show $\nabla\times{\bm H^k}=0$ and
\begin{equation}\label{DisElecFlux}
\nabla\times(\epsilon_{\infty}{\bm E}^k+{\bm P}^k)=0,
\end{equation}
provided that the initial fields ${\bm E}^0$ and ${\bm H^0}$ are divergence free. With this, we can recursively prove that $\nabla\times{\bm E}^k=0,$ so is ${\bm P}^k.$ Indeed,  when $k=1$, substituting \eqref{hntefd3} into \eqref{DisElecFlux} and using Lemma \ref{lempositiveker}, we find $\nabla\times{\bm E}^1=\nabla\times{\bm P}^1=0$. Similarly, we are able to show the result for   $k=2$. Recursively, we deduce this property  for general $k\ge 3.$}
\end{rem}

%substituting \eqref{hntefd3} into \eqref{DisElecFlux}, noting Lemma \ref{lempositiveker} and combining with $\nabla\times{\bm E}^1=0$, we have $\nabla\times{\bm E}^2=0$  and then $\nabla\times{\bm P}^2=0$. Thus, we can proof that ${\bm E}^k$ and ${\bm P^k}$ are divergence free by \eqref{hntefd3}, \eqref{DisElecFlux} and Lemma \ref{lempositiveker} recursively. Here, we only consider $\varpi^{(\alpha,\beta)}_{k}>0$, it is similar otherwise. From the above, the time-discrete Gauss divergence laws hold true.}
%\end{rem}}
 %We add the non-homogeneous data ${\bm f} ^k, {\bm g} ^k$ and   ${\bm h} ^k$ for the convenience of convergence analysis.
%For simplicity, we denote
%where for $k=0,$  the summation  term vanishes.
\subsection{Stability  and discrete energy dissipation} \label{subsect32}
In the convergence analysis,  it is necessary to consider a more general setting:
\begin{subequations}\label{39eqn2}
\begin{align}
& \epsilon_{\infty}\! \left(\delta_t {\bm E}^k, {\bm \phi}\right)+\left(\delta_t {\bm P}^k, {\bm \phi}\right)-\left({\bm H}^k, \nabla\times{\bm \phi}\right)= ({\bm f} ^k,{\bm \phi}), \quad  \forall {\bm \phi} \in H_0({\rm curl},\Omega),  \label{zeng-hntefd1} \\[4pt]
& \left(\delta_t {\bm H}^k, {\bm \psi}\right)+\left(\nabla\times{\bm E}^k, {\bm \psi}\right)=
({\bm g} ^k, {\bm \psi}), ~~\;\qquad\;\;  \qquad \qquad\qquad~~~ \forall {\bm \psi} \in {\bm L}^2(\Omega),  \label{zeng-hntefd2}\\
& \left({\bm P}^k, {\bm \varphi}\right)={\Delta \epsilon}\sum_{j=1}^{k}\varpi^{(\alpha,\beta)}_{k-j}\left({\bm E}^j,{\bm \varphi}\right)
+ {\Delta \epsilon}\,({\bm h} ^k,{\bm \varphi}), \quad  \quad \quad\;    \forall {\bm \varphi} \in {\bm L}^2(\Omega),  \label{zeng-hntefd3}
\end{align}
\end{subequations}
where  ${\bm E}^0={\bm E}_0({\bm x})$, ${\bm H}^0={\bm H}_0({\bm x})$ and ${\bm P}^0=\bs 0$. We shall see  from  the error equations \eqref{325eqn}-\eqref{RkdefnA} for convergence analysis that  these non-homogeneous data will correspond to the
time-discretisation errors of the fields.
%%We first analyze the stability of the following discrete system:
%%given ${\bm E}^j(0\leq j \leq k-1)$, ${\bm H}^j(0\leq j \leq k-1)$,
%%and ${\bm P}^j(0\leq j \leq k-1)$, find
%% ${\bm E}^k \in H_0({\rm curl},\Omega) \cap {\bm L}^2(\Omega)$, ${\bm H}^k \in {\bm L}^2(\Omega)$ and ${\bm P}^k \in {\bm L}^2(\Omega)$ such that
%\begin{subequations}\label{39eqn2}
%\begin{align}
%& \epsilon_{\infty} \left(\delta_t {\bm E}^k, {\bm \phi}\right)+\left(\delta_t {\bm P}^k, {\bm \phi}\right)-\left({\bm H}^k, \nabla\times{\bm \phi}\right)= ({\bm f} ^k,\mathbf{\phi}), \quad  \forall {\bm \phi} \in H_0({\rm curl},\Omega),  \label{zeng-hntefd1} \\[4pt]
%& \left(\delta_t {\bm H}^k, {\bm \psi}\right)+\left(\nabla\times{\bm E}^k, {\bm \psi}\right)=
%({\bm g} ^k,\mathbf{\psi}), ~~\;\qquad\;\;  \qquad \qquad\qquad ~~ \forall {\bm \psi} \in {\bm L}^2(\Omega),  \label{zeng-hntefd2}\\
%& \left({\bm P}^k, {\bm \varphi}\right)={\Delta \epsilon}\sum_{j=1}^{k}\varpi^{(\alpha,\beta)}_{k-j}\left({\bm E}^j,{\bm \varphi}\right)
%+ ({\bm h} ^k,\mathbf{\varphi}), \quad  \quad \qquad\;\;\;   \forall {\bm \varphi} \in {\bm L}^2(\Omega),  \label{zeng-hntefd3}
%\end{align}
%\end{subequations}
%We have the following main result on the
\begin{thm}\label{zeng-thmSemiDisStab}
Let  ${\bm E}^k, {\bm P}^k, {\bm H}^k$ be the solutions of \eqref{39eqn}  or \eqref{39eqn2}, and define
\begin{equation}\label{zeng-defnEk}
 \mathscr E^k:=\epsilon_{\infty}\|{\bm E}^{k}\|^2+ \|{\bm H}^{k}\|^2+{\Delta \epsilon}\sum_{j=1}^{k}\varpi^{(\alpha,\beta)}_{k-j}\,\|{\bm E}^j\|^2,\;\; k\ge 1;\;\;\; \mathscr E^0:=\epsilon_{\infty}\|{\bm E}^{0}\|^2+ \|{\bm H}^{0}\|^2.
\end{equation}
Then the scheme \eqref{39eqn}  is unconditionally stable in the sense that  for all $\Delta t>0,$
\begin{equation}\label{defnEk}
\mathscr E^k \leq  \mathscr E^{k-1} \le\cdots\le  \mathscr E^1 \leq  \mathscr E^{0}.
\end{equation}
For the scheme \eqref{39eqn2} with nonhomogeneous data,  if $ \varrho \Delta t < 1$ for some  given positive constant $\varrho>0,$  and
 \begin{equation}\label{fghcondition}
\bm Q^k:= {\Delta t}\sum_{j=1}^k\big(\|{\bm f} ^j\|^2
+   (\Delta \epsilon)^2 \|\delta_t {\bm h} ^j \|^2+ \|{\bm g} ^j\|^2\big)<\infty,
 \end{equation}
% \eqref{zeng-hntefd1}--\eqref{zeng-hntefd3}.
 then we have
%there exists a positive constant $C$ independent of $\Delta t,k$ and any function
%such that
\begin{equation}\label{zeng-disstabthm1}
\mathscr E^k
\le   \frac{1}{1-\varrho \Delta t} {\rm exp}\Big( \frac{\varrho\, t_{k-1}}{1-\varrho \Delta t}\Big)\,\Big(\mathscr E^{0} +  \frac{1}{\varrho} \bm Q^k \Big),\quad k\ge 1.
\end{equation}
 %where $\mathscr E^0:=\epsilon_{\infty}\|{\bm E}^{0}\|^2+ \|{\bm H}^{0}\|^2$ and
\end{thm}
\begin{proof} We first prove \eqref{zeng-disstabthm1}, and then \eqref{defnEk} follows straightforwardly.

\smallskip
Taking  ${\bm \phi}={{\Delta t}} {\bm E}^k$ in \eqref{zeng-hntefd1} and $\bm\psi={{\Delta t}} {\bm H}^k$ in \eqref{zeng-hntefd2}, respectively,  and  adding two resulted equations together, we  obtain
\begin{equation}
\label{zeng-disstab1}
\begin{split}
& \epsilon_{\infty} \left({\bm E}^k-{\bm E}^{k-1},{\bm E}^k\right)+\left({\bm H}^k-{\bm H}^{k-1},{\bm H}^k\right)+\left({\bm P}^k-{\bm P}^{k-1}, {\bm E}^k\right)\\
&\qquad \qquad ={\Delta t}({\bm f} ^k,{\bm E}^k) + {\Delta t}({\bm g} ^k,{\bm H}^k).
\end{split}
\end{equation}
We eliminate $\bm P$ from the above identity  by using \eqref{zeng-hntefd3}
with ${\bm \varphi}=\bm E^k$, so we can  rewrite \eqref{zeng-disstab1}  as
\begin{equation}\label{disstab2}
\begin{split}
 \epsilon_{\infty} \left({\bm E}^k,{\bm E}^k\right)  & +  \left({\bm H}^k,{\bm H}^k\right)  +{\Delta \epsilon}\sum_{j=1}^{k}\varpi^{(\alpha,\beta)}_{k-j}\left({\bm E}^j, {\bm E}^k\right)  \\
 &= \epsilon_{\infty}\left({\bm E}^{k-1},{\bm E}^k\right)+\left({\bm H}^{k-1},{\bm H}^k\right)+{\Delta \epsilon}\sum_{j=1}^{k-1}\varpi^{(\alpha,\beta)}_{k-1-j}\left({\bm E}^j, {\bm E}^k\right) \\[2pt]
&\quad +
{\Delta t}\,({\bm f} ^k  + {\Delta \epsilon} \delta_t  {\bm h} ^k ,{\bm E}^k) + {\Delta t}\,({\bm g} ^k,{\bm H}^k).
\end{split}
\end{equation}
Rearranging \eqref{disstab2}  yields
\begin{equation} \label{disstab3}
\begin{split}
\big(\epsilon_{\infty} &+{\Delta \epsilon}\, \varpi^{(\alpha,\beta)}_{0}\big) \|{\bm E}^k\|^2 + \|{\bm H}^k\|^2\\
& =\epsilon_{\infty}\left({\bm E}^{k-1},{\bm E}^k\right)+\left({\bm H}^{k-1},{\bm H}^k\right) +{\Delta \epsilon}\sum_{j=1}^{k-1}\big(\varpi^{(\alpha,\beta)}_{k-1-j}
-\varpi^{(\alpha,\beta)}_{k-j}\big)\left({\bm E}^j, {\bm E}^k\right)\\
&\quad  +{\Delta t}\,({\bm f} ^k + {\Delta \epsilon} \delta_t  {\bm h} ^k ,{\bm E}^k)
+ {\Delta t}\,({\bm g} ^k,{\bm H}^k) .
\end{split}
\end{equation}
For $k\ge 2$, using  the Cauchy--Schwarz inequality and Lemma \ref{lemdisker},  we obtain
\begin{equation}\label{zeng-kstep-0}
\begin{split}
 \big(\epsilon_{\infty}+{\Delta \epsilon}\,\varpi^{(\alpha,\beta)}_{0}\big)\|{\bm E}^{k}\|^2
  + \|{\bm H}^{k}\|^2 & \leq \frac{\epsilon_{\infty}} 2 (\|{\bm E}^k\|^2+ \|{\bm E}^{k-1}\|^2)+\frac{1}{2}(\|{\bm H}^k\|^2+ \|{\bm H}^{k-1}\|^2)  \\
&    + \frac{\Delta \epsilon} 2 \sum_{j=1}^{k-1}\big(\varpi^{(\alpha,\beta)}_{k-1-j}-\varpi^{(\alpha,\beta)}_{k-j}\big)
\big(\|{\bm E}^k\|^2+\|{\bm E}^j\|^2\big)\\
&+ \frac{ \varrho \Delta t}{2 } (\|{\bm E}^k\|^2 + \|{\bm H}^k\|^2)+ \frac{ {\Delta t}}{2 \varrho}\big(\|{\bm f} ^k + {\Delta \epsilon} \delta_t  {\bm h} ^k \|^2
+ \|{\bm g} ^k\|^2\big),
\end{split}
\end{equation}
where $\varrho>0$ is a constant independent of ${\Delta t}$ and $k$. It is evident that
\begin{equation}\label{omegak0}
 \sum_{j=1}^{k-1}\big(\varpi^{(\alpha,\beta)}_{k-1-j}-\varpi^{(\alpha,\beta)}_{k-j}\big)
\|{\bm E}^k\|^2= \big(\varpi^{(\alpha,\beta)}_{0}- \varpi^{(\alpha,\beta)}_{k-1}\big) \|{\bm E}^k\|^2.
\end{equation}
Consequently, by \eqref{zeng-defnEk}, we find from \eqref{zeng-kstep-0}-\eqref{omegak0} immediately that
\begin{equation}\label{zeng-kstep-010}
\begin{split}
 {\mathscr E^k} -  \varrho {\Delta t}(\|{\bm E}^k\|^2 + \|{\bm H}^k\|^2) & \leq {\mathscr E^{k-1}}
+ \frac {\Delta t} {\varrho} (\|{\bm f} ^k + {\Delta \epsilon} \delta_t  {\bm h} ^k \|^2
+ \|{\bm g} ^k\|^2),
\end{split}
\end{equation}
which implies
\begin{equation}\label{zeng-kstepB}
(1- \varrho {\Delta t})\mathscr E^k
\le \mathscr E^{k-1} + \frac{\Delta t}{\varrho}  \big(\|{\bm f} ^k \|^2
+ (\Delta \epsilon)^2 \|\delta_t {\bm h} ^k \|^2+ \|{\bm g} ^k\|^2\big).
\end{equation}
In fact, \eqref{zeng-kstepB} also holds for $k=1.$ Indeed,    by \eqref{zeng-disstab1}  with $k=1$,
and understanding the summation $\sum_{n=1}^0=0$ in  \eqref{disstab2}-\eqref{omegak0},
we can get \eqref{zeng-kstepB} with $k=1$ readily.

For clarity, we set $k=j$ and rewrite \eqref{zeng-kstepB} as
\begin{equation}\label{zeng-kstep}
(1-\varrho {\Delta t})(\mathscr E^j- \mathscr E^{j-1})
\le  {\varrho}  {\Delta t} \, \mathscr E^{j-1} +\frac{\Delta t}{\varrho} \big(\|{\bm f} ^j \|^2
+  (\Delta \epsilon)^2 \|\delta_t {\bm h}^j \|^2+ \|{\bm g}^j\|^2\big),\quad j\ge 1.
\end{equation}
Summing it up for $1\le j\le k,$ we have that if $\varrho {\Delta t}<1,$
\begin{equation}\label{zeng-kstep-20}
\begin{split}
\mathscr E^k
\le \frac 1 {1-\varrho {\Delta t}}\bigg\{\varrho {\Delta t} \sum_{j=1}^{k-1} \mathscr E^j +
\mathscr E^{0} + \frac{\Delta t}{\varrho}\sum_{j=1}^k\big(\|{\bm f} ^j \|^2
+  (\Delta \epsilon)^2 \|\delta_t {\bm h} ^j \|^2+ \|{\bm g} ^j\|^2\big)\bigg\}.
\end{split}
\end{equation}
Using the discrete Gr\"{o}nwall's inequality (see, e.g., \cite[Lemma 1.4.2]{Quarteroni1994}), we derive
\eqref{zeng-disstabthm1}  directly.
%
%%Suppose that $1-{\Delta t} \ge C_0$, then for ${\Delta t} \le 1- C_0$, we have
%%\begin{equation}
%%\label{zeng-kstep-2}
%%\begin{split}
%%\mathscr E^n
%%\le \frac{1}{1-C_0}\left[{\Delta t} \sum_{k=1}^{n-1} \mathscr E^k +
%%\mathscr E^{0} + {\Delta t}\sum_{k=1}^n\big(\|{\bm f} ^k \|^2
%%+  (\Delta \epsilon)^2 \|\delta_t {\bm h} ^k \|^2+ \|{\bm g} ^k\|^2\big)\right].
%%\end{split}
%%\end{equation}
%Recall the discrete Gr\"{o}nwall's inequality \cite[Lemma 1.4.2]{Quarteroni1994}, i.e., if a sequence $\phi_n$ satisfies
%\begin{equation*}
%\left\{
%\begin{aligned}
%\phi_0 &= \chi_0,  \\
%\phi_n &= \chi_0+\sum_{s=0}^{n-1}p_s+\sum_{s=0}^{n-1}\kappa_s\phi_s,~~n \ge 1, \\
%\end{aligned}
%\right.
%\end{equation*}
%then $\phi_n$ satisfies
%\begin{equation*}
%\phi_n \le e^{\sum_{s=0}^{n-1}\kappa_s}\left(\chi_0+\sum_{s=0}^{n-1}p_s\right), ~~n \ge 1,
%\end{equation*}
%where $\chi_0 \ge 0, p_n \ge 0$ and $\kappa_n$ is a non-negative sequence for $n \ge 0$, then \eqref{zeng-kstep-2} yields \eqref{zeng-disstabthm1} for $n\geq 1$.
%%If $n=1$, then \eqref{zeng-kstep-2} still holds due to  $\sum_{n=1}^0=0$,
%%which leads to \eqref{zeng-disstabthm1} for $k=1$.
%

We now turn to \eqref{39eqn}.  In this case, we have $\bs f^k= \bs g^k=\bs h^k=\bs 0 $ in  \eqref{zeng-disstab1}. Following the same lines as in the above derivations, we find readily that \eqref{zeng-kstep-010} becomes  $\mathscr E^k \leq  \mathscr E^{k-1},$ so the decay of the discrete energy  in \eqref{defnEk} holds. The proof is completed.
\end{proof}
\begin{rem}\label{exp-const} {\em We can represent the constant in the bound  \eqref{zeng-disstabthm1} more explicitly so that it does not depend on $\Delta t.$ For example,  we take $\varrho=1$ and assume that $1-\Delta t\ge c_*>0,$ i.e.,
$\Delta t<1- c^*.$ Then
\begin{equation}\label{zeng-disstabthm100}
\mathscr E^k
\le   \frac{1}{c_*} {\rm exp}\Big( \frac{t_{k-1}}{c_*}\Big)\,\big(\mathscr E^{0} +   \bm Q^k \big),\quad k\ge 1.
\end{equation}
  }
\end{rem}

%%%%%%%%%%%%%%%%%%%%%%%%%%%%%%%%%%%%%%%%%%%%%%%%%%%%%%%%%%%%%%%%%%%%%%%%%%%%%%%%%%%%%
%%%%%%%%%%%%%%%%%%%%%%%%%%%%%%%%% Theorem SemiDisStab %%%%%%%%%%%%%%%%%%%%%%%%%%%%%%%%%%%
%%%%%%%%%%%%%%%%%%%%%%%%%%%%%%%%%%%%%%%%%%%%%%%%%%%%%%%%%%%%%%%%%%%%%%%%%%%%%%%%%%%%%
%The stability result for  \eqref{hntefd1}-\eqref{hntefd3}  is shown in the following theorem.
With the aid of Theorem \ref{zeng-thmSemiDisStab}, we can further derive the following bound for ${\bm P}^k$ in \eqref{hntefd3}.
\begin{corollary}\label{corSemiDisStab}
Let  ${\bm E}^k, {\bm P}^k, {\bm H}^k$ be the solution of \eqref{39eqn}. Then we have  %\eqref{hntefd1}--\eqref{hntefd3}.
% Then
%\begin{equation}\label{defnEk}
%\mathscr E^k \leq  \mathscr E^{k-1} ,\quad k\ge 1.
%\end{equation}
%Moreover,
\begin{equation}\label{disstabthm3}
\|{\bm P}^k\| \leq \big({\Delta \epsilon}\, e_{\alpha,\alpha\beta+1}^{\beta}(t_k;-1)\big)\,
\max_{1\leq j \leq k}\|{\bm E}^j\|\le \big({\Delta \epsilon}\, e_{\alpha,\alpha\beta+1}^{\beta}(t_k;-1)\big)\sqrt{{\mathscr E}^0},
\end{equation}
for $k\ge 1,$ where $e_{\alpha,\alpha\beta+1}^{\beta}(t_k;-1)$ defined in  \eqref{gml}  is finite for $0<t_k\le T.$
\end{corollary}
\begin{proof}
%Let ${\bm f} ={\bm g} ={\bm h} =0$ in \eqref{zeng-hntefd1}--\eqref{zeng-hntefd3}, then the last two terms of \eqref{zeng-kstep-0} vanish. Thus, by repeating the proof of Theorem \ref{zeng-thmSemiDisStab}, we can get \eqref{defnEk}.
From \eqref{hntefd3}, we obtain
\begin{equation*}
%\label{disstab10}
\|{\bm P}^k\|^2=\left({\bm P}^k, {\bm P}^k\right)={\Delta \epsilon}\sum_{j=1}^{k}\varpi^{(\alpha,\beta)}_{k-j}\left({\bm E}^j,{\bm P}^k\right)\leq {\Delta \epsilon}\sum_{j=1}^{k}\varpi^{(\alpha,\beta)}_{k-j}\,\|{\bm E}^j\|\,\|{\bm P}^k\|. % \quad k \in (0,N_t],
\end{equation*}
Thus, one has
\begin{equation}
\label{disstab10}
\begin{split}
\|{\bm P}^k\|
& \leq {\Delta \epsilon}\sum_{j=1}^{k}\varpi^{(\alpha,\beta)}_{k-j}\,\|{\bm E}^j\|
\leq {\Delta \epsilon}\,\Big(\max_{1\leq j \leq k}\|{\bm E}^j\|\Big)\Big(\sum_{j=1}^{k}\varpi^{(\alpha,\beta)}_{k-j}\Big)\\
&\leq {\Delta \epsilon}\,t_k^{\alpha\beta}E_{\alpha,\alpha\beta+1}^{\beta}(-t_k^{\alpha})
\,\Big(\max_{1\leq j \leq k}\|{\bm E}^j\|\Big),
\end{split}
\end{equation}
where we used  \eqref{fracrectw1} and \eqref{BoundML} to arrive at
\begin{equation}\label{bndS}
\begin{split}
\sum_{j=1}^{k}\varpi^{(\alpha,\beta)}_{k-j}
& = e_{\alpha,\alpha\beta+1}^{\beta}\left(t_k;-1\right)-
e_{\alpha,\alpha\beta+1}^{\beta}\left(0;-1\right)
 =  t_k^{\alpha\beta}E_{\alpha,\alpha\beta+1}^{\beta}(-t_k^{\alpha})= e_{\alpha,\alpha\beta+1}^{\beta}(t_k;-1).
\end{split}
\end{equation}
Then we obtain the second inequality in \eqref{disstabthm3} from \eqref{defnEk} immediately.
The proof is completed.
\end{proof}

\begin{rem}\label{Li-case} {\em  For the C-C model {\rm(}i.e., $\beta=1${\rm)}, the energy dissipation  was proved  by Li et al. \cite{li2011developing},  where the  equation of the induced polarization and electric field was formulated as the Caputo fractional differential form   {\rm (}see \eqref{caputodef0} with $\beta=1${\rm).} However, it appeared nontrivial to show the strong energy dissipation and stability similar to  \eqref{defnEk}, as the bound therein contained a constant $C>1$  between  $k$th and $(k-1)$th steps {\rm(}see  \cite[Theorem 3.1]{li2011developing}{\rm)}.   Though our result does not resolve this deficiency,  as we work with the integral formulation of  the induced polarization and electric field, and the discretisation schemes  are different,  we believe our argument can shed light on the analysis of the scheme based on the fractional differential form.
%their discrete energy dissipation law didn't have the exact form as the continuous one. If we denote the discrete energy in \cite{li2011developing} (see (3.13)) at $t_k=k{\Delta t}$ as $\mathcal{E}^k_h$, then by the proof of the Theorem 3.1 in \cite{li2011developing}, one only can obtain
%$$\mathcal{E}^1_h\leq C_1\mathcal{E}^0_h, \mathcal{E}^2_h\leq C_2\mathcal{E}^2_h,\cdots, \mathcal{E}^k_h\leq C_k\mathcal{E}^{k-1}_h, \cdots,$$ which leads to $C=C_1C_2\cdots C_k$ in the Theorem 3.1 of \cite{li2011developing}. Therefore, $C=1$ which means their discrete energy dissipation law has the exact form as the continuous one is not reasonable. How to remove this constant is not easy task due to the existence of the fractional derivative operator.
\qed}
\end{rem}
%\begin{rem}Our discrete energy dissipation law, based on approximating the polarization via a time convolution of the electric field, has the exact form as the continuous one \eqref{constabthm2} obtained in Theorem \ref{thmConStab} and is valid for the H-N model including the C-C model, though the entire scheme is first-order accurate (see Theorem \ref{thmsemicon}). Moreover, a monotonic energy decay \eqref{defnEk} is also derived. One can get second order schemes by using piecewise linear approximation \eqref{fratrap} for the induced polarization field and treating other terms as did in \cite{li2011developing}. But how to get the corresponding energy dissipation results is much involved.
%\end{rem}
%\begin{rem} For the H-N model excluding the C-C model, developing a discrete scheme based on the ADE between the induced polarization and electric fields is very complicate due to a fractional pseudo-differential derivative operator (see Definition \ref{PrabhakarOper}) is involved in the ADE.
%\end{rem}
%

\subsection{Convergence analysis} \label{subsect33}
Next, we carry out the convergence analysis of the semi-discrete scheme \eqref{39eqn}. Denote ${\bm \varepsilon}_1^k={\bm E}^k-{\bm E}(t_k)$, ${\bm \varepsilon}_2^k={\bm H}^k-{\bm H}(t_k)$, and ${\bm \varepsilon}_3^k={\bm P}^k-{\bm P}(t_k)$. Then we can derive the following error equations from  subtracting \eqref{31eqn} from \eqref{39eqn}:
\begin{subequations}\label{325eqn}
\begin{align}
& \epsilon_{\infty} \left(\delta_t {\bm \varepsilon}_1^k, {\bm \phi}\right)+\left(\delta_t {\bm \varepsilon}_3^k, {\bm \phi}\right)-\left({\bm \varepsilon}_2^k, \nabla\times{\bm \phi}\right)=({\bm R}^k_1,{\bm \phi}), \quad\;  \forall {\bm \phi} \in H_0({\rm curl},\Omega),  \label{hntefderr1} \\[4pt]
& \left(\delta_t {\bm \varepsilon}_2^k, {\bm \psi}\right)+\left(\nabla\times{\bm \varepsilon}_1^k, {\bm \psi}\right)=({\bm R}^k_2,{\bm \phi}), ~~~\;\qquad \qquad \qquad~~\qquad\;\;\;  \forall {\bm \psi} \in {\bm L}^2(\Omega),  \label{hntefderr2}\\
& \left({\bm \varepsilon}_3^k, {\bm \varphi}\right)={\Delta \epsilon}\sum_{j=1}^{k}\varpi^{(\alpha,\beta)}_{k-j}\left({\bm \varepsilon}_1^j,{\bm \varphi}\right)-{\Delta \epsilon}({\bm R}^k_0,{\bm \varphi}), \quad \quad\quad \; \forall {\bm \varphi} \in {\bm L}^2(\Omega), \label{hntefderr3}
\end{align}
\end{subequations}
for $k=1,2,\cdots, N_t$, where ${\bm R}^k_0$ is defined in \eqref{TruncErr0} and
\begin{equation}\label{RkdefnA}
\begin{split}
& {\bm R}^k_1(\bm x)=\epsilon_{\infty}\left(\partial_t {\bm E}(\bm x, t_k)-\delta_t {\bm E}(\bm x, t_k)\right)+\partial_t {\bm P}(\bm x, t_k)-\delta_t {\bm P}(\bm x, t_k),\\
& {\bm R}^k_2(\bm x)=\partial_t {\bm H}(\bm x, t_k)-\delta_t {\bm H}(\bm x, t_k).
\end{split}
\end{equation}

%{\color{blue} The regularity of $P$ depends on $E$. We may delete the
%regularity assumption on $P$ in lemmas and theorems}
Now, we can present the following convergence result for the semi-discrete scheme. Before the proof, we first give the following lemma.
\begin{lmm}\label{lem3-3} If ${\bm E}, {\bm P}, {\bm H} \in H^2(0,T;{\bm L}^2(\Omega))$
and $k\geq 1$, then we have
\begin{equation}\label{rmfirstest}
\|{\bm R}^k_1 \| \leq C{\Delta t} \left(\|{\bm E}\|_{H^2(0,T;{\bm L}^2(\Omega))}
+\|{\bm P}\|_{H^2(0,T;{\bm L}^2(\Omega))}\right); \quad \|{\bm R}^k_2 \| \leq C{\Delta t} \|{\bm H}\|_{H^2(0,T;{\bm L}^2(\Omega))},
\end{equation}
and
\begin{equation}\label{rm2ndest}
\|{\bm R}^k_0 \| \leq C{\Delta t} \|{\bm E}\|_{H^1(0,T;{\bm L}^2(\Omega))},\quad \|{\delta_t}{\bm R}^k_0 \|\leq C{\Delta t}\|{\bm E}\|_{H^2(0,T;{\bm L}^2(\Omega))},
\end{equation}
where $C$ is a generic positive constant independent of $\Delta t$ and any field but depending on $t_{k-1}$.
\end{lmm}
\begin{proof} From  \eqref{RkdefnA}, we can obtain readily the estimates  \eqref{rmfirstest} using   the standard finite difference analysis for $1\le k\le N_t$.

We now derive \eqref{rm2ndest}. Denote ${\bm r}^k={\bm E}(\bm x, t)-{\bm E}(\bm x, t_k), t \in (t_{k-1},t_k).$ Then by \eqref{TruncErr0}, \eqref{IntgML} and the completely monotonicity of $e_{\alpha,\alpha\beta}^{\beta}\left(z;-(\Delta t)^{\alpha}\right) (0<\alpha<1, 0<\beta \leq 1, z>0)$ (see \eqref{CMML}), we have
\begin{equation*}
\|{\bm R}^k_0\|\leq {\Delta t}\|{\bm E}\|_{H^1((0,T);{\bm L}^2(\Omega))}\int_{0}^{t_j} e_{\alpha,\alpha\beta}^{\beta}(t_j-s;-1)\,{\rm d}s={\Delta t}\,e_{\alpha,\alpha\beta+1}^{\beta}(t_j;-1)\|{\bm E}\|_{H^1(0,T;{\bm L}^2(\Omega))},
\end{equation*}
and
\begin{equation*}
\label{DiffR0bound}
\begin{split}
& \|{\delta_t}{\bm R}^k_0\|=\frac{1}{\Delta t}\Big\|\sum_{j=1}^{k}\int_{t_{j-1}}^{t_j} e_{\alpha,\alpha\beta}^{\beta}(t_k-s;-1)\,{\bm r}^j(s){\rm d}s-\sum_{j=1}^{k-1}\int_{t_{j-1}}^{t_j} e_{\alpha,\alpha\beta}^{\beta}(t_{k-1}-s;-1)\,{\bm r}^j(s){\rm d}s\Big\|\\
& \le \frac{1}{\Delta t}\Big\|\sum_{j=1}^{k-1}\int_{t_{j-1}}^{t_j} e_{\alpha,\alpha\beta}^{\beta}(t_{k-1}-s;-1)\,\left[{\bm r}^j(s+\Delta t)-{\bm r}^j(s)\right]{\rm d}s\Big\|+\frac{1}{\Delta t}\Big\|\int_{0}^{t_1} e_{\alpha,\alpha\beta}^{\beta}(t_{k}-s;-1)\,{\bm r}^1(s){\rm d}s\Big\|\\
& \le \frac{1}{\Delta t}\bigg(C({\Delta t})^2\|{\bm E}\|_{H^2(0,T;{\bm L}^2(\Omega))}\int_{0}^{t_{k-1}} e_{\alpha,\alpha\beta}^{\beta}(t_{k-1}-s;-1){\rm d}s\\
&\quad +C{\Delta t}\|{\bm E}\|_{H^1(0,T;{\bm L}^2(\Omega))}\int_{0}^{\Delta t} e_{\alpha,\alpha\beta}^{\beta}(t_{k}-s;-1)\,{\rm d}s\bigg)\\
& \le C{\Delta t}\|{\bm E}\|_{H^2(0,T;{\bm L}^2(\Omega))}\int_{0}^{t_{k-1}} e_{\alpha,\alpha\beta}^{\beta}(t_{k-1}-s;-1){\rm d}s+C e_{\alpha,\alpha\beta}^{\beta}(t_{k-1};-1)\|{\bm E}\|_{H^1(0,T;{\bm L}^2(\Omega))}\int_{0}^{\Delta t} {\rm d}s\\
& \le C{\Delta t}\big(e_{\alpha,\alpha\beta+1}^{\beta}(t_{k-1};-1)\|{\bm E}\|_{H^2(0,T;{\bm L}^2(\Omega))}  +e_{\alpha,\alpha\beta}^{\beta}(t_{k-1};-1)\|{\bm E}\|_{H^1(0,T;{\bm L}^2(\Omega))}\big).
\end{split}
\end{equation*}
Then by \eqref{BoundML}, the proof is completed.
\end{proof}

%%%%%%%%%%%%%%%%%%%%%%%%%%%%%%%%%%%%%%%%%%%%%%%%%%%%%%%%%%%%%%%%%%%%%%%%%%%%%%%%%%%%%
%%%%%%%%%%%%%%%%%%%%%%%%%%%%%%%%% Theorem SemiCon %%%%%%%%%%%%%%%%%%%%%%%%%%%%%%%%%%%
%%%%%%%%%%%%%%%%%%%%%%%%%%%%%%%%%%%%%%%%%%%%%%%%%%%%%%%%%%%%%%%%%%%%%%%%%%%%%%%%%%%%%
%{\color{blue} The regularity of $P$ depends on $E$. We may delete the
%regularity assumption on $P$}

In light of Theorem \ref{zeng-thmSemiDisStab}  and Lemma \ref{lem3-3}, we obtain the following convergence result on the semi-discrete scheme \eqref{39eqn}.
\begin{thm}\label{thmsemicon} Let ${\bm E}({\bm x},t), {\bm P}({\bm x},t), {\bm H}({\bm x},t)$ be the  solution of \eqref{scaledsystem}, and let ${\bm E}^k, {\bm P}^k, {\bm H}^k$ be the solution of \eqref{39eqn}. Assume that
$${\bm E}\in H^2\big(0,T;  H_0({\rm curl},\Omega) \cap {\bm L}^2(\Omega)\big) ~~ {\rm and}~~ {\bm P}, {\bm H} \in H^2\big(0,T;{\bm L}^2(\Omega)\big).
$$
Then for $\Delta t<1- c^*$ for given constant $c^*\in (0,1)$ as in Remark \ref{exp-const},   we have the error estimate
\begin{equation}\label{SemiConResult}
\begin{split}
& \|{\bm E}^k(\cdot)-{\bm E}(\cdot,t_k)\|^2+\|{\bm H}^k(\cdot)-{\bm H}(\cdot,t_k)\|^2
+\|{\bm P}^k(\cdot)-{\bm P}(\cdot,t_k)\|^2 \\
&\quad \leq  C(\Delta t)^2 \big(\|{\bm E}\|^2_{H^2(0,T;{\bm L}^2(\Omega))}
+\|{\bm P}\|^2_{H^2(0,T;{\bm L}^2(\Omega))}+\|{\bm H}\|^2_{H^2(0,T; {\bm L}^2(\Omega))} \big),
\end{split}
\end{equation}
where the constant $C$ is inherited from Remark \ref{exp-const}, Corollary \ref{corSemiDisStab} and Lemma \ref{lem3-3}.
\end{thm}
\begin{proof}
Taking ${\bm \phi}={\Delta t}\,{\bm \varepsilon}_{1}^k$ in \eqref{hntefderr1}, ${\bm \psi}={\Delta t} {\bm \varepsilon}_{2}^k$ in \eqref{hntefderr2}, and ${\bm \varphi}={\bm \varepsilon}_{3}^k$ in \eqref{hntefderr3}, and following  the derivation of \eqref{disstab3}, we obtain
\begin{equation*}
%\label{erreqution}
\begin{split}
\big(\epsilon_{\infty} &+{\Delta \epsilon}\, \varpi^{(\alpha,\beta)}_{0}\big) \|{\bm \varepsilon}_{1}^k\|^2 + \|{\bm \varepsilon}_{2}^k\|^2\\
& =\epsilon_{\infty}\left({\bm \varepsilon}_{1}^{k-1},{\bm \varepsilon}_{1}^k\right)+\left({\bm \varepsilon}_{2}^{k-1},{\bm \varepsilon}_{2}^k\right) +{\Delta \epsilon}\sum_{j=1}^{k-1}\big(\varpi^{(\alpha,\beta)}_{k-1-j}
-\varpi^{(\alpha,\beta)}_{k-j}\big)\big({\bm \varepsilon}_{1}^j, {\bm \varepsilon}_{1}^k\big)\\
&\quad  +{\Delta t}\,({\bm R}^k_1 + {\Delta \epsilon} \delta_t  {\bm R}^k_0 ,{\bm \varepsilon}_{1}^k)
+ {\Delta t}\,({\bm R}^k_2,{\bm \varepsilon}_{2}^k) .
\end{split}
\end{equation*}
Then we can derive from Theorem \ref{zeng-thmSemiDisStab} and Remark \ref{exp-const}  that
\begin{equation}\label{wbnd1}
\begin{split}
\|{\bm \varepsilon}_{1}^k\|^2+\|{\bm \varepsilon}_{2}^k\|^2
\le   C  {\Delta t}\sum_{j=1}^k (\|{\bm R}_1^j\|^2
+ (\Delta \epsilon)^2  \|\delta_t {\bm R}_0^j \|^2+ \|{\bm R}_2^j\|^2),
\end{split}
\end{equation}
where we used  the facts $\epsilon_{\infty} \ge 1$, ${\Delta \epsilon}\sum_{j=1}^{k}\varpi^{(\alpha,\beta)}_{k-j}\,\|{\bm \varepsilon}_{1}^j\|^2 \ge 0,$ and ${\bm \varepsilon}_{1}^0={\bm \varepsilon}_{2}^0=\bs 0$.

Similar to the proof of Corollary \ref{corSemiDisStab}, we find  from \eqref{hntefderr3} that
\begin{equation}\label{wbnd2}
\begin{split}
\|{\bm \varepsilon}_3^k\| & \le{\Delta \epsilon}\sum_{j=1}^{k}\varpi^{(\alpha,\beta)}_{k-j}\|{\bm \varepsilon}_1^j\|+{\Delta \epsilon}\|{\bm R}^k_0\|%\\&
\le \big({\Delta \epsilon}\, e_{\alpha,\alpha\beta+1}^{\beta}(t_k;-1)\big)\, \max_{1\leq j \leq k}\|{\bm \varepsilon}_{1}^j\| +{\Delta \epsilon}\,\|{\bm R}^k_0\|.
\end{split}
\end{equation}
Finally, applying Lemma \ref{lem3-3} to \eqref{wbnd1}-\eqref{wbnd2}, we arrive at the error estimate \eqref{SemiConResult}.
\end{proof}

\begin{rem}\label{2ndscheme}  {\em   In principle, we can upgrade the first-order  temporal scheme \eqref{39eqn}  to  a  second-order scheme. The essential component is to apply  the piecewise linear approximation to  \eqref{rehnmodelp} that yields
\begin{equation*}\label{fratrap}
\begin{split}
{\bm P}(\bs x, t_k)  & \approx  {\Delta \epsilon} \sum_{j=1}^{k}\int_{t_{j-1}}^{t_j} e_{\alpha,\alpha\beta}^{\beta}(t_k-s;-1)\bigg\{\frac{t_j-s}{t_j-t_{j-1}}{\bm E}^{j-1}(\bs x)+\frac{s-t_{j-1}}{t_j-t_{j-1}}{\bm E}^{j}(\bs x)\bigg\}\,{\rm d}s
\\ & := {\Delta \epsilon}\sum_{j=0}^{k}\rho^{(\alpha,\beta)}_{k-j}{\bm E}^j(\bs x),  \quad k \geq 1,
\end{split}
\end{equation*}
where  the weights  can be computed by
$$
\rho^{(\alpha,\beta)}_{k-j}=\int_{t_{j-1}}^{t_j} e_{\alpha,\alpha\beta}^{\beta}\left(t_k-s;-1\right)\frac{s-t_{j-1}}{t_j-t_{j-1}}{\rm d}s+
\int_{t_{j}}^{t_{j+1}} e_{\alpha,\alpha\beta}^{\beta}\left(t_k-s;-1\right)\frac{t_{j+1}-s}{t_{j+1}-t_{j}}{\rm d}s,
$$
for $1\le j\le k-1,$ and $\big\{\rho^{(\alpha,\beta)}_{0}, \rho^{(\alpha,\beta)}_{k}\big\}$ have similar expressions. It is seen  that the monotonicity of the  weights in Lemma \ref{lemdisker} has played a critical role in the analysis.  However, it is still unknown if
$\{\rho^{(\alpha,\beta)}_{k-j}\}$ enjoys the same property.  In fact, we have observed from some numerical evidences that it is parametric dependent as this property is not true for all  $\alpha,\beta.$}
%%{\footnotesize \begin{eqnarray*}
%%%\label{fratrapw1}
%%\varpi^{(\alpha,\beta)}_{k-j}=
%%\left\{\begin{array}{lll}
%%\int_{t_{0}}^{t_1} e_{\alpha,\alpha\beta}^{\beta}\left(t_k-s;-1\right)\frac{t_1-s}{t_1-t_{0}}{\rm d}s, \qquad j=0,\\ [0.3cm]
%%\int_{t_{j-1}}^{t_j} e_{\alpha,\alpha\beta}^{\beta}\left(t_k-s;-1\right)\frac{s-t_{j-1}}{t_j-t_{j-1}}{\rm d}s+
%%\int_{t_{j}}^{t_{j+1}} e_{\alpha,\alpha\beta}^{\beta}\left(t_k-s;-1\right)\frac{t_{j+1}-s}{t_{j+1}-t_{j}}{\rm d}s, \quad 1\leq j \leq k-1,\\ [0.3cm]
%%\int_{t_{k-1}}^{t_k} e_{\alpha,\alpha\beta}^{\beta}\left(t_k-s;-1\right)\frac{s-t_{k-1}}{t_k-t_{k-1}}{\rm d}s, \qquad j=k.
%%\end{array}\right.
%%\end{eqnarray*}}
%It is evident that unlike \eqref{fracrectw1},  they take a more complicated form.   whether or not the discrete kernels would be monotonic is unknown. }
\end{rem}

\section{Implementation and numerical results}\label{Sect4}
\setcounter{equation}{0}
\setcounter{lmm}{0}
\setcounter{thm}{0}

In this section, we follow the idea of \cite{Lubich2002fast,zeng2018stable} to introduce a fast temporal convolution algorithm
that can alleviate  the history dependence of the temporal convolution in the scheme \eqref{39eqn}. It is noteworthy that the semi-discrete scheme and the analysis
 in Section \ref{Sect3}, together with the fast algorithm to be described below, can be incorporated with various spatial discretisation. Compared with the classical Maxwell’s equations, the most  challenging issue in the H-N model is the treatment of the temporal  convolution with a singular kernel function.  Here, we demonstrate the full discretisation via a spectral-Galerkin method for the two-dimensional model.  This situation is reminiscent to the comparison of several methods for the Maxwell's equations in  \cite{monk1992a}: ``Of course this is not the setting of real physical interest; however, the two-dimensional case makes a convenient test problem.''  We remark  that one can extend the method to finite element methods or finite differences  in two or more  dimensions. 

\subsection{Fast temporal convolution algorithm} \label{FastC}
%In this section, let us illustrate a fast convolution algorithm for the implementation of the time-discretisation scheme \eqref{39eqn}. The key is to design a fast algorithm for \eqref{rehnmodelp}.

Among many recent developments of fast algorithms in particular for fractional integral/derivatives, there are a few works on developing fast algorithms for the much more involved kernel function, i.e., the ML function with three parameters.
This algorithm  can incorporate  into \eqref{39eqn} with different spatial discretisations.

\medskip
 We summarise the algorithm as follows.
\medskip
\begin{itemize}
  \item[\underline{\bf Step 1}] ~ Decompose \eqref{rehnmodelp} as
  \begin{equation}\label{LandHparts}
  \begin{split}
  {\bm P}(\bs x, t)
  & =   {\Delta \epsilon}\int_{t-{\Delta t}}^{t} e_{\alpha,\alpha\beta}^{\beta}(t-s;-1)\,{\bm E}(\bm x, s)\,{\rm d}s+{\Delta \epsilon}\int_{0}^{t-{\Delta t}} e_{\alpha,\alpha\beta}^{\beta}(t-s;-1)\,{\bm E}(\bm x, s)\,{\rm d}s  \\
  & :=  {\mathcal L}({\bm E}; t)+{\mathcal H}({\bm E}; t),
  \end{split}
  \end{equation}
  where ${\mathcal L}({\bm E};t)$ and ${\mathcal H}({\bm E}; t)$ are respectively the local  and history parts.  Corresponding to the discretisation  in \eqref{frarect},  we have
  $$
 {\mathcal L}({\bm E}; t_k)\approx {\mathcal L}\left(I_{\Delta t} {\bm E}; t_k\right)=\varpi^{(\alpha,\beta)}_{0}{\bm E}^k; \quad  {\mathcal H}({\bm E}; t_k)\approx  {\mathcal H}\left(I_{\Delta t} {\bm E}; t_k\right).
 %={\Delta \epsilon} \sum_{j=1}^{k-1}\varpi^{(\alpha,\beta)}_{k-j}{\bm E}^j.
 %{\mathcal H}({\bm E}^0,\cdots, {\bm E}^{k-1}),
  $$
 %where we refer to  \eqref{frarect} for the expression of the latter.
 % for $I_{\Delta t} {\bm E}$ is introduced in and $\varpi^{(\alpha,\beta)}_{0}$ is defined in \eqref{fracrectw1}.
We remark that the direct implementation based on the above   requires $O(N_t)$ storage and $O(N_t^2)$ operations, which is computationally expensive for long time  and multi-dimensional simulations. The essence of the fast algorithm is to further approximate the kernel function $e_{\alpha,\alpha\beta}^{\beta}(\cdot;-1)$ that allows for computing the history part in a recursive manner.

\medskip
  \item [\underline{\bf Step 2}]~ Given an integer $B\ge 2,$ let $L$ be the smallest integer satisfying $t_k < 2B^L{\Delta
      t}$. For $\ell=1,\cdots,L-1$, we can determine the integer $q_{\ell}$  and  $s_{\ell}=q_{\ell}B^{\ell}{\Delta t}$
      such that
      $$ t_k-s_{\ell} \in I_{\ell}:=[B^{\ell-1}{\Delta t}, (2B^{\ell}-1){\Delta t}].$$
      %where $q_{\ell} \in \big[\frac{k}{B^{\ell}}+\frac{1}{B^{\ell}}-2,\frac{k}{B^{\ell}}-\frac{1}{B}\big]$.
As such, we have (see  \cite{Lubich2002fast})
       $$ t_k-{\Delta t}=s_0>s_1>\cdots>s_{L-1}>s_L=0.$$

\medskip

  \item[\underline{\bf Step 3}]~ Seek the approximation of $e_{\alpha,\alpha\beta}^{\beta}(t;-1)$ on $I_{\ell}$ by applying the trapezoidal rule to a parametrisation of the contour integral for the inverse Laplace transform:
      \begin{equation}
      \label{SOEs}
      \begin{split}
      e_{\alpha,\alpha\beta}^{\beta}(t;-1)
      & = \frac{1}{2\pi i}\int_{\Gamma_{\ell}} \mathscr{L}\big[e_{\alpha,\alpha\beta}^{\beta}(t;-1)\big](\lambda)e^{t\lambda}\,{\rm d}\lambda= \frac{1}{2\pi i}\int_{\Gamma_{\ell}} \frac{e^{t\lambda}}{(\lambda^{\alpha}+1)^{\beta}}{\rm d}\lambda\\
%       =\sum_{j=-N_{\rm col}}^{N_{\rm col}} {\hat \omega}_j^{(\ell)}e^{t\lambda_j^{(\ell)}}\mathscr{L}\big[e_{\alpha,\alpha\beta}^{\beta}(t;-1)\big](\lambda_j^{(\ell)})
%      +O(\varepsilon_{\rm f})\\
      & \approx \sum_{j=-N_{\rm col}}^{N_{\rm col}-1} \frac{{\hat \omega}_j^{(\ell)}e^{t\lambda_j^{(\ell)}}}{\big(\big({\lambda}_j^{(\ell)}\big)^{\alpha}+1\big)^{\beta}}, % +O(\varepsilon_{\rm f}),
      \end{split}
      \end{equation}
      with a precision $\varepsilon_{\rm f}>0$ and a complex contour $\Gamma_{\ell}$ which can be suitably chosen following the ideas in \cite{Lubich2002fast,garrappa2015numerical}.
%      with a suitably chosen complex contour $\Gamma_{\ell}$ to be illustrated in detail below and a given precision $\varepsilon_{\rm f}>0$.
      Here, ${\hat \omega}_j^{(\ell)}$, $\lambda_j^{(\ell)}$ are the weights and quadrature points for the contour $\Gamma_{\ell}$. The number of quadrature points on $\Gamma_{\ell}$, $2N_{\rm col}$ is chosen independent of $\ell$.

\medskip
  \item[\underline{\bf Step 4}] ~ Using \eqref{SOEs}, the history part ${\mathcal H}\left(I_{\Delta t} {\bm E}; t_k\right)$ can be approximated by
      {\small \begin{eqnarray*}
      %\label{FastHist}
      {\mathcal H}\left(I_{\Delta t} {\bm E}; t_k\right)={\rm Im}\Bigg\{\sum_{\ell=1}^{L}\sum_{j=-N_{\rm col}}^{N_{\rm col}-1}  \frac{{\hat \omega}_j^{(\ell)}e^{(t_k-s_{\ell-1})\lambda_j^{(\ell)}}}{\big[\big({\lambda}_j^{(\ell)}\big)^{\alpha}+1\big]^{\beta}}
      {\bm y}\left(s_{\ell-1},s_{\ell},\lambda_j^{(\ell)}\right)\Bigg\},
      \end{eqnarray*}}
      where ${\rm Im}\{u\}$ stands for the  imaginary part of $u$, and
      $${\bm y}(s)={\bm y}\big(s,s_{\ell},\lambda_j^{(\ell)}\big)=\int_{s_{\ell}}^{s}e^{-(s-s_{\ell})
      \lambda_j^{(\ell)}}I_{\Delta t} {\bm E}(s) {\rm d}s$$
       satisfies the following ODE
      \begin{equation*}
      %\label{FastHistODE}
      {\bm y}'(s)=\lambda_j^{(\ell)}{\bm y}(s)+I_{\Delta t} {\bm E}(s),\quad {\bm y}(s_{\ell})=0.
      \end{equation*}
\end{itemize}
\smallskip
\begin{rem}
{\em This fast convolution algorithm has the same storage and computational cost as that in \cite{Lubich2002fast}, i.e., it requires $O(\log N_t)$ storage and $O(N_t \log N_t)$ operations over $N_t$ time steps, when only cost in time direction is considered. However, the direct implementation of the scheme \eqref{frarect} would require $O(N_t)$ storage and $O(N_t^2)$ operations, which is computational expensive and forms a bottleneck for long time simulation. It is worthy noting that the kernel function in H-N model {\rm(}see \eqref{gml}{\rm)} is much more complex than the kernel functions in \cite{Lubich2002fast,zeng2018stable} and references therein, so how to develop a fast convolution algorithm for the H-N model is much more involved. We also point out that  some different approaches were developed in   \cite{causley2011incorporating,xu2013a}.
%The fast algorithm in \cite{causley2011incorporating} got the approximation of local part, which is only efficient for $0.5 < \alpha, \beta \leq 1$, by the Laurent expansion in the frequency domain, while got the approximation of the history part by following the idea in \cite{Alpert2000rapid,jiang2004fast}. A bootstrap method was presented by Xu et al. \cite{xu2013a} for finding efficient sum-of-poles approximations of causal functions in the complex domain which contain H-N, C-C and D-C kernel functions. Different from these two works, we follow the idea in \cite{Lubich2002fast,zeng2018stable} to deal with this involved kernel function.
 }
\end{rem}

%\section{The fully discrete scheme of two-dimensional H-N model}\label{Sect5}
%\setcounter{equation}{0}
%\setcounter{lmm}{0}
%\setcounter{thm}{0}

\subsection{Full discretisation of a two-dimensional H-N model} \label{subsect42}
As an illustration of full-discrete scheme,  we consider the spatial discretisation of the H-N model using the spectral-Galerkin method in two dimensions.
More precisely, we  consider \eqref{scaledsystem} on the rectangular  domain $\Omega=(a, b)\times(c, d)$ of the form:
\begin{subequations}\label{hnte2D}
\begin{align}
& \epsilon_{\infty} \pt{\bm E} +\pt {\bm P} = {\rm {\bf {curl}}}~H, \quad\,
\pt H =-{\rm curl}~{\bm E} & {\rm in} \;\;  \Omega \times (0,T], \label{hnte1}\\
%\frac{\partial H}{\partial t}=-{\rm curl}~{\bm E}, \qquad \qquad \qquad \qquad \qquad \qquad \qquad \qquad \quad ~~~\;  {\rm in} \quad \Omega \times (0,T] \label{hnte2}\\[0.2cm]
& {\bm P}(x,y,t)={\Delta \epsilon}\int_{0}^{t} e_{\alpha,\alpha\beta}^{\beta}(t-s;-1){\bm E}(x,y,s)\,{\rm d}s    & {\rm in} \;\; \Omega \times (0,T], \label{hnte3}\\
& {\bm E}(x,y,0)={\bm E}_0(x,y),\quad H(x,y,0) = H_0(x,y), &  {\rm in} \;\; \Omega, \label{hntei}\\[2pt]
& E_x(x,c,t)=E_x(x,d,t)=E_y(a,y,t)=E_y(b,y,t)=0    & {\rm for} \;\;  t \in (0,T], \label{hnteb}
\end{align}
\end{subequations}
where ${\bm E}=(E_x,E_y)^T$ and ${\bm P}=(P_x,P_y)^T$ are vectors, but $H$ is a scalar unknown. Recap on  the two-types of curl operators:
$${\rm {\bf {curl}}}~H=\left(\py{H},-\px{H}\right)^T, \quad {\rm curl}~{\bm E}=\py{E_x}-\px{E_y}.$$

%In the following of this section, we adopt the Legendre-Galerkin method to discretize the model \eqref{hnte} in space. It is worthy noting that the space can be approximated by the finite difference method and the finite element method.

  %Denote $I_x=(a,b)$ and $I_y=(c,d)$, so $\Omega=I_x \times I_y$.
  Let $\mathbb P_{N}$ be the space of the algebraic polynomials in one variable of degree not more than $N,$
  and let  $\mathbb P_{N}^0$ be the subspace of $\mathbb P_{N},$ where each polynomial vanishes at the two end-points of the interval.
   We further denote $V_N=\mathbb P_{N}\times \mathbb P_{N},$ and define
  \begin{equation*}
  \bm V_N^0=\big\{(u,v)^T\in (V_N)^2 \,:\, u|_{y=c}=u|_{y=d}=0, \; v|_{x=a}=v|_{x=b}=0 \big\}.
  \end{equation*}
  The full-discrete scheme for \eqref{hnte} is to find ${\bm E}_{N}^k, {\bm P}_{N}^k \in   \bm V_N^0$ and $H_N^k \in V_N$ such that for $k \geq 1,$
\begin{subequations}\label{hntebelg}
%\begin{numcases}{}
\begin{align}
& \epsilon_{\infty} \left(\delta_t {\bm E}_{N}^k, {\bm \phi}\right)+\left(\delta_t {\bm P}_{N}^k , {\bm \phi}\right)=\left({\rm {\bf {curl}}}~{H}_N^k, {\bm \phi}\right), \qquad \quad \forall {\bm \phi} \in  \bm V_N^0,  \label{hntebelg1} \\[4pt]
&\left(\delta_t H_{N}^k, \psi\right)=-\left({\rm curl}~{{\bm E}}_{N}^k, \psi\right), \qquad \qquad \qquad \qquad~\qquad  \;\;\;\; \forall \psi \in V_N,  \label{hntebelg2}
%\end{equation}
\\& \left({\bm P}_{N}^k, {\bm \varphi}\right)={\Delta \epsilon}\sum_{j=1}^{k}\varpi^{(\alpha,\beta)}_{k-j}\big({\bm E}_{N}^j, {\bm \varphi}\big), \qquad \qquad  \qquad  \quad\;\; \forall {\bm \varphi} \in   \bm V_N^0, \label{hntebelg3}
%%\end{numcases}
\end{align}
\end{subequations}
where the initial values are
\begin{equation}
\label{hntecnlgicon}
{\bm E}_{N}^0 = \mathcal{I}_N{\bm E}_0, \quad H_{N}^0 = \mathcal{I}_N H_{0},  \quad {\bm P}_{N}^0 = \bf 0.
\end{equation}
Here, $\mathcal{I}_N: C(\bar \Omega)\to V_N$ is the tensorial Legendre-Gauss-Lobatto (LGL) interpolation operator.
%\textcolor[rgb]{0.00,0.00,1.00}{We note that the divergence-free nature of the magnetic flux density are lost in this two-dimensional case \cite{monk1992a}. But it is easy to check that the divergence-free nature of the electric flux density, both the continuous and discrete ones.}
%{\color{blue} \begin{rem}\label{FullDisDivFree} {\em Similar to the derivation in Remark \ref{HNdivergence} and Remark \ref{TDisDivFree}, we can check that electric and polarization fields, both the continuous and discrete ones, are divergence free if the initial electric field is divergence free, which lead the divergence-free nature of the electric flux density. But the divergence-free nature of the magnetic flux density is lost in this two-dimensional case \cite{monk1992a}.}
%\end{rem}}

\begin{rem}\label{FullDisDivFree} {\em Note that ${\bm E}_N^k$, ${\bm P}_N^k$ and $H_N^k$ are expansions in terms of  Legendre basis polynomials. Then taking the divergence of \eqref{hntebelg1},  we find that $\nabla\cdot(\epsilon_{\infty} {\bm E}_N^k +{\bm P}_N^k)=0$ as we can show that $\nabla\cdot({\rm {\bf {curl}}}~H_N^k) \equiv 0$.
%To demonstrate the fact $\nabla\cdot({\rm {\bf {curl}}}~H_N^k) \equiv 0$, we select $\{L_i(x)L_j(y)\}_{i,j=0}^N$ to be the basis functions of $V_N$, then we have
Indeed, we can write $H_N^k$ in terms of the Legendre polynomials:
$$H_N^k = \sum_{i=0}^{N}\sum_{j=0}^{N} {\tilde H}_{ij}^k L_i(x)L_j(y),$$
where ${\tilde H}_{ij}^k$ are the expansion coefficients. From direct calculation, we  obtain
%$${\rm {\bf {curl}}}~H_N^k = \sum_{i=0}^{N}\sum_{j=0}^{N} {\tilde H}_{ij}^k \big(L_i(x)L'_j(y),-L'_i(x)L_j(y)\big)^T,$$
%and thus
$$\nabla\cdot({\rm {\bf {curl}}}~H_N^k) = \sum_{i=0}^{N}\sum_{j=0}^{N} {\tilde H}_{ij}^k \big(L'_i(x)L'_j(y)-L'_i(x)L'_j(y)\big) \equiv 0.$$
Therefore, we can claim $\nabla\cdot {\bm P}_N^k=0$ and $\nabla\cdot{\bm E}_N^k=0$ like in the derivation in Remark \ref{TDisDivFree}. However, we note that the discrete magnetic field $H_N^k$ is a scalar, thus the divergence of it is not defined \cite{monk1992a}.   We shall provide some numerical verifications in  Table \ref{Tab2}. }
%Therefore, the divergence-free nature of the discrete magnetic flux density is lost in this two-dimensional case \cite{monk1992a}.}
\end{rem}

%
%
%  Let $\{\hat x_i,\hat w_i\}_{i=0}^N$ be the Legendre-Gauss-Lobatto (LGL) points and weights  on the reference interval
%$\hat I=(-1,1).$
%  Let $w_i^x=\frac{b-a}{2}{\hat w}_i$ and $w_i^y=\frac{d-c}{2}{\hat w}_i$. The corresponding LGL points on $\Omega$ are given  by
%\begin{equation}\label{xyj}
%x_i=\frac{(b-a){\hat x}_i+a+b}{2} \in {\bar I}_x, \quad y_j=\frac{(d-c){\hat x}_j+c+d}{2} \in {\bar I}_y, \quad 0\le i,j\le N.
%\end{equation}
%Accordingly, we can define the LGL interpolation operator $\mathcal I_N={\mathcal I}_N^x\circ {\mathcal I}_N^y,$ which is a tensor product of the one-dimensional LGL interpolation operators as in  \cite{shen2011spectral}.
%
%Let $V_N=\mathbb{P}_N^2$. In order to construct our numerical scheme clearly, we also introduce the following spaces
%\begin{eqnarray*}
%%\label{spaces}
%& &H^{1,x}_0(\Omega)=\left\{u: u\in H^1(\Omega), u(a,y)=u(b,y)=0\right\}, \nonumber\\
%& &H^{1,y}_0(\Omega)=\left\{u: u\in H^1(\Omega), u(x,c)=u(x,d)=0\right\}, \nonumber\\
%& &V^{x,0}_N=V_N\cap H^{1,x}_0,\quad V^{y,0}_N=V_N\cap H^{1,y}_0,   \\
%& &{\bm U}_N^0=\left\{(u,v)^T: u\in V^{y,0}_N,~ v\in V^{x,0}_N\right\}, \nonumber \\
%& &{\bm U}_N=\left\{(u,v)^T: u\in V_N,~ v\in V_N\right\}
%\end{eqnarray*}
%
%We define the LGL interpolation operator $\mathcal{I}_N: C({\bar \Omega}) \rightarrow V_N$ such that
%\begin{eqnarray}
%\label{interp}
%(\mathcal{I}_N u)(x_i,y_j)=u(x_i,y_j), \quad i,j=0,1,\cdots,N,
%\end{eqnarray}
%where $x_i$ and $y_j$ are the LGL points defined by \eqref{refx} and \eqref{refy}, respectively.
%
%

The stability and well-posedness of the scheme  \eqref{hntebelg} is a direct consequence of  Theorem \ref{zeng-thmSemiDisStab}.
\begin{thm}\label{thmFullyDisStab} The full-discrete scheme \eqref{hntebelg} is unconditionally stable in the sense that  for all $\Delta t>0,$
\begin{equation*}
%\label{fdisstabthm1}
\mathscr E_{N}^k\le \mathscr E_{N}^{k-1}\le \cdots\le \mathscr E_{N}^0, \quad k\ge 1,
\end{equation*}
where $\mathscr E^0_N:=\epsilon_{\infty}\|{\bm E}_N^{0}\|^2+ \|{H}_N^{0}\|^2$ and
\begin{equation*}
%\label{fdefnEk}
\mathscr E_{N}^k:=\epsilon_{\infty}\|{\bm E}_{N}^{k}\|^2+ \|{H}_{N}^{k}\|^2+{\Delta \epsilon}\sum_{j=1}^{k}\varpi^{(\alpha,\beta)}_{k-j}\|{\bm E}_{N}^j\|^2,\quad   k\ge 1.
\end{equation*}
\end{thm}

Following the argument for proving Theorem \ref{thmsemicon}, we can show the convergence. To this end, we sketch the proof with an emphasis on the estimation of spatial error.

  Let ${\bm E}_{*}=(E_{x*}, E_{y*})^T\in {\bm V}_N^0, {\bm P}_{*}=(P_{x*}, P_{y*})^T \in {\bm V}_N^0$, and  $H_*\in V_N$ be some suitable orthogonal projections to be specified later. We introduce
\begin{equation*}
\begin{split}
& {\bm e}_1^k = {\bm E}_N^k-{\bm E}_{*}^k,\quad\;  {e}_2^k = H_N^k-H_{*}^k,\;
\quad  {\bm e}_3^k = {\bm P}_N^k-{\bm P}_{*}^k,\\
&  {\bm \eta}_1^k = {\bm E}|_{t_k}-{\bm E}_{*}^k,\quad {\eta}_2^k = H|_{t_k}-H_{*}^k,  \quad {\bm \eta}_3^k = {\bm P}|_{t_k}-{\bm P}_{*}^k.
\end{split}
\end{equation*}
%$${\bm e}_1^k = {\bm E}_N^k-{\bm E}_{*}^k,\qquad {\bm \eta}_1^k = {\bm E}(t_k)-{\bm E}_{*}^k,$$
%$${e}_2^k = H_N^k-H_{*}^k,\qquad {\eta}_2^k = H(t_k)-H_{*}^k,$$
%$${\bm e}_3^k = {\bm P}_N^k-{\bm P}_{*}^k,\qquad {\bm \eta}_3^k = {\bm P}(t_k)-{\bm P}_{*}^k.$$
%
We infer from  \eqref{hnte}-\eqref{hntebelg} the error equations:
\begin{subequations}\label{errorequ}
%\begin{numcases}{}
\begin{align}
& \epsilon_{\infty} \left(\delta_t {\bm e}_{1}^k, {\bm \phi}\right)+
\left(\delta_t {\bm e}_{3}^k, {\bm \phi}\right)-\left({e}_2^k, {\rm curl}\,{\bm \phi}\right)=\left(\delta_t {\bm \eta}_{3}^k, {\bm \phi}\right)+\left({\bm f}_1^k, {\bm \phi}\right),\label{errorequ1} \\[4pt]
&\left(\delta_t e_{2}^k, \psi\right)+\left({\rm curl}~{{\bm e}}_{1}^k, \psi\right)=(f_2^k,\psi),  \label{errorequ2}
%\end{equation}
\\& \left({\bm e}_{3}^k, {\bm \varphi}\right)={\Delta \epsilon}\sum_{j=1}^{k}\varpi^{(\alpha,\beta)}_{k-j}\big({\bm e}_{1}^j, {\bm \varphi}\big)+\big({\bm \eta}_{3}^k, {\bm \varphi}\big)-{\Delta \epsilon}\left({\bm f}_3^k, {\bm \varphi}\right), \label{errorequ3}
%%\end{numcases}
\end{align}
\end{subequations}
where
\begin{eqnarray}
{\bm f}_1^k := {\bm R}^k_1+\epsilon_{\infty}\delta_t {\bm \eta}_{1}^k -{\rm {\bf {curl}}}~{\eta}_2^k,\quad f_2^k := {R}^k_2+\delta_t {\eta}_{2}^k+{\rm curl}~{{\bm \eta}}_{1}^k,\quad {\bm f}_3^k := \sum_{j=1}^{k}\varpi^{(\alpha,\beta)}_{k-j}{\bm \eta}_1^j+{\bm R}^k_0. %\nonumber
\end{eqnarray}
Here,  ${\bm R}^k_0$, ${\bm R}^k_1$, and ${ R}^k_2$ are defined in  \eqref{TruncErr0} and \eqref{RkdefnA} with reduction  to  the two-dimensional setting.
Like \eqref{wbnd1}, we can derive
\begin{equation*}
\begin{split}
\|{\bm e}_{1}^k\|^2+\|e_{2}^k\|^2
\le   C  \Big(\|{\bm e}_{1}^0\|^2+\|e_{2}^0\|^2+{\Delta t}\sum_{j=1}^k \big(\|{\bm f}_1^j\|^2
+ (\Delta \epsilon)^2  \|\delta_t {\bm f}_3^j \|^2+ \|f_2^j\|^2\big) \Big),
\end{split}
\end{equation*}
and similar to the proof of Corollary \ref{corSemiDisStab}, we can obtain
\begin{eqnarray*}
%\label{conP}
\|{\bm e}_{3}^k\|
\le C %{\Delta \epsilon}\, e_{\alpha,\alpha\beta+1}^{\beta}(t_k;-1)\,
\max_{1\leq j \leq k}\|{\bm e}_{1}^j\|+ \|{\bm \eta}_{3}^k\| +{\Delta \epsilon}\|{\bm f}_{3}^k\|.
\end{eqnarray*}
Recall that Lemma \ref{lem3-3} provides the error bounds of  ${\bm R}^k_0$, ${\bm R}^k_1$, and ${ R}^k_2$, so it suffices to estimate the errors involving ${\bm \eta}_1^k, {\eta}_2^k$ and ${\bm \eta}_3^k.$ We first deal with the summation in ${\bm f}_{3}^k.$ Following the same lines as  deriving the last estimate in Lemma \ref{lem3-3}, one has
\begin{equation}
\label{temp0}
\begin{split}
& \Big\|\delta_t \Big(\sum_{i=1}^{j}\varpi^{(\alpha,\beta)}_{j-i}{\bm \eta}_1^i\Big)\Big\| =\frac{1}{\Delta t}\Big\|\sum_{i=1}^{j}\varpi^{(\alpha,\beta)}_{j-i}{\bm \eta}_{1}^i-\sum_{i=1}^{j-1}\varpi^{(\alpha,\beta)}_{j-1-i}{\bm \eta}_{1}^i\Big\|\\
& \le \frac{1}{\Delta t}\Big\|\sum_{i=1}^{j} \int_{t_{i-1}}^{t_i} e_{\alpha,\alpha\beta}^{\beta}(t_{j}-s;-1){\bm \eta}_{1}^{i}\,{\rm d}s-\sum_{i=1}^{j-1} \int_{t_{i-1}}^{t_i} e_{\alpha,\alpha\beta}^{\beta}(t_{j-1}-s;-1){\bm \eta}_{1}^{i}\,{\rm d}s \Big\|\\
&= \frac{1}{\Delta t}\Big\|\int_{0}^{t_1} e_{\alpha,\alpha\beta}^{\beta}(t_{j}-s;-1){\bm \eta}_{1}^{1}\,{\rm d}s+\sum_{i=1}^{j-1} \int_{t_{i-1}}^{t_i} e_{\alpha,\alpha\beta}^{\beta}(t_{j-1}-s;-1)({\bm \eta}_{1}^{i+1}-{\bm \eta}_{1}^{i})\,{\rm d}s \Big\| \\
&\le e_{\alpha,\alpha\beta}^{\beta}(t_{j-1};-1)\|{\bm \eta}_{1}^{1}\|+\sum_{i=1}^{j-1} \int_{t_{i-1}}^{t_i} e_{\alpha,\alpha\beta}^{\beta}(t_{j-1}-s;-1)\|\delta_t{\bm \eta}_{1}^{i+1}\|\,{\rm d}s.
\end{split}
\end{equation}

 We proceed with introducing some   orthogonal projections, and review the relevant  approximation results in \cite{shen2011spectral}.  Let $\pi_{N,x}^{1}: H^1(I_x) \to {\mathbb P}_N$ be the $H^1$-orthogonal projection, and  let $\pi_{N,x}^{1,0}: H^1_0(I_x) \to {\mathbb P}_N^0$
be the $H^1_0$-orthogonal projection. Likewise,  we can define the operators $\pi_{N,y}^{1}$ and $\pi_{N,y}^{1,0}$ on the interval $I_y.$ Here we choose
\begin{equation*}
\{E_{x*}; P_{x*}\}=(\pi_{N,x}^{1}\circ \pi_{N,y}^{1,0})\{E_x; P_{x}\}, \;\; \{E_{y*}; P_{y*}\}=(\pi_{N,x}^{1,0}\circ \pi_{N,x}^{1}) \{E_y; P_y\}, \;\; H_*= (\pi_{N,x}^{1}\circ \pi_{N,y}^{1}) H.
\end{equation*}
According to \cite{shen2011spectral}, we have
\begin{equation}\label{ProjErr}
\|U_{x*}-U_x\|_{H^s(\Omega)}\le  cN^{s-r}\|U_x\|_{H^r(\Omega)},  \;\; \|U_{y*}-U_y\|_{H^s(\Omega)} \leq cN^{s-r} \|U_y\|_{H^r(\Omega)},\;\;  s=0,1,\;\; r\ge 1,
\end{equation}
%where $U$ can be one of the variables $E$, $P$ and $H$. For the LGL interpolation operator $\mathcal{I}_N$, we have \cite{shen2011spectral}
and
\begin{equation}\label{InterpErr}
\|U-\mathcal{I}_N U\| \leq cN^{-r} \|U\|_{H^r(\Omega)},\;\;\; r\ge 1.
\end{equation}
Below, we shall set $U$ to be the unknowns.
%where $U$ can be one of the variables ${\bm E}$ and $H$.
Now we are in a position to give the error estimates involving ${\bm \eta}_1^k, {\eta}_2^k$ and ${\bm \eta}_3^k.$ From \eqref{ProjErr}, we have
\begin{eqnarray*}
\|{\bm \eta}_{1}^j\|\leq cN^{-r}\left(\|E_x(\cdot, t_j)\|_{H^r(\Omega)}+\|E_y(\cdot,t_j)\|_{H^r(\Omega)}\right) \leq
cN^{-r}\|{\bm E}\|_{L^{\infty}(0,T;{\bm H}^r(\Omega))},
\end{eqnarray*}
\begin{eqnarray*}
\|{\bm \eta}_{3}^k\|\leq cN^{-r}\left(\|P_x(\cdot,t_k)\|_{H^r(\Omega)}+\|P_y(\cdot,t_k)\|_{H^r(\Omega)}\right)\leq cN^{-r}\|{\bm P}\|_{L^{\infty}(0,T;{\bm H}^r(\Omega))},
\end{eqnarray*}
\begin{eqnarray*}
\|\delta_t {\bm \eta}_{1}^k\|\leq c \|\partial_t {\bm \eta}_{1}^k\| \leq cN^{-r}\left(\|\partial_t E_x(\cdot,t_k)\|_{H^r(\Omega)}+\|\partial_t E_y(\cdot,t_k)\|_{H^r(\Omega)}\right)
\leq cN^{-r}\|\partial_t {\bm E}\|_{L^{\infty}(0,T;{\bm H}^r(\Omega))},
\end{eqnarray*}
\begin{equation*}
\begin{split}
\|{\rm curl}~{{\bm \eta}}_{1}^k\|
&  =   \|\partial_y(E_x-E_{x*})-\partial_x(E_y-E_{y*})\|
\leq \|\partial_y(E_x-E_{x*})\|+\|\partial_x(E_y-E_{y*})\| \\
&\leq cN^{1-r}\left(\|E_x(\cdot,t_k)\|_{H^r(\Omega)}+\|E_y(\cdot,t_k)\|_{H^r(\Omega)}\right)\leq cN^{(1-r)}\|{\bm E}\|_{L^{\infty}(0,T;{\bm H}^r(\Omega))},
\end{split}
\end{equation*}
\begin{eqnarray*}
\|\delta_t {\eta}_{2}^k\|\leq c \|\partial_t {\eta}_{2}^k\| \leq cN^{-r}\|\partial_t H(\cdot,t_k)\|_{H^r(\Omega)}\leq cN^{-r}\|\partial_t H\|_{L^{\infty}(0,T;{H}^r(\Omega))},
\end{eqnarray*}
and
\begin{eqnarray*}
\|{\rm {\bf {curl}}}~{\eta}_2^k\| = \|\partial_x {\eta}_2^k\|+\|\partial_y {\eta}_2^k\|\leq cN^{(1-r)}\|H^k\|_{H^r(\Omega)}\leq cN^{(1-r)}\|H\|_{L^{\infty}(0,T;{H}^r(\Omega))}.
\end{eqnarray*}
Using  the triangular inequality  and the approximation results \eqref{ProjErr}-\eqref{InterpErr}, we obtain
\begin{eqnarray*}
\|{\bm e}_{1}^0\|\le cN^{-r}\|{\bm E}_0\|_{{\bm H}^r(\Omega)}, \quad \|e_{2}^0\|\leq cN^{-r}\|H_0\|_{H^r(\Omega)},
\end{eqnarray*}
for the initial errors. Collecting all the estimates above, and noting \eqref{temp0}, we present the following convergence result.
%%%%%%%%%%%%%%%%%%%%%%%%%%%%%%%%%%%%%%%%%%%%%%%%%%%%%%%%%%%%%%%%%%%%%%%%%%%%%%%%%%%%%
%%%%%%%%%%%%%%%%%%%%%%%%%%%%%%%%% Theorem Con %%%%%%%%%%%%%%%%%%%%%%%%%%%%%%%%%%%%%%%
%%%%%%%%%%%%%%%%%%%%%%%%%%%%%%%%%%%%%%%%%%%%%%%%%%%%%%%%%%%%%%%%%%%%%%%%%%%%%%%%%%%%%
\begin{thm}\label{thmcon} %Suppose that $N, k$ and $r$ are positive integers with $r>1$ and $1\leq k \leq N_t$.
Let ${\bm E}_{N}^k, {\bm P}_{N}^k, H_N^k$ be the solution of \eqref{hntebelg} that approximates the solution  of  \eqref{hnte}.
%$${\bm E}(x,y,t)=(E_x(x,y,t),E_y(x,y,t))^T, {\bm P}(x,y,t)=(P_x(x,y,t),P_y(x,y,t))^T, H(x,y,t),
%$$
%of  \eqref{hnte}.
 Assume
$$
E_x,P_x \in H^2\big(0,T;H^r(\Omega)\cap \big(H^{1}(I_x)\otimes H_0^{1}(I_y)\big)\big),~E_y, P_y \in H^2\big(0,T;H^r(\Omega)\cap \big(H_0^{1}(I_x)\otimes H^{1}(I_y)\big)\big),
$$
and $H \in H^2\left(0,T;H^r(\Omega)\right)$, then for $k\ge 1,$
\begin{equation*}
%\label{ConResult}
\begin{split}
&  \|{\bm E}(\cdot, t_k)-{\bm E}_{N}^k\|^2+\|H(\cdot, t_k)-{H}_{N}^k\|^2+\|{\bm P}(\cdot, t_k)-{\bm P}_{N}^k\|^2 \\
& \leq C(\Delta t)^2\big(\|{\bm E}\|^2_{H^2(0,T;{\bm L}^2(\Omega))}+\|{\bm P}\|^2_{H^2(0,T;{\bm L}^2(\Omega))}+\|H\|^2_{H^2(0,T;L^2(\Omega))} \big)\\
& ~+ CN^{2(1-r)}\big(N^{-2}\|{\bm E}_{0}\|^2_{{\bm H}^r(\Omega)}+N^{-2}\|H_0\|^2_{{H}^r(\Omega)}+N^{-2}\|\partial_t {\bm E}\|^2_{L^{\infty}(0,T,{\bm H}^r(\Omega))}+N^{-2}\|\partial_t H\|^2_{L^{\infty}(0,T;{H}^r(\Omega))} \\
& ~+N^{-2}\|{\bm E}\|^2_{L^{\infty}(0,T,{\bm H}^r(\Omega))}+N^{-2}\|{\bm P}\|^2_{L^{\infty}(0,T;{\bm H}^r(\Omega))}+\|H\|^2_{L^{\infty}(0,T;{H}^r(\Omega))}+\|{\bm E}\|^2_{L^{\infty}(0,T;{\bm H}^r(\Omega))}\big),
\end{split}
\end{equation*}
for a suitable $\Delta t$ {\rm(}see Theorem \ref{zeng-thmSemiDisStab} and Remark \ref{exp-const}{\rm).}
Here,  $C$ is a positive constant independent of   ${\Delta t},N$ and any function.
\end{thm}

\subsection{Numerical results}\label{subsect43}

In this subsection, we provide ample numerical results to show the efficiency and accuracy of the proposed methods with a focus on the performance of the treatment in time discretisation.
\subsubsection{Accuracy and efficiency tests} \label{subsubsect431}
%\begin{example}
%\label{exmp1}
Consider the system \eqref{hnte} with  the exact solution:
\begin{equation*}
\begin{split}
& {\bm E}(x,y,t)=\frac{t^{4}}{\Gamma(5)}{\bm w}(x,y),\;\;  {\bm P}(x,y,t)={\Delta\epsilon}\,e^{\beta}_{\alpha,\alpha\beta+5}(t;-1)\,{\bm w}(x,y),\;\; {\bm w}(x,y)=\binom{-\cos(\pi x)\sin(\pi y)}{\quad\sin(\pi x)\cos(\pi y)},\\
&H(x,y,t)=\Big(\frac{4\epsilon_{\infty}}{\pi\Gamma(5)}t^{3}
+{\Delta\epsilon}\, e^{\beta}_{\alpha,\alpha\beta+4}(t;-1)\Big)\cos(\pi x)\cos(\pi y).
\end{split}
\end{equation*}
%where ${\bm w}(x,y)=\binom{-\cos(\pi x)\sin(\pi y)}{\sin(\pi x)\sin(\pi y)}$, and
%$$H(x,y,t)=\Big(\frac{4\epsilon_{\infty}}{\pi\Gamma(5)}t^{3}
%+{\Delta\epsilon}e^{\beta}_{\alpha,\alpha\beta+4}(t;-1)\Big)\cos(\pi x)\cos(\pi y),$$
 As such,  the second equation in \eqref{hnte} must have  a source term
$$ f(x,y,t)=\bigg(\frac{2\pi }{\Gamma(5)}t^4+\frac{12\epsilon_{\infty}}{\pi\Gamma(5)}t^{2}+\frac{\Delta \epsilon }{\pi}e^{\beta}_{\alpha,\alpha\beta+3}(t;-1)\bigg)\cos(\pi x)\cos(\pi y),$$
which one can verify  by using the formulas in  \cite[(2.10) and (2.26)]{kilbas2004generalized}.
%\end{example}

For notational simplicity, we denote by $U_{N,D}^k$ and $U_{N,F}^k$  the numerical solutions derived by the direct and fast algorithms at $t_k=k{\Delta t}.$ Correspondingly, we denote the discrete $L^2$-errors by  ${\rm ErrU_{F}}:=\| U(\cdot, t_k)-U_{N,F}^{k}\|_{N}$ and ${\rm ErrU_{\rm DF}}:=\| U_{N,D}^{k}-U_{N,F}^{k}\|_{N},$  respectively, where $U$ can be ${\bm E}$, $H$ or ${\bm P}$.
In the following tests, we take $\Omega=(-1,1)^2$ and $\epsilon_{\infty}=\Delta \epsilon=1$.

\begin{table}[!th]
\centering\small
\caption{Errors and convergence rates of the fast temporal convolution algorithm.}\label{Tab1}
\vspace*{-6pt}
\begin{tabular}{|c|c|c|c|c|c|c|c|c|c|c|c|}
\hline
\cline{1-1}
${\Delta t}$   & ${\rm ErrE_F}$ & ${\rm Order}$ & ${\rm ErrH_F}$ & ${\rm Order}$ & ${\rm ErrP_F}$ & ${\rm Order}$ & ${\rm ErrE_{DF}}$ &${\rm ErrH_{DF}}$ &${\rm ErrP_{DF}}$\\ \hline\hline
$2^{-4}$  &6.4914e-03 &  -   &2.0905e-03 &  -   &6.1861e-03  &  -   &2.8917e-16 &1.4197e-16 &2.5075e-16 \\\hline
%$2^{-5}$  &3.3393e-03 & 0.96 &9.1368e-04 & 1.20 &3.2732e-03  & 0.92 &4.4050e-16 &1.9221e-16 &2.2719e-16 \\\hline
$2^{-6}$  &1.6836e-03 & 0.99 &4.0510e-04 & 1.17 &1.7062e-03  & 0.94 &3.6776e-16 &1.7787e-16 &1.4391e-16 \\\hline
%$2^{-7}$  &8.4067e-04 & 1.00 &1.8257e-04 & 1.15 &8.8139e-04  & 0.95 &7.5962e-16 &5.3546e-16 &4.3589e-16 \\\hline
$2^{-8}$  &4.1803e-04 & 1.01 &8.3436e-05 & 1.13 &4.5249e-04  & 0.96 &2.7708e-16 &4.3866e-16 &5.4712e-17 \\\hline
%$2^{-9}$  &2.0758e-04 & 1.01 &3.8559e-05 & 1.11 &2.3122e-04  & 0.97 &6.5736e-16 &5.0518e-16 &2.9108e-16 \\\hline
$2^{-10}$ &1.0307e-04 & 1.01 &1.7985e-05 & 1.10 &1.1771e-04  & 0.97 &3.8889e-16 &5.4694e-16 &5.9642e-17 \\\hline
%$2^{-11}$ &5.1209e-05 & 1.01 &8.4565e-06 & 1.09 &5.9752e-05  & 0.98 &5.3031e-16 &3.4459e-16 &1.1021e-16 \\\hline
$2^{-12}$ &2.5458e-05 & 1.01 &4.0045e-06 & 1.08 &3.0253e-05  & 0.98 &1.6182e-15 &1.8799e-16 &1.3962e-15 \\\hline
%$2^{-13}$ &1.2666e-05 & 1.01 &1.9085e-06 & 1.07 &1.5286e-05  & 0.98 &1.3534e-15 &8.5955e-16 &1.6426e-15 \\\hline
$2^{-14}$ &6.3060e-06 & 1.01 &9.1495e-07 & 1.06 &7.7104e-06  & 0.99 &4.6538e-16 &3.1352e-15 &1.2561e-15 \\\hline
%$2^{-15}$ &3.1415e-06 & 1.01 &4.4095e-07 & 1.05 &3.8835e-06  & 0.99 &6.9033e-15 &6.0573e-15 &1.1016e-14 \\\hline
%$2^{-16}$ &1.5659e-06 & 1.00 &2.1353e-07 & 1.05 &1.9537e-06  & 0.99 &2.8642e-14 &2.2088e-14 &4.1228e-14 \\\hline
\end{tabular}
%\vspace{2.5mm}
\end{table}

Firstly, in Table \ref{Tab1},  we tabulate  the discrete $L^2$-errors between the exact  and  numerical solutions, together with  convergence orders,  obtained by the schemes with $\alpha=\beta=0.5$ and $N=50$  at $T=1$.  In the rightmost three columns,
we list the errors between the numerical solutions by direct and fast algorithms (with $N_{\rm col}=30$), which are apparently  negligible.     We also observe that the first-order convergence  as expected.

\begin{wrapfigure}{r}{0.5\textwidth}
 \begin{center}
 \includegraphics[width=0.43\textwidth]{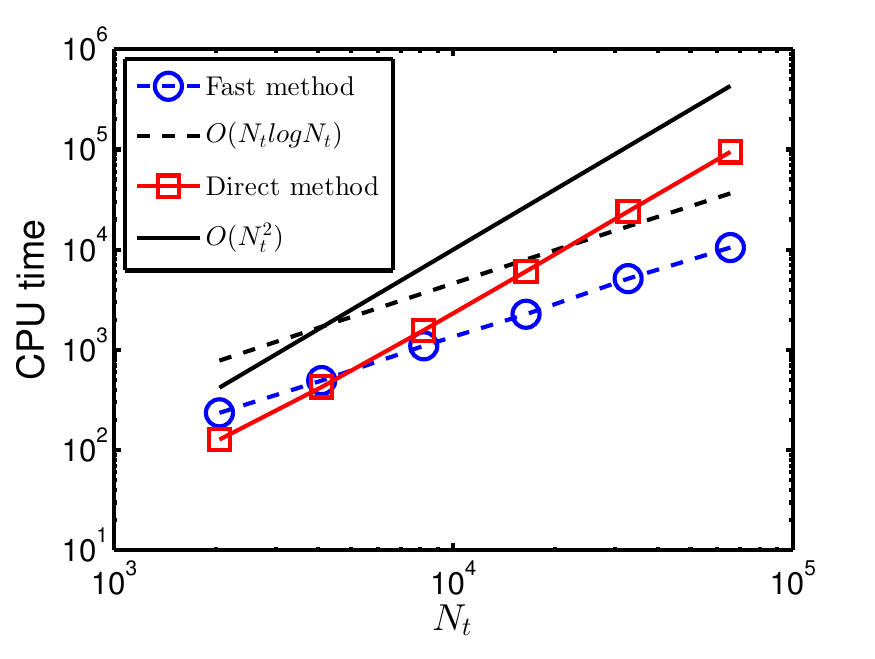}
\caption{\small Direct  versus  fast  algorithms in log-log scale} \label{fig1}
 \end{center}
\end{wrapfigure}

%%%%%%%%%%%%%%%%%%%%%%%%%%%%%%%%%%%%%%%%%%%%%%%%%%%%%%%%%%%%%%%%%%%%%%%%%%%%%%%%%%%%%
%%%%%%%%%%%%%%%%%%%%%%%%%%%%%%%%% Table 1 & Figure 1 %%% %%%%%%%%%%%%%%%%%%%%%%%%%%%%%%%%%%%%%%
%%%%%%%%%%%%%%%%%%%%%%%%%%%%%%%%%%%%%%%%%%%%%%%%%%%%%%%%%%%%%%%%%%%%%%%%%%%%%%%%%%%%%
%\begin{minipage}{1.0\textwidth}
%\centering
%\begin{minipage}{0.4\textwidth}
%\centering
%\makeatletter\def\@captype{table}\makeatother
%\caption{\footnotesize  $L^2$-errors between the analytical solution and the numerical ones and convergence orders.}
%\label{Tab1}
%\begin{tabular}{|c|c|c|c|}
%\hline
%\cline{1-1}
%$\tau$   & ${\rm ErrE_D}$ & ${\rm ErrE_F}$ &${\rm Order}$\\ \hline\hline
%$2^{-4}$  &6.4914e-03 &6.4914e-03  &0.97  \\\hline
%$2^{-6}$  &1.6836e-03 &1.6836e-03  &1.00  \\\hline
%$2^{-8}$  &4.1803e-04 &4.1803e-04  &1.00  \\\hline
%$2^{-10}$ &1.0307e-04 &1.0307e-04  &1.00  \\\hline
%$2^{-12}$ &2.5458e-05 &2.5458e-05  &1.00  \\\hline
%$2^{-14}$ &6.3060e-06 &6.3060e-06  &1.00  \\\hline
%\end{tabular}
%\end{minipage}
%\begin{minipage}{0.4\textwidth}
%\centering
%\rotatebox[origin=cc]{-0}{\includegraphics[width=1\textwidth]{CPUtime}}
%\makeatletter\def\@captype{figure}\makeatother
%\caption{\footnotesize Comparison the computational times between the direct and fast convolution algorithm.}
%\label{fig1}
%\end{minipage}
%\end{minipage}

Secondly, we compare in Figure \ref{fig1} the computational time in seconds against $N_t$ between the direct and fast convolution algorithms  with $\alpha=\beta=0.5$, $N=50, N_{\rm col}=30$ and  with  different  $\Delta t$ at $T=1$. Note that the fast convolution algorithm requires $O(N_t \log N_t)$ operations over $N_t$ time steps, while the direct algorithm requires $O(N_t^2)$ operations. As such, much saving can be achieved by using the fast convolution algorithm which is therefore necessary for long time simulation.

Thirdly, we depict in Figures \ref{fig2}-\ref{fig3}  the convergence rates in both time and space   with different parameters  $\alpha,\beta$. As expected, we observe from Figures \ref{fig2}  the first-order convergence order in time, while from Figures \ref{fig3} the spectral accuracy in space (given the spatial smooth exact solution). Here, we understand $O(\Delta t)=0.5\Delta t.$ For the latter, we choose $\Delta t=0.00001$ so that we can demonstrate the spatial errors.  Indeed, the numerics confirm the convergence $O(\Delta t+e^{-cN})$ for some $c>0$.
%%%%%%%%%%%%%%%%%%%%%%%%%%%%%%%%%%%%%%%%%%%%%%%%%%%%%%%%%%%%%%%%%%%%%%%%%%%%%%%%%%%%%
%%%%%%%%%%%%%%%%%%%%%%%%%%%%%%%%% Figure 2 %%%%%%%%%%%%%%%%%%%%%%%%%%%%%%%%%%%%
%%%%%%%%%%%%%%%%%%%%%%%%%%%%%%%%%%%%%%%%%%%%%%%%%%%%%%%%%%%%%%%%%%%%%%%%%%%%%%%%%%%%%
\begin{figure}[!th]
\subfigure[$\alpha=0.3, \beta=0.3$]{
\begin{minipage}[t]{0.42\textwidth}
\centering
\rotatebox[origin=cc]{-0}{\includegraphics[width=1\textwidth]{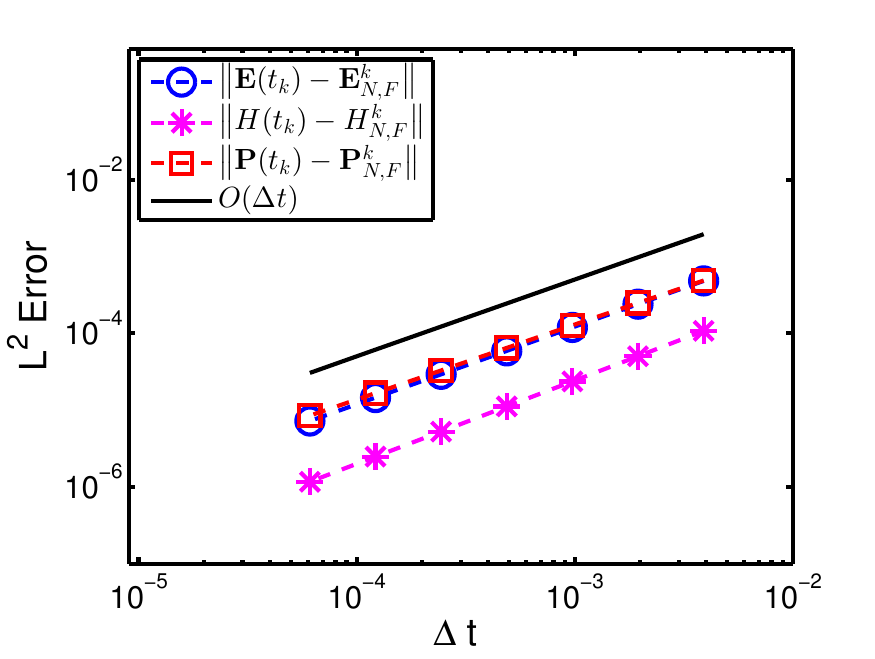}}
\end{minipage}} \quad
\subfigure[$\alpha=0.3, \beta=0.7$]{
\begin{minipage}[t]{0.42\textwidth}
\centering
\rotatebox[origin=cc]{-0}{\includegraphics[width=1\textwidth]{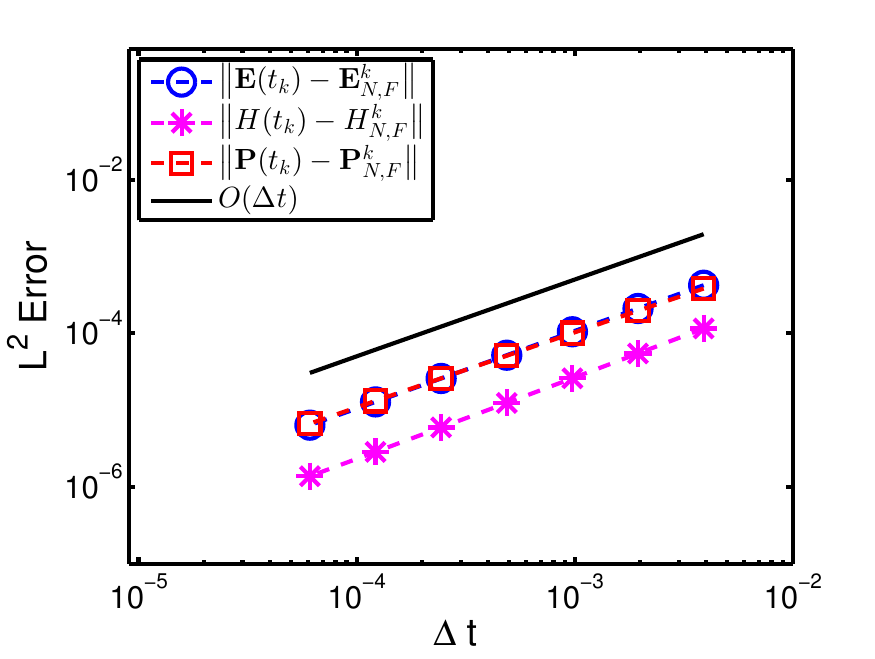}}
\end{minipage}}
\subfigure[$\alpha=0.7, \beta=0.3$]{
\begin{minipage}[t]{0.42\textwidth}
\centering
\rotatebox[origin=cc]{-0}{\includegraphics[width=1\textwidth]{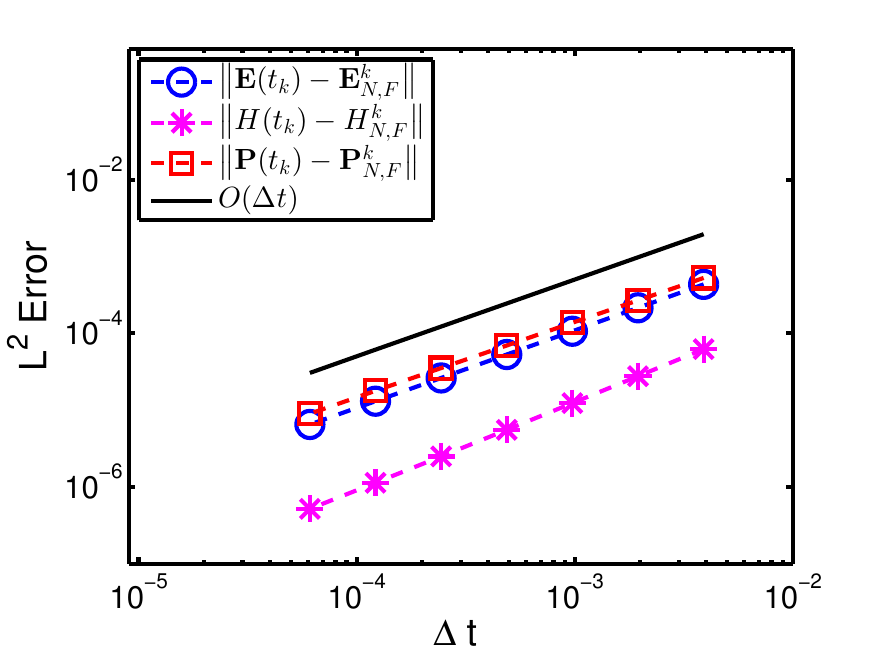}}
\end{minipage}}\quad
\subfigure[$\alpha=0.7, \beta=0.7$]{
\begin{minipage}[t]{0.42\textwidth}
\centering
\rotatebox[origin=cc]{-0}{\includegraphics[width=1\textwidth]{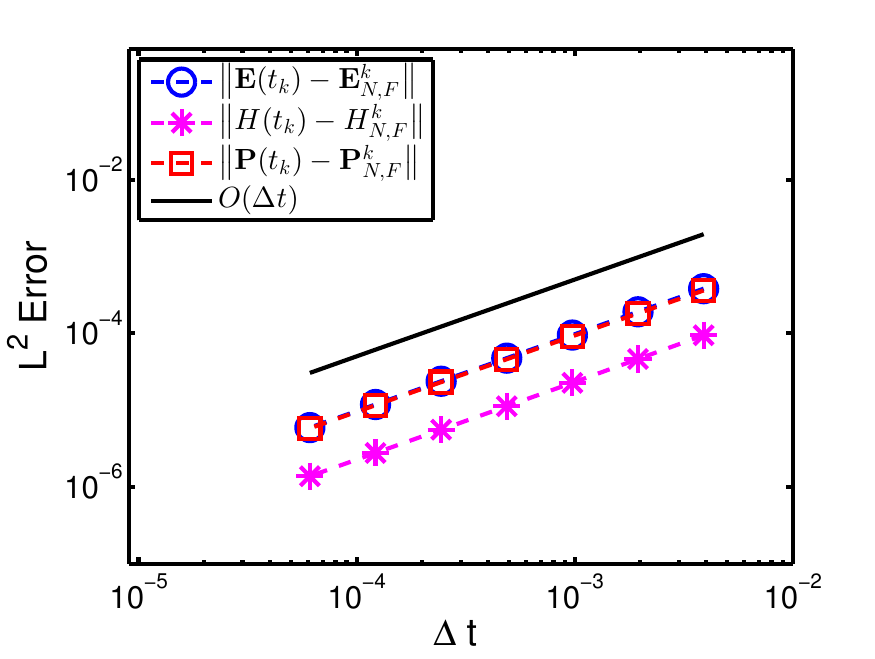}}
\end{minipage}}
\caption
{\small  Convergence order in time with $N=50$, $N_{\rm col}=30$ and different parameters $\alpha, \beta$ in log-log scale.}\label{fig2}
\end{figure}
%%%%%%%%%%%%%%%%%%%%%%%%%%%%%%%%%%%%%%%%%%%%%%%%%%%%%%%%%%%%%%%%%%%%%%%%%%%%%%%%%%%%
%%%%%%%%%%%%%%%%%%%%%%%%%%%%%%%% Figure 3 %%% %%%%%%%%%%%%%%%%%%%%%%%%%%%%%%%%%%%%%%
%%%%%%%%%%%%%%%%%%%%%%%%%%%%%%%%%%%%%%%%%%%%%%%%%%%%%%%%%%%%%%%%%%%%%%%%%%%%%%%%%%%%
\begin{figure}[!th]
\subfigure[$\alpha=0.3, \beta=0.6$]{
\begin{minipage}[t]{0.42\textwidth}
\centering
\rotatebox[origin=cc]{-0}{\includegraphics[width=1\textwidth]{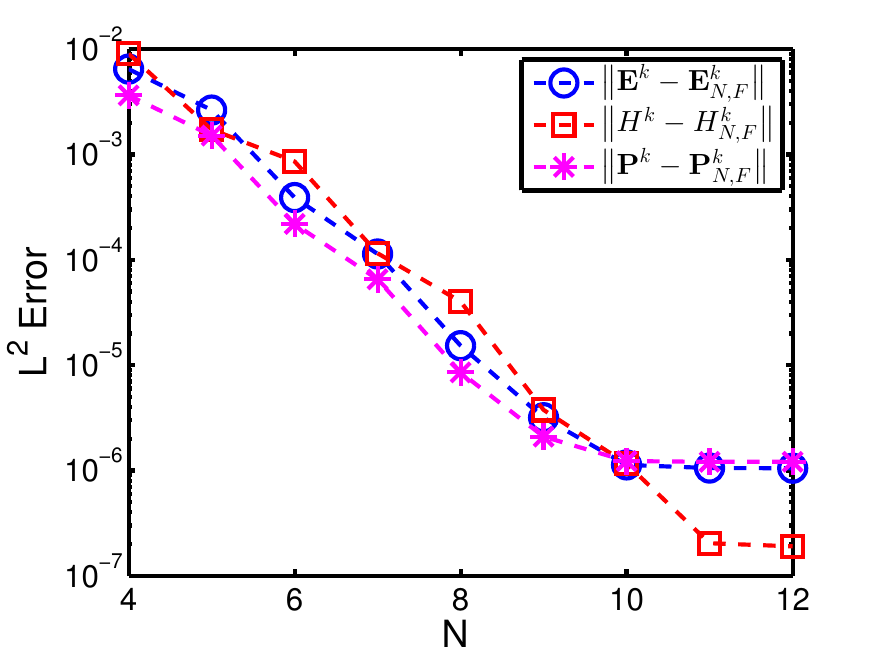}}
\end{minipage}}\quad
\subfigure[$\alpha=0.6, \beta=0.3$]{
\begin{minipage}[t]{0.42\textwidth}
\centering
\rotatebox[origin=cc]{-0}{\includegraphics[width=1\textwidth]{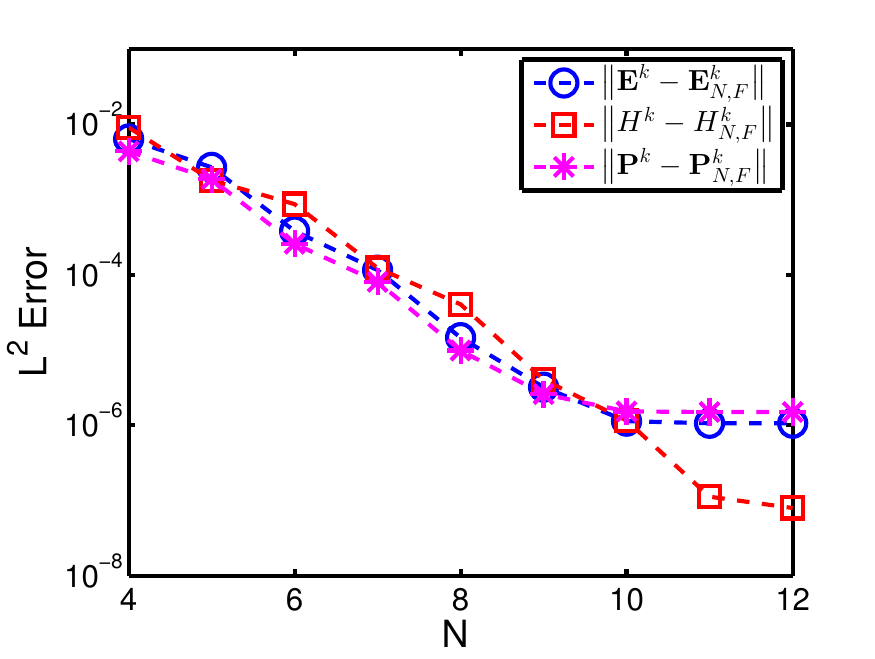}}
\end{minipage}}
\caption{\small Convergence behaviour  in space  for %with $\Delta t=0.00001$, $N_{\rm col}=30$ and
 different parameters $\alpha, \beta$ in semi-log scale.}
\label{fig3}
\end{figure}

Finally, we tabulate the discrete $L^{\infty}$-norm of the divergence of electric and polarisation fields with $\Delta t =0.001$, $N=50$ and different $\af, \bt$ at different times $t_k=k{\Delta t}$ in Table \ref{Tab2}, which  shows the scheme can preserve this property well.
\begin{table}[!th]
\centering\small 
\caption{Discrete $L^{\infty}$-norm of the divergence of electric and polarisation fields.}\label{Tab2}
\vspace*{-6pt}
\begin{tabular}{|c|c|c|c|c|c|c|}
\hline
\multirow{2}*{$k$} & \multicolumn{2}{c|}{$\af=0.3,\;\; \bt=0.7$} &  \multicolumn{2}{c|}{$\af=0.5,\;\; \bt=0.5$} &  \multicolumn{2}{c|}{$\af=0.7,\;\; \bt=0.3$} \\
\cline{2-7}
           & $\|\nabla\cdot {\bm E}^k\|_{\infty}$ & $\|\nabla\cdot {\bm P}^k\|_{\infty}$ & $\|\nabla\cdot {\bm E}^k\|_{\infty}$ & $\|\nabla\cdot {\bm P}^k\|_{\infty}$ & $\|\nabla\cdot {\bm E}^k\|_{\infty}$ & $\|\nabla\cdot {\bm P}^k\|_{\infty}$ \\
\hline\hline
%100   & 6.7975e-18 & 2.4536e-18 & 6.3760e-18 & 2.2245e-18 & 6.7918e-18 & 2.9125e-18 \\
%\hline
200   & 1.3034e-16 & 5.3398e-17 & 1.3084e-16 & 5.3862e-17 & 1.2869e-16 & 6.4280e-17 \\
\hline
400   & 2.2010e-15 & 1.0001e-15 & 2.2415e-15 & 1.0719e-15 & 2.2190e-15 & 1.2638e-15 \\
\hline
600   & 1.1832e-14 & 5.6432e-15 & 1.1797e-14 & 6.1230e-15 & 1.1630e-14 & 7.1018e-15 \\
\hline
800   & 3.9361e-14 & 1.9593e-14 & 3.9541e-14 & 2.1420e-14 & 3.8714e-14 & 2.4878e-14 \\
\hline
1000  & 1.0110e-13 & 5.1943e-14 & 1.0217e-13 & 5.7495e-14 & 9.9812e-14 & 6.5602e-14 \\
\hline
\end{tabular}
\end{table}

\subsubsection{Discrete energy decay} \label{subsubsect432}
In order to illustrate the discrete energy dissipation shown in Theorem \ref{thmFullyDisStab},
%\begin{example}\label{exmp2}
%Let $\Omega=(-1,1)\times(-1,1)$, $\epsilon_{\infty}=1$, $\Delta \epsilon=1$ in \eqref{hnte}.
we set the initial values to be
$$E_x(x,y,0)=\;\frac{1}{\sqrt 2}\cos(\pi x)\sin(\pi y), \;\;
E_y(x,y,0)=-\frac{1}{\sqrt 2}\sin(\pi x)\cos(\pi y), \;\;  H(x,y,0)=0.$$
Note that  the system  must be homogeneous to possess such a property (see  \eqref{hntebelg}).  As a result,  we use  sufficiently
fine mesh grids to verify  the accuracy and convergence order as observed previously. Here, we record in Figure \ref{fig4} the evolution of the discrete energy $\mathscr E_{N}^k$ obtained by the scheme with $\Delta t=0.01$, $N=50$ and $N_{\rm col}=30$ for some different parameters $\alpha, \beta$. Indeed, these numerical evidences validate  this behaviour.
Interestingly, when it comes to the discrete analogue of the energy in Theorem \ref{thmConStab}: $\tilde {\mathscr E}_{N}^k:=\epsilon_{\infty}\|{\bm E}_N^{k}\|^2+ \|{H}_N^{k}\|^2,$ we observe from Figure \ref{fig5} that it fails to satisfy  this decaying property. Indeed, as shown  in Theorem \ref{thmConStab},  this energy  at continuous level can only  be controlled by the initial energy.  In fact, a similar behaviour has been observed for the Cole-Cole model in \cite{huang2018accurate}.

%   with different parameters $\alpha, \beta$ is presented in . Now, the decay of this discrete energy can not be observed, but can be controlled by the initial energy.

\begin{figure}[!th]
\subfigure{
\begin{minipage}[t]{0.42\textwidth}
\centering
\rotatebox[origin=cc]{-0}{\includegraphics[width=1\textwidth]{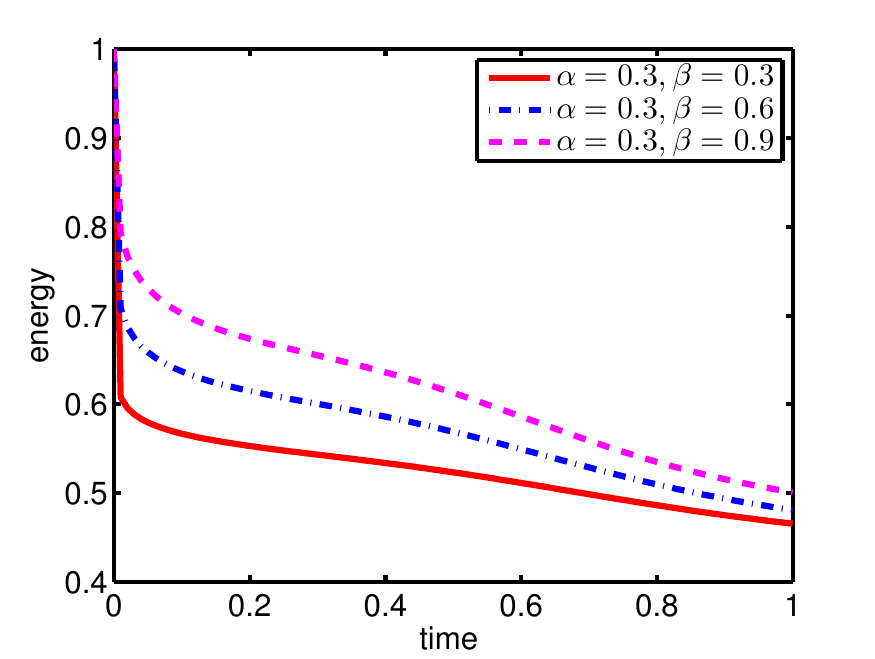}}
\end{minipage}}\quad
\subfigure{
\begin{minipage}[t]{0.42\textwidth}
\centering
\rotatebox[origin=cc]{-0}{\includegraphics[width=1\textwidth]{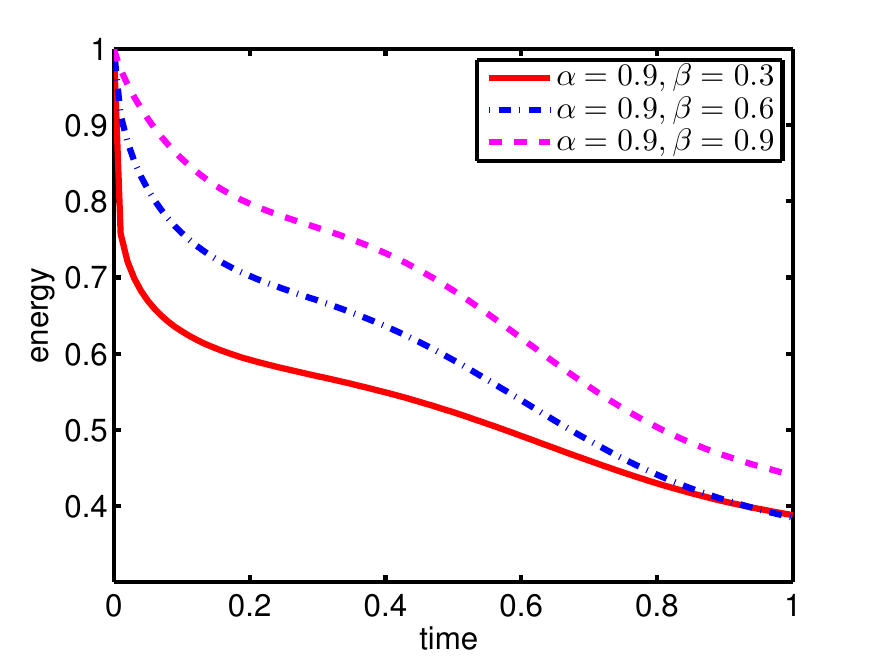}}
\end{minipage}}\quad
\subfigure{
\begin{minipage}[t]{0.42\textwidth}
\centering
\rotatebox[origin=cc]{-0}{\includegraphics[width=1\textwidth]{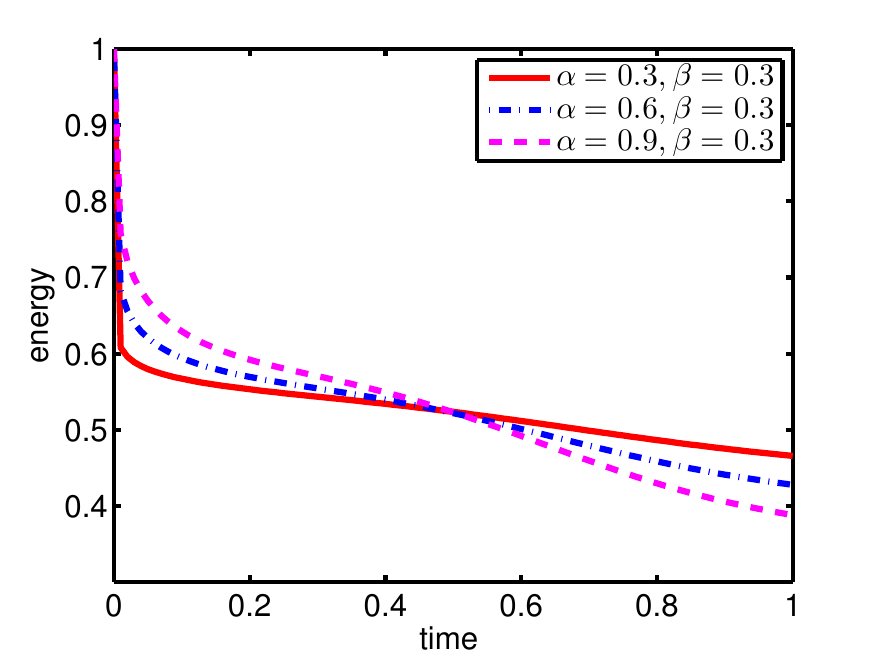}}
\end{minipage}}\quad
\subfigure{
\begin{minipage}[t]{0.42\textwidth}
\centering
\rotatebox[origin=cc]{-0}{\includegraphics[width=1\textwidth]{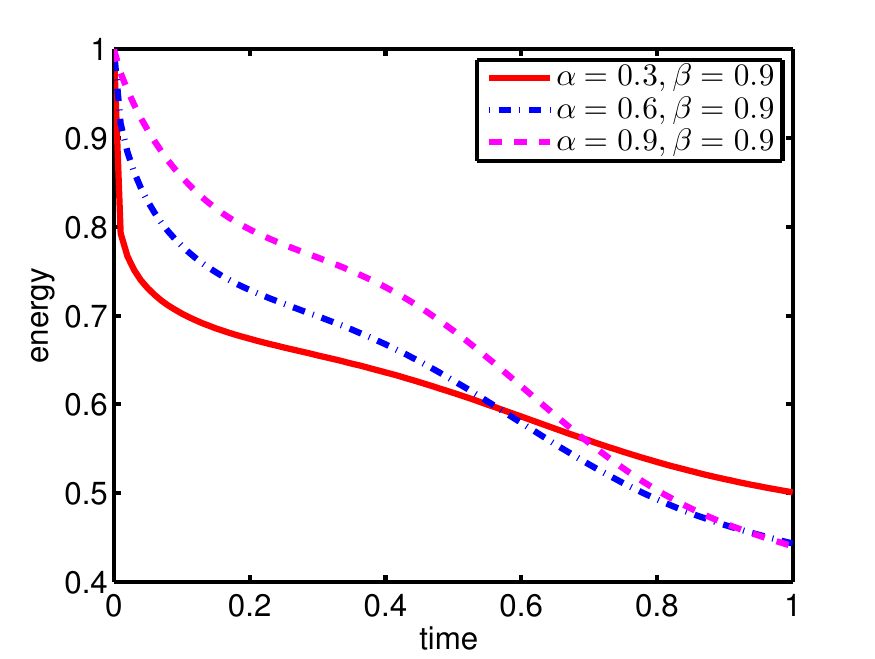}}
\end{minipage}}
\caption{\small Evolution of the discrete energy $\mathscr E_{N}^k$  with different $\alpha$ and $\beta$.}
\label{fig4}
\end{figure}

%%%%%%%%%%%%%%%%%%%%%%%%%%%%%%%%%%%%%%%%%%%%%%%%%%%%%%%%%%%%%%%%%%%%%%%%%%%%%%%%%%%%%
%%%%%%%%%%%%%%%%%%%%%%%%%%%%%%%%% Figure 4 %%%%%%%%%%%%%%%%%%%%%%%%%%%%%%%%%%%%
%%%%%%%%%%%%%%%%%%%%%%%%%%%%%%%%%%%%%%%%%%%%%%%%%%%%%%%%%%%%%%%%%%%%%%%%%%%%%%%%%%%%%

%%%%%%%%%%%%%%%%%%%%%%%%%%%%%%%%%%%%%%%%%%%%%%%%%%%%%%%%%%%%%%%%%%%%%%%%%%%%%%%%%%%%%
%%%%%%%%%%%%%%%%%%%%%%%%%%%%%%%%% Figure 5 %%%%%%%%%%%%%%%%%%%%%%%%%%%%%%%%%%%%
%%%%%%%%%%%%%%%%%%%%%%%%%%%%%%%%%%%%%%%%%%%%%%%%%%%%%%%%%%%%%%%%%%%%%%%%%%%%%%%%%%%%%
\begin{figure}[!th]
\subfigure{
\begin{minipage}[t]{0.42\textwidth}
\centering
\rotatebox[origin=cc]{-0}{\includegraphics[width=1\textwidth]{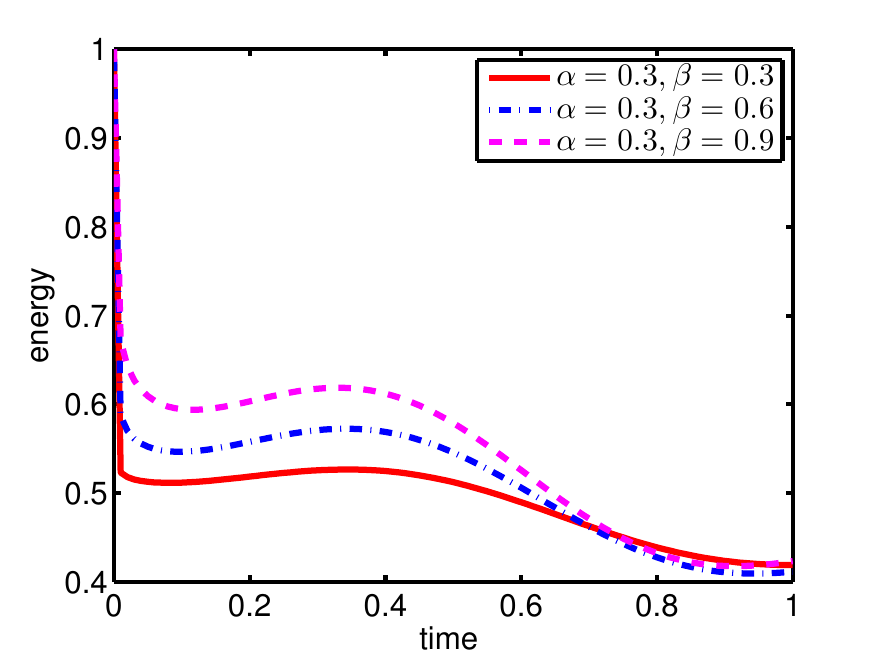}}
\end{minipage}}\quad
\subfigure{
\begin{minipage}[t]{0.42\textwidth}
\centering
\rotatebox[origin=cc]{-0}{\includegraphics[width=1\textwidth]{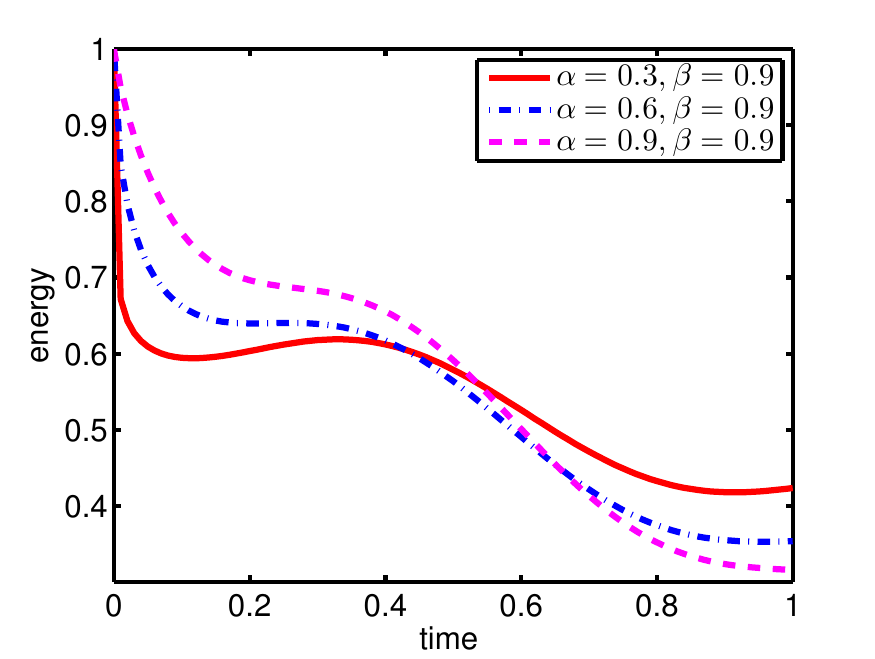}}
\end{minipage}}
\caption{\small Evolution of the discrete energy $\tilde {\mathscr E}_{N}^k$ with different $\alpha$ and $\beta$.}
\label{fig5}
\end{figure}

\subsection{Application: recovery of the relative permittivity, reflection coefficient and transfer function.} \label{subsect44}
%\begin{example}\label{exmp3}
As already mentioned in the introductory section,  the dispersive media in which the electromagnetic waves propagate, can be characterised by the relative permittivity:
\begin{equation}\label{hnperm2}
\epsilon_r(\omega)=\epsilon_{\infty}+\frac{\epsilon_s-\epsilon_{\infty}}{\left(1+(i\omega\tau_0)^{\alpha}\right)^{\beta}},
\end{equation}
in terms of the frequency variable $\omega,$ for  given $\epsilon_{\infty}$, $\epsilon_s$, $\tau_0$, $\alpha$ and $\beta.$
%are given, we can get the analytical of this H-N medium in frequency domain from \eqref{hnperm}.
It is  of physical interest to study  the associated  reflection coefficient (cf. \cite{antonopoulos2017fdtd}) in magnitude:
\begin{equation}\label{ReflCoeff}
|\mathcal{R}(\omega)|=\big|\big(1-\sqrt{\epsilon_r(\omega)}\big)/\big(1+\sqrt{\epsilon_r(\omega)}\big)\big|.
\end{equation}
%where $\epsilon_r$ is defined in \eqref{hnperm}.
Another closely related notion is the transfer function $T(d,\omega)$ (see, e.g., \cite{Rekanos2010Auxiliary,rekanos2012fdtd,Rekanos2012b,antonopoulos2017fdtd}) given  by
\begin{equation}\label{TraFunc}
T(d,\omega)= e^{\Upsilon(\omega)d},\quad \Upsilon(\omega)
=-i\,\omega\sqrt{\epsilon_r(\omega)}/c_0:=\Upsilon_R(\omega)+i\,\Upsilon_I(\omega),
\end{equation}
where $c_0=3.0\times10^8$ is the speed of light in free space.  It describes the transfer rate  of
 the electric field in  frequency domain  from the point ${\bm x}$ to the point ${\bm x}+d$:
\begin{equation}\label{RelTraFunc}
{\widehat {\bm E}}({\bm x}+d,\omega)= T(d,\omega)\, {\widehat {\bm E}}({\bm x},\omega),
\end{equation}
where ${\widehat {\bm E}}({\bm x},\omega)$ denotes the Fourier transform of the electric field ${\bm {E}}({\bm x},t)$.
% in time domain, and $T(d,\omega)$ is related to the complex relative permittivity by
%where $\epsilon_r$ is defined in \eqref{hnperm} and .

In view of the above relations, one  can compute ${\bm {E}}({\bm x},t)$  in time domain with fixed $\epsilon_{\infty}$, $\epsilon_s$, $\tau_0$, $\alpha$ and $\beta$ by solving the Maxwell's system \eqref{hnte}-\eqref{hnteiOri}, and then transform the field   to the frequency domain.  From \eqref{RelTraFunc}, we can compute the approximate transfer function $\tilde T(d,\omega)$ in $\omega$ (as the field $\bm E$ is computed numerically),  from which
we can work out the approximate $\tilde \epsilon_r(\omega)$ and $|\tilde {\mathcal{R}}(\omega)|$ by using the relations \eqref{TraFunc} and \eqref{ReflCoeff}, respectively.  We are interested in fitting and recovering the analytic values of $\epsilon_r(\omega),
|{\mathcal{R}}(\omega)|$ and  $T(d,\omega)$ (evaluated exactly by \eqref{hnperm2}-\eqref{TraFunc} with given $\epsilon_{\infty}$, $\epsilon_s$, $\tau_0$, $\alpha$ and $\beta$) by the corresponding approximate values as in \cite{Rekanos2010Auxiliary,rekanos2012fdtd,Rekanos2012b,antonopoulos2017fdtd}.

Similar to the setting in \cite{antonopoulos2017fdtd}, we consider  the Maxwell's system \eqref{hnte} in one spatial dimension with $z\in (a,b)$ and $t\in (0,T],$ but adding the source term $f(z,t):={E}_{\rm inc}(t)\chi_{z_*}(z)$ to the first equation of \eqref{hnte}.  Here,   ${E}_{\rm inc}(t)$ is  a modulated Gaussian pulse (cf.  \cite{antonopoulos2017fdtd}):
%A plane wave propagates along the $z$ direction with the electric field polarized in the $x$ axis, and excites the H-N medium. The incident wave, a modulated Gaussian pulse, is given by
\begin{equation}
{E}_{\rm inc}(t)= e^{-a_e^2(t-4/a_e)^2}\sin(2{\pi}f_e(t-4/a_e))u(t),
\label{IncWave}
\end{equation}
where $a_e = 5 \times 10^{9}$ ${\rm s}^{-1}$, the central frequency $f_e = 6$ {\rm GHz}, and $u(t)$ is the unit step function, i.e., $u(t)=1$ when $t\ge 0$ while $u(t)=0$ when $t < 0$. Note that the energy of the pulse ranges from $0.1$ {\rm GHz} to $10$ {\rm GHz}. In the source term, $z_*\in (a,b)$ is the location where the pulse is excited, and  $\chi_{z_*}(z)=1$ at $z=z_*,$ but
it is equal to $0$ elsewhere on $(a,b).$
%with $(z,t) \in (a,b) \times (0,T]$, where $z$ is the spatial variable, and ${\mathbbm 1}_{\{z=z_*\}}=1$ when $z=z_*$, otherwise ${\mathbbm 1}_{\{z=z_*\}}=0$ for $z_* \in (a,b)$.
It is noteworthy that  the vector fields in the system \eqref{hnte} reduce to the scalar fields ${E}_x(z,t), {H}_y(z,t)$ and ${P}_x(z,t)$ in one dimension.

\medskip
For clarity, we sketch the algorithm as follows.
\smallskip
\begin{itemize}
  \item[\bf (i)]  Solve the Maxwell's system for given  $\epsilon_{\infty}$, $\epsilon_s$, $\tau_0$, $\alpha$ and $\beta$.
  Here, we adopt  the finite-difference time-domain (FDTD) method to discretise the one dimensional system \cite{antonopoulos2017fdtd}, but use the fractional integral formulation of  the polarisation relation together with the aforementioned temporal discretisation, and  fast convolution algorithm.  With these, we can obtain the numerical approximation
  $E_{x,m}^k$ of $E_x(z,t)$ on the space-time grids: $t_k=k\Delta t$ and $z_m=a+m\Delta z.$
  %   illustrated in the above paragraph. Note that we approximate the polarisation field by the direct or fast temporal algorithm developed in section \ref{Sect3} and subsection \ref{FastC}, but in \cite{antonopoulos2017fdtd} a sum of C-C terms was used to approximate the polarisation field. We partition the domain $[a,b] \times [0,T]$ as
%       $$ t_k=k{{\Delta t}},\quad k=0,1,\cdots,N_t,\quad \Delta t=[T/N_t],$$
%       $$ z_m=m{{\Delta z}},\quad m=0,1,\cdots,N_z,\quad \Delta z=[(b-a)/N_z],$$
%      where ${\Delta t}$ and ${\Delta z}$ be the time step size and the space step size, respectively. Consequently, we obtain the numerical fields $E_x^k(m)$, $H_y^k(m+1/2)$ and $P_x^k(m)$ with the PEC boundary conditions and homogeneous initial conditions, where $U^k(m)$ stands for the approximation of $U(z_m,t_k)$ and $U$ can be $E_x$, $H_y$ or $P_x$. As usual (cf. \cite{Elsherbeni2015,taflove2005computational}), we approximate  $E_x$ and $E_y$ on $\{z_m\}_{m=0}^{N_z}$, while approximate $H_y$ on $\{z_{m+1/2}\}_{m=0}^{N_z-1}$, at time $\{t^k\}_{k=1}^{N_t}$.

      \smallskip
\item[\bf (ii)]   Apply the discrete Fourier transform  (cf. \cite[P. 156]{Elsherbeni2015})
to   $\big\{E_{x, m_*}^k\big\}_{k=1}^{N_t}$ and $\big\{E_{x, m_*+l}^k\big\}_{k=1}^{N_t}$ (at the locations $z_*=z_{m_*}$ and
$z=z_*+d$ with $d=l\Delta z$) from the  time domain to the  frequency domain that leads to
 $\big\{{\widehat {E}}_{x, m_*}^{\,\omega_j}\big\}_{j=1}^{N_{\omega}}$ and $\big\{{\widehat {E}}_{x, m_*+l}^{\,\omega_j}\big\}_{j=1}^{N_{\omega}},$  Then the approximate transfer function is  %$ {T}(d,\omega)$ by \eqref{RelTraFunc}, i.e.,
    \begin{equation} \label{AppTranF}
    {\tilde T}(d,\omega_j)=  {\widehat {E}}_{x, m_*+l}^{\,\omega_j}\big/{\widehat {E}}_{x, m_*}^{\,\omega_j}.
    \end{equation}
    %\smallskip
\item[\bf (iii)] Substitute \eqref{AppTranF} into \eqref{TraFunc} leading to the approximation:
%, we firstly get the numerical approximations of the real and the imaginary part of $\Upsilon(\omega)$ by direction calculation
    \begin{equation}\label{RealFuncChi}
    {\tilde \Upsilon}_R(\omega_j) = \ln \big(\big| {\widehat {E}}_{x, m_*+l}^{\,\omega_j}\big/{\widehat {E}}_{x, m_*}^{\,\omega_j}\big|\big)/d,\;\;  {\tilde \Upsilon}_I(\omega_j) = \big(\arg\big\{{\widehat {E}}_{x, m_*+l}^{\,\omega_j}\big\}-\arg\big\{{\widehat {E}}_{x, m_*}^{\,\omega_j}\big\}\big)/d.
    \end{equation}
%    and
%    \begin{equation}\label{ImagFuncChi}
%    {\tilde \Upsilon}_I(\omega_j) = \big(\arg\big\{{\widehat {E}}_{x, m_*+l}^{\,\omega_j}\big\}-\arg\big\{{\widehat {E}}_{x, m_*}^{\,\omega_j}\big\}\big)/d,
%    \end{equation}
%    respectively. Then, by substituting \eqref{RealFuncChi} and \eqref{ImagFuncChi} into the second equation in \eqref{TraFunc}, we obtain the numerical approximations of
Accordingly,   we derive from \eqref{TraFunc} and \eqref{RealFuncChi} the real and the imaginary part of the approximate relative permittivity:
    \begin{equation*}%\label{AppPerm}
    {\tilde \epsilon}_r(\omega_j) = {\tilde \epsilon}'(\omega_j) - i{\tilde \epsilon}''(\omega_j) = -\big(c_0({\tilde \Upsilon}_R(\omega_j)+i{\tilde \Upsilon}_I(\omega_j))/\omega_j \big)^2,
    \end{equation*}
   and from  \eqref{ReflCoeff} the approximate magnitude of the reflection coefficient:
   %can be obtained by the numerical relative permittivity through \eqref{ReflCoeff}, i.e.,
    \begin{equation*}%\label{AppReflCoeff}
    |\mathcal{\tilde R}(\omega_j)|=\Big|\big(1-\sqrt{{\tilde \epsilon}_r(\omega_j)}\big)\Big/\big(1+\sqrt{{\tilde \epsilon}_r(\omega_j)}\big)\Big|.
    \end{equation*}
\end{itemize}
%\medskip
In the computation,   we take $a=0, b=1.1$ {\rm m}, $T=5.304 \times 10^{-9}$ {\rm s}, $\epsilon_s=50$, $\epsilon_{\infty}=2$, $\tau_0=1.53 \times 10^{-10}$ {\rm s}, $\Delta t=1.768 \times 10^{-12}$ {\rm s}, $\Delta z=1.1$ {\rm mm}, $z_*=0.55$ {\rm m}, and
sample $\big\{\omega_j\big\}_{j=1}^{N_{\omega}}$ in $(0.1, 10)$ {\rm GHz} as in  \cite{antonopoulos2017fdtd}.
%In the following numerical experiments, we don't clarify the numerical results derived by the direct and fast temporal algorithms because they are almost identical as validated in subsection \ref{subsubsect431}. Like inAnd we set $z_*=0.55$ {\rm m} as in \cite{bia2015novel} which implies the choices in \cite{antonopoulos2017fdtd, Rekanos2010Auxiliary,rekanos2012fdtd,Rekanos2012b}.
In Figure \ref{fig6},  we plot the  analytical magnitude of the reflection coefficient $|\mathcal{R}(\omega)|$ and approximate values against samples of $\{\omega_j\}$  with different $\alpha, \beta$. In Figure \ref{fig7},  we show
the analytical complex transfer function $T(d,\omega_j)$ and its approximation $\tilde T(d,\omega_j)$  with different $\alpha$, $\beta$ and $d$. We observe a better approximation than that in \cite{antonopoulos2017fdtd}, which shows the advantage of our approach.  In  Figure \ref{fig8}, we depict the complex relative permittivity $\epsilon_r=\epsilon'-i\epsilon''$ and the approximate
 ${\tilde \epsilon}_r={\tilde \epsilon}'-i{\tilde \epsilon}''$ with different $\alpha$ and $\beta,$ which is not presented in \cite{antonopoulos2017fdtd}. Indeed,  we observe a good agreement of the exact and numerical values.

%From Figure \ref{fig6}, we observe that the numerical magnitude of the reflection coefficients are a bit good than in \cite{antonopoulos2017fdtd} and agree well with the analytical ones. From Figure \ref{fig7}, we observe that the numerical complex transfer functions are as good as in \cite{antonopoulos2017fdtd} and agree very well with the analytical ones. Moreover, the comparisons between the analytical complex relative permittivity $\epsilon_r=\epsilon'-i\epsilon''$ and the approximate one ${\tilde \epsilon}_r={\tilde \epsilon}'-i{\tilde \epsilon}''$ with different $\alpha$ and $\beta$, which is not presented in \cite{antonopoulos2017fdtd}, is presented in Figure \ref{fig8}. Again, we see that the numerical results agree well with the analytical ones. Thus, our direct and fast temporal algorithms developed in section \ref{Sect3} and subsection \ref{FastC} are reliable and accurate tools for electromagnetic analysis of wideband waves propagating in H-N media.
%%%%%%%%%%%%%%%%%%%%%%%%%%%%%%%%%%%%%%%%%%%%%%%%%%%%%%%%%%%%%%%%%%%%%%%%%%%%%%%%%%%%%
%%%%%%%%%%%%%%%%%%%%%%%%%%%%%%%%% Figure 6 %%%%%%%%%%%%%%%%%%%%%%%%%%%%%%%%%%%%
%%%%%%%%%%%%%%%%%%%%%%%%%%%%%%%%%%%%%%%%%%%%%%%%%%%%%%%%%%%%%%%%%%%%%%%%%%%%%%%%%%%%%
\begin{figure}[!th]
\subfigure[$\alpha=0.8, \beta=0.9$]{
\begin{minipage}[t]{0.42\textwidth}
\centering
\rotatebox[origin=cc]{-0}{\includegraphics[width=1\textwidth]{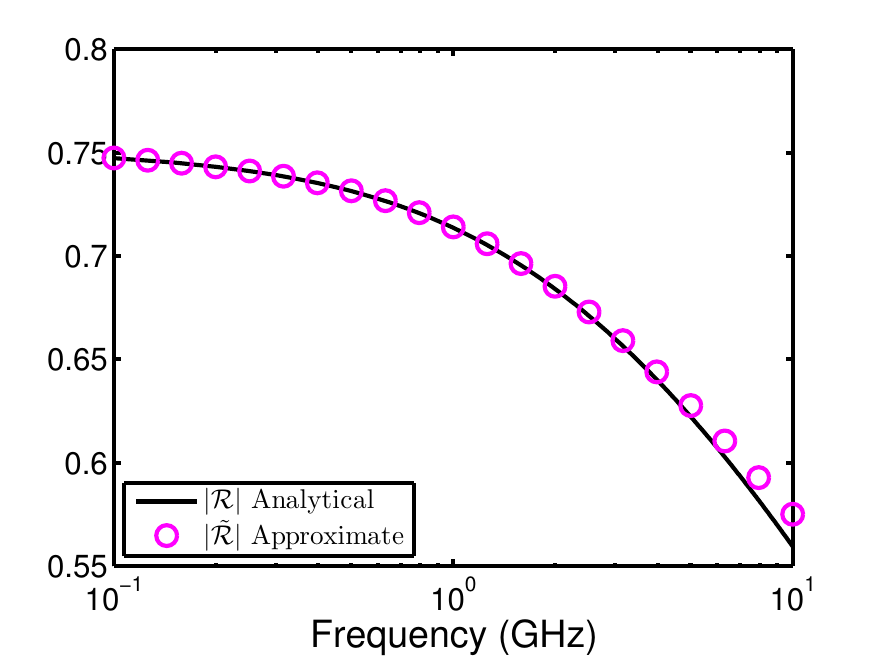}}
\end{minipage}}\quad
\subfigure[$\alpha=0.9, \beta=0.6$]{
\begin{minipage}[t]{0.42\textwidth}
\centering
\rotatebox[origin=cc]{-0}{\includegraphics[width=1\textwidth]{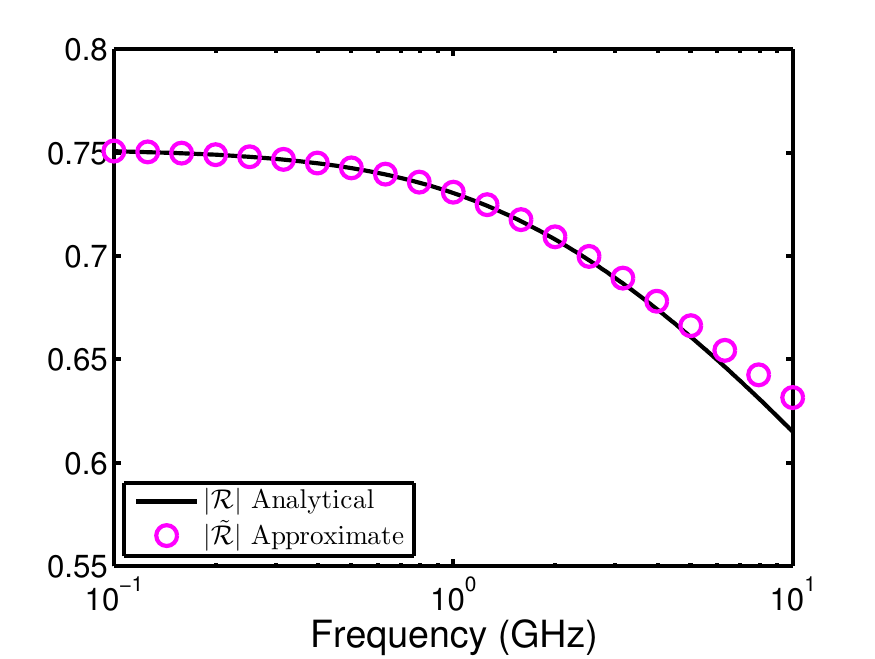}}
\end{minipage}}
\caption{\small Real and imaginary part of the analytical magnitude of the reflection coefficient and the approximate value.}
\label{fig6}
\end{figure}
%%%%%%%%%%%%%%%%%%%%%%%%%%%%%%%%%%%%%%%%%%%%%%%%%%%%%%%%%%%%%%%%%%%%%%%%%%%%%%%%%%%%%
%%%%%%%%%%%%%%%%%%%%%%%%%%%%%%%%% Figure 7 %%%%%%%%%%%%%%%%%%%%%%%%%%%%%%%%%%%%
%%%%%%%%%%%%%%%%%%%%%%%%%%%%%%%%%%%%%%%%%%%%%%%%%%%%%%%%%%%%%%%%%%%%%%%%%%%%%%%%%%%%%
\begin{figure}[!th]
\subfigure[$\alpha=0.8, \beta=0.9$]{
\begin{minipage}[t]{0.42\textwidth}
\centering
\rotatebox[origin=cc]{-0}{\includegraphics[width=1\textwidth]{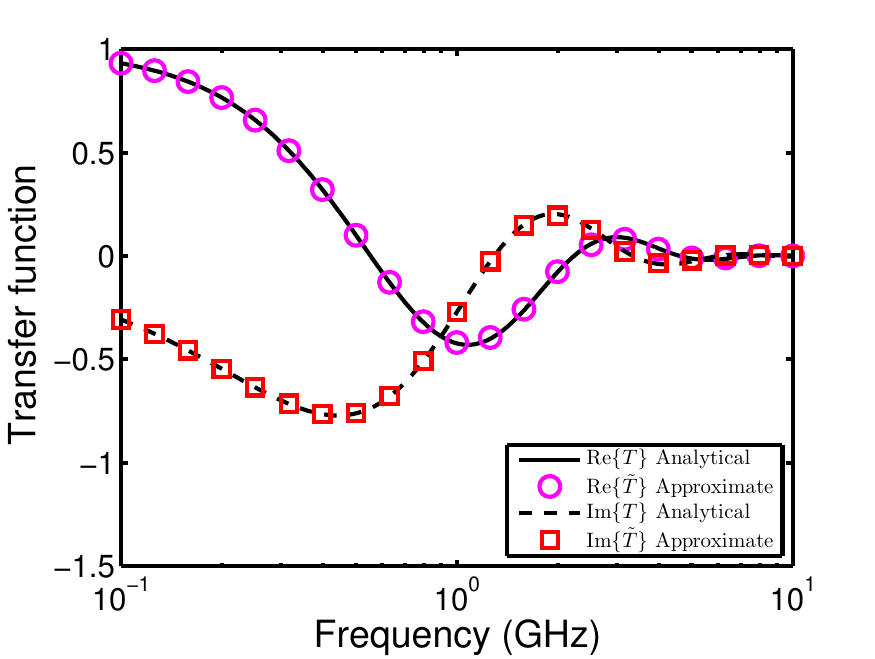}}
\end{minipage}}\quad
\subfigure[$\alpha=0.8, \beta=0.9$]{
\begin{minipage}[t]{0.42\textwidth}
\centering
\rotatebox[origin=cc]{-0}{\includegraphics[width=1\textwidth]{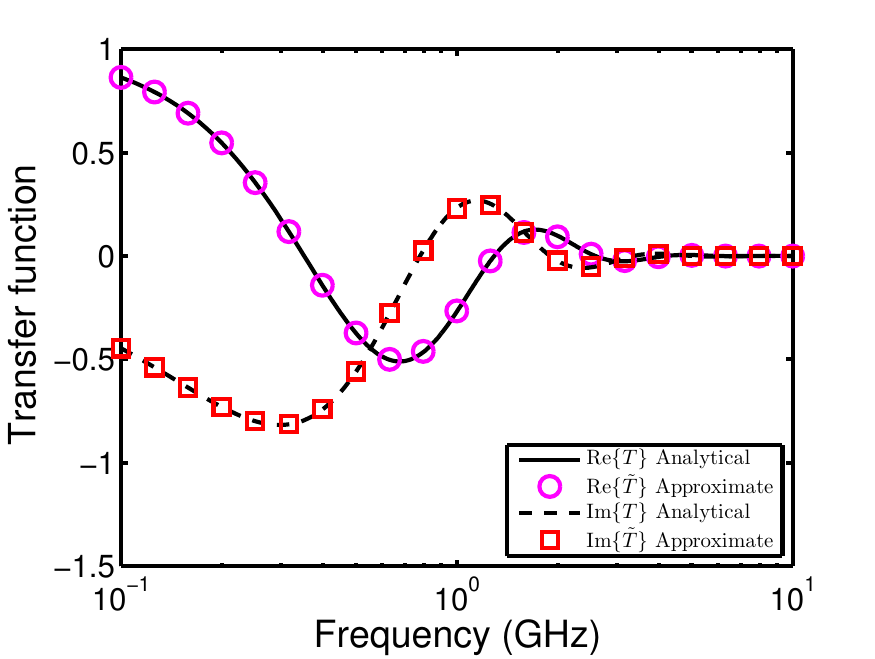}}
\end{minipage}}\quad
\subfigure[$\alpha=0.9, \beta=0.6$]{
\begin{minipage}[t]{0.42\textwidth}
\centering
\rotatebox[origin=cc]{-0}{\includegraphics[width=1\textwidth]{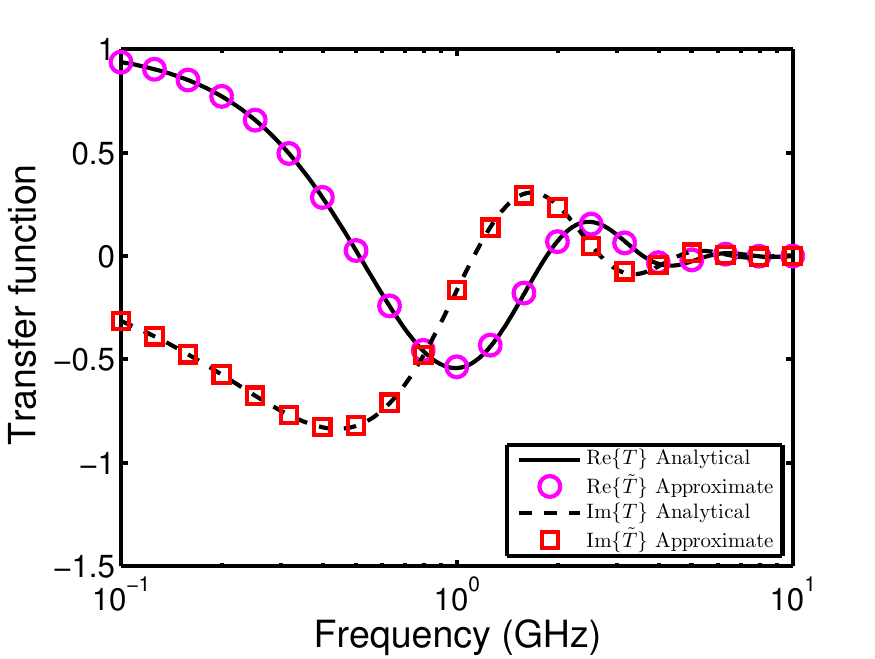}}
\end{minipage}}\quad
\subfigure[$\alpha=0.9, \beta=0.6$]{
\begin{minipage}[t]{0.42\textwidth}
\centering
\rotatebox[origin=cc]{-0}{\includegraphics[width=1\textwidth]{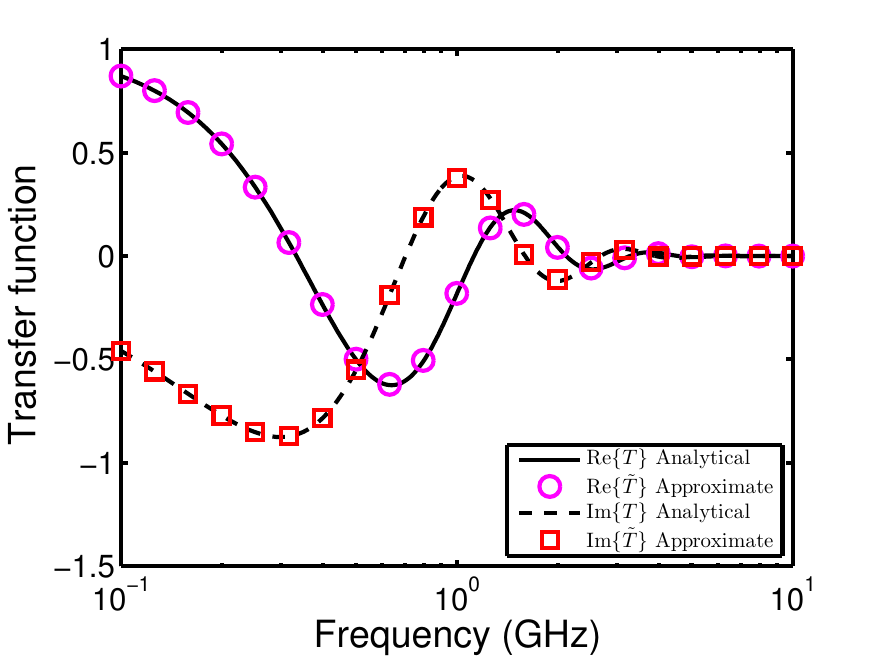}}
\end{minipage}}
\caption{\small Real and imaginary part of the analytical complex transfer function and the approximate one of the H-N medium. (Left): $d=20\Delta z$; (Right): $d=30\Delta z$.}
\label{fig7}
\end{figure}
%%%%%%%%%%%%%%%%%%%%%%%%%%%%%%%%%%%%%%%%%%%%%%%%%%%%%%%%%%%%%%%%%%%%%%%%%%%%%%%%%%%%%
%%%%%%%%%%%%%%%%%%%%%%%%%%%%%%%%% Figure 8 %%%%%%%%%%%%%%%%%%%%%%%%%%%%%%%%%%%%
%%%%%%%%%%%%%%%%%%%%%%%%%%%%%%%%%%%%%%%%%%%%%%%%%%%%%%%%%%%%%%%%%%%%%%%%%%%%%%%%%%%%%
\begin{figure}[!th]
\subfigure[$\alpha=0.8, \beta=0.9$]{
\begin{minipage}[t]{0.42\textwidth}
\centering
\rotatebox[origin=cc]{-0}{\includegraphics[width=1\textwidth]{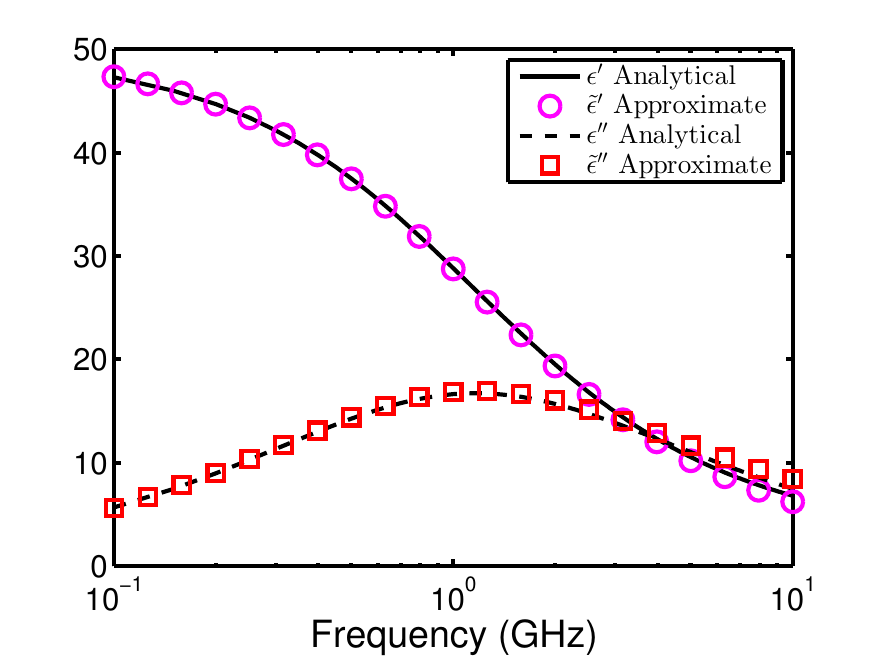}}
\end{minipage}}\quad
\subfigure[$\alpha=0.9, \beta=0.6$]{
\begin{minipage}[t]{0.42\textwidth}
\centering
\rotatebox[origin=cc]{-0}{\includegraphics[width=1\textwidth]{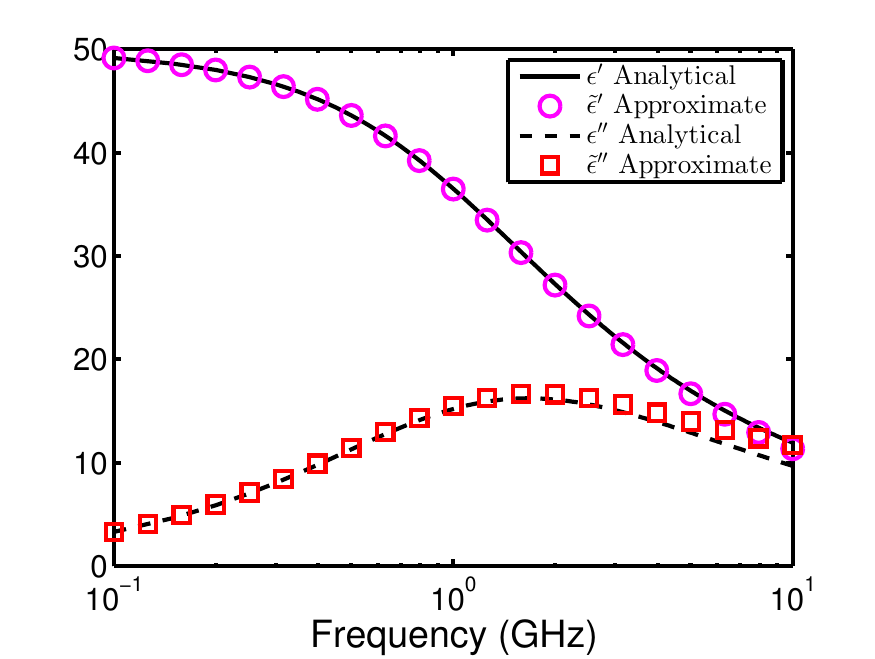}}
\end{minipage}}
\caption{\small Real and imaginary part of the analytical complex relative permittivity and the approximate one of the H-N media.}
\label{fig8}
\end{figure}

\section{Concluding remarks}\label{Sect5} 
In this paper, we rigorously showed the energy dissipation law and $L^2$-stability of  the Maxwell’s equations in a H-N dispersive medium, which were  unavailable in any literature.   We  proposed a   backward Euler-type  time discretisation based on the piecewise constant approximation of the induced fractional electric polarisation relation.  With a delicate study of the discrete weights resulted from the approximation of this relation,   we  proved the semi-discrete scheme satisfies an analogous (modified) energy dissipation law, so we  further showed the unconditional stability and  convergence of the semi-discretised scheme.  We then introduced a fast convolution algorithm so that the time-stepping algorithm can incorporate with   various spatial discretisation such as finite elements, finite differences and spectral elements. As an illustration, we  considered the full-discretisation scheme for  the two-dimensional model with the spatial discretisation by a spectral-Galerkin method, and conducted the error analysis. We provided ample numerical results to show the efficiency and robustness of the proposed method.  We also applied the solver to a real application in the recovery of  the relative permittivity, reflection coefficient and transfer function.

There are some related issues worthy of future investigation.  Here, we developed the first-order time discretisation scheme,  but the generalisation to higher order schemes appears far from trivial.  One challenge lies in how to  show the important property of the weights  similar to that in  Lemma \ref{lemdisker}.   In this work, we only implemented and analysed the spatial discretisation for two-dimensional problems, but  it is of more practical interest to consider the full three-dimensional model using e.g., the edge elements.

\begin{appendix}
\section{Properties of the Mittag-Leffler Function}\label{AppendixA}
\setcounter{equation}{0}
\renewcommand\theequation{A.\arabic{equation}}

We collect below some properties of the ML function that are used in this paper.
\smallskip
\begin{itemize}
\item  According to \cite[(5.1.6)]{gorenflo2014mittag},  the Laplace transform of the ML function in \eqref{gml}  takes the form
\begin{equation}\label{LTeML}
\mathscr{L}\left[t^{\mu-1}E_{\rho,\mu}^{\gamma}(\sigma t^\rho)\right](s)=\frac{s^{-\mu}}{(1-\sigma s^{-\rho})^{\gamma}}=\frac{s^{\rho\gamma-\mu}}{(s^{\rho}-\sigma)^{\gamma}},  \quad{\rm for}\;\, \rho,\mu>0 \;{\rm and \;real} \;\gamma.
\end{equation}
%for $\rho,\mu>0 $ and real $\gamma.$
%\smallskip
\item We have the integral identity (cf.\! \cite[(4)-(5)]{garrappa2017fractional} or \cite[(5.1.15), (5.1.19)]{gorenflo2014mittag}):
\begin{equation}\label{IntgML}
\int_{0}^{z} t^{\mu-1}E_{\rho,\mu}^{\gamma}(\sigma t^\rho)\,{\rm d}t = z^{\mu}E_{\rho,\mu+1}^{\gamma}(\sigma z^\rho),\quad {\rm for}\;\, \rho,\mu>0 \;{\rm and \;real} \;\gamma.
\end{equation}

\item We know from \cite[p.\!\! 9]{prabhakar1971singular} that for all $\rho >0$ and real $\gamma,\mu$, the Mittag-Leffler function with three parameters $E_{\rho,\mu}^{\gamma}(z)$ is bounded in a finite interval, i.e.,
\begin{equation}\label{BoundML}
|E_{\rho,\mu}^{\gamma}(z)|\le M, \quad \forall z \in [a,b],
\end{equation}
where $a, b$ and $M>0$ are constants.

\medskip
\item   For all $\sigma >0$, $t^{\mu-1}E_{\rho,\mu}^{\gamma}(-\sigma t^{\rho})$ is completely monotonic, if only if~~$0 < \rho < \mu \leq 1$ and $0 < \gamma \leq \mu/\rho$ (cf.\! \cite[(5.1.10)]{gorenflo2014mittag}),  that is,
\begin{equation}\label{CMML}
(-1)^n\frac{d^n}{dt^n}[t^{\mu-1}E_{\rho,\mu}^{\gamma}(-\sigma t^{\rho})]\geq 0,\quad \forall \;t \in (0,\infty),
\end{equation}
where we refer to \cite[Definition 3.22]{gorenflo2014mittag} for  the definition of the completely monotonicity function.
\end{itemize} 
\end{appendix}

%\section*{Acknowledgements}

\end{document}